\documentclass[]{article}

\usepackage{tablefootnote}
\usepackage{graphicx}
\usepackage{amsmath}
\usepackage{amssymb}
\usepackage{latexsym}
\usepackage{subfigure}
\usepackage{crop}
\usepackage{algorithm}
\usepackage{algcompatible}
\usepackage{multirow}
\usepackage{bm}
\usepackage{bbm}
\usepackage{enumerate}
\usepackage{framed} % or, "mdframed"
\usepackage[framed]{ntheorem}
\usepackage{url}
\usepackage{placeins}
\usepackage[colorlinks = true, pdfstartview = FitV, linkcolor = blue, citecolor = blue, urlcolor = blue]{hyperref}
\usepackage{array}
\usepackage{paralist}
\usepackage{mathtools}

\usepackage{fullpage}
\usepackage[sort,nocompress]{cite}
\usepackage{enumitem}
\setlist[enumerate,1]{leftmargin=*,wide=0em, noitemsep,nolistsep, label = {\bfseries \arabic*.}}
\setlist[itemize,1]{leftmargin=*,wide=0em, noitemsep,nolistsep}
\usepackage{titlesec}
\titleformat*{\section}{\large\bfseries}
\titleformat*{\subsection}{\large\bfseries}
\titleformat*{\subsubsection}{\large\bfseries}
\titleformat*{\paragraph}{\normalsize\bfseries}
\titleformat*{\subparagraph}{\normalsize\bfseries}
\DeclarePairedDelimiter{\ceil}{\lceil}{\rceil}
\newcommand {\uu}  { {\bf u} }
\newcommand {\zz}  { {\bf z} }
\newcommand {\bgg}  { {\bf g} }
\newcommand {\vg}  { {\bf g} }
\newcommand {\vs}  { {\bf s} }

\newcommand {\xx}  { {\bf x} }
\renewcommand {\aa}  { {\bf a} }

\newcommand {\yy}  { {\bf y} }

\newcommand {\pp}  { {\bf p} }

\newcommand {\vv}  { {\bf v} }
\newcommand {\ww}  { {\bf w} }

\newcommand {\bb}  { {\bf b} }

\newcommand{\hf}{\frac12}

\renewcommand{\vec}[1]{\ensuremath{\mathbf{#1}}}

\newcommand{\s}{\vec{s}}
\newcommand{\defeq}{\mathrel{\mathop:}=}
\newcommand{\defeqr}{=\mathrel{\mathop:}}

\renewcommand{\Pr}{\hbox{\bf{Pr}}}

\newcommand {\C} { {\mathcal{C}} }
\newcommand {\reals} { {\mathbb{R}} }
\newcommand{\prox}{\text{\textbf{prox}}}

\definecolor{forestgreen}{rgb}{0.13, 0.55, 0.13}

\newcounter{comment}

\theoremclass{Theorem}
\theoremstyle{break}
\newframedtheorem{theorem}{Theorem}
\newframedtheorem{corollary}{Corollary}
\newframedtheorem{lemma}{Lemma}
\newframedtheorem{definition}{Definition}
\newframedtheorem{proposition}{Proposition}
\newframedtheorem{assumption}{Assumption}
\newtheorem{example}{Example}

\newenvironment{proof}[1][Proof]{\begin{trivlist}
\item[\hskip \labelsep {\bfseries #1}]}{\end{trivlist}}
\newenvironment{remark}[1][Remark]{\begin{trivlist}
\item[\hskip \labelsep {\bfseries #1}]}{\end{trivlist}}

\newcommand{\qed}{\nobreak \ifvmode \relax \else
      \ifdim\lastskip<1.5em \hskip-\lastskip
      \hskip1.5em plus0em minus0.5em \fi \nobreak
      \vrule height0.75em width0.5em depth0.25em\fi}
\usepackage{tcolorbox} % for boxed text
\tcbuselibrary{breakable}
\tcbuselibrary{skins}
\usepackage{listings} % to inser code

\definecolor{mygreen}{rgb}{0,0.6,0}
\definecolor{mygray}{rgb}{0.5,0.5,0.5}
\definecolor{mymauve}{rgb}{0.58,0,0.82}

\newcommand*\lin[1]{\langle #1\rangle}
\usepackage{array}
\newcolumntype{P}[1]{>{\centering\arraybackslash}p{#1}}

\author{
	Xiang Cheng
	\thanks{
		Department of Computer Science, 
		UC Berkeley.
		Email: \tt x.cheng@berkeley.edu
	}
	\and
	Farbod Roosta-Khorasani
	\thanks{School of Mathematics and Physics,
	University of Queensland.
	Email: \tt fred.roosta@uq.edu.au 
	}
	\and
	Stefan Palombo
	\thanks{
		Department of Computer Science, 
		UC Berkeley.
		Email: \tt s.palombo@berkeley.edu
	}
	\and
	Peter L. Bartlett
	\thanks{
		Department of Statistics,
		UC Berkeley.
		Email: \tt bartlett@cs.berkeley.edu}	
	\and
	Michael W. Mahoney
	\thanks{
		ICSI and Department of Statistics, 
		UC Berkeley.
		Email: \tt mmahoney@stat.berkeley.edu
	}
}

\title{FLAG n' FLARE: Fast Linearly-Coupled Adaptive Gradient Methods}

\begin{document}

\maketitle

\begin{abstract}
	We consider first order gradient methods for effectively optimizing a composite objective in the form of a sum of smooth and, potentially, non-smooth functions. We present accelerated and adaptive gradient methods, called FLAG and FLARE, which can offer the best of both worlds. They can achieve the optimal convergence rate by attaining the optimal first-order oracle complexity for smooth convex optimization. Additionally, they can adaptively and non-uniformly re-scale the gradient direction to adapt to the limited curvature available and conform to the geometry of the domain. We show theoretically and empirically that, through the compounding effects of acceleration and adaptivity, FLAG and FLARE can be highly effective for many data fitting and machine learning applications. 
\end{abstract}

\section{Introduction}
\label{sec:introduction}

Optimization problems which exhibit particular structure appear often in many science, engineering, data analysis and machine learning applications~\cite{sra2012optimization, bubeck2015convex, bottou2016optimization}.
It is, by now, a well-known fact that taking proper advantage of the problem structure can lead to better performance guarantees and more effective algorithms compared to black-box, structure-oblivious methods; see~\cite[Section 4.1]{nesterov2004introductory} for a more detailed discussion and~\cite{nesterov94interior,nesterov2005smooth,nesterov2008rounding, parikh2014proximal, rodoas1,rodoas1,roszas, palomar2010convex, tibshirani1996regression, friedman2001elements, kulis2012metric} for many practical examples.

Here, we consider the optimization problem with the particular form
\begin{equation}
\label{eq:obj}
\min_{\xx \in \C} F(\xx) = f(\xx) + h(\xx),
\end{equation}
where $f: \reals^{d} \rightarrow \reals$ and $h: \reals^{d} \rightarrow \reals$ are, respectively, smooth and potentially non-smooth, closed proper convex functions and $\C$ is a closed convex set. Optimization problems of the form~\eqref{eq:obj} are often known as \emph{composite optimization} and arise in many applications. Notable examples are those in which $ h $ encapsulates an {\it a priori} assumption on the sought-after parameter $ \xx $, e.g., sparsity or low-rank structure.
Specific examples include the following.
\begin{example}
	The class of generalized linear models (GLMs)~\cite{mccullagh1989generalized} is used to model a wide variety of regression and classification problems. The process of data fitting using such GLMs usually consists of a training data set containing n response-covariate pairs, and the goal is to predict some
	output response based on some covariate vector, which is given after the training phase. Sparse maximum a posteriori (MAP) estimation using any GLM with canonical
	link function and Laplace prior leads to problems of the form~\eqref{eq:obj} 
	$$
	\min_{\xx \in \reals^{d}} \sum_{i=1}^{n} \left( \Phi(\aa_{i}^{T} \xx) - b_{i} \aa_{i}^{T} \xx \right) + \lambda\|\xx\|_{1},
	$$ 
	where $\{(\aa_{i},b_{i})\}_{i=1}^n$ form the response-covariate pairs in the training set (typically $\aa_{i} \in \mathbb{R}^{d}$, but the domain of $b_i$ depends on the type of GLM), and $\lambda \geq 0$ is the regularization parameter. The \textit{cumulant generating function}, $\Phi$, determines the type of GLM. For example, $\Phi(t) = 0.5 t^{2}$ gives rise to Lasso, while $\Phi(t) = \ln\left(1+\exp(t)\right)$ and $\Phi(t) = \exp(t)$ yield $\ell_{1}$-regularized logistic regression ($\ell_{1}$-LR) and $\ell_{1}$-regularized Poisson regression ($\ell_{1}$-PR), respectively.
\end{example}	
	
\begin{example}
	The problem of estimating a sparse undirected graphical model from empirical covariance matrix through the use of $ \ell_{1} $-type regularization gives rise to graphical lasso~\cite{friedman2008sparse}. Graphical Lasso is typically written as minimizing the penalized negative log-likelihood 
	$$\hat{X} = \arg \min_{X \succeq 0} \; \text{trace}(CX) - \log \det(X) + \lambda \|X\|_{1},$$ where $C$ is the empirical covariance matrix. Graphical Lasso is essentially the matrix equivalent of the linear regression Lasso.
\end{example}	
	
\begin{example}
Sparse matrix decompositions and approximations~\cite{hastie2015statistical} constitute a large class of problems in which the goal is to find an estimate which is close to a (partially observed) data in the form of a matrix $X^* \in \mathbb{R}^{m \times n}$ while satisfying certain properties, such as sparsity or low-rankness. This class of applications has particularly been, and continues to be, an active area of  research in recent years. In such problems, the objective is to find the estimate matrix $\hat{B}$ such that, for example, 
	$$\hat{X} = \arg \min \|\mathcal{P}(X^{*}) - \mathcal{P}(X)\|_{F}^{2} + \lambda \mathcal{R}(X),$$ 
	where $\mathcal{P}$ is the projection onto the observed set, $\|.\|_{F}$ is the Frobenius norm of a matrix, and $\mathcal{R}$ is a regularization that encourages $\hat{X}$ to satisfy certain structure.  The manner in which we define $\mathcal{R}$ leads to a variety of useful procedures, e.g., sparse matrix approximation is given by $\mathcal{R}(X) = \|X\|_{1}$ ($\|.\|_{1}$ is the sum of the absolute values of the matrix entries), low-rank matrix approximation is done via setting $\mathcal{R}(X) = \|X\|_{*}$ ($\|\cdot\|_{*}$ is the nuclear norm), sparse Principal Components Analysis (PCA) and other sparse-and-low rank additive matrix decompositions are given by considering $X = L + R$ with $\mathcal{R}(X) = \alpha \|L\|_{1} + \beta \|R\|_{*} $;
\end{example}	

In problems of the form~\eqref{eq:obj} with non-smooth $ h $, sub-gradient methods~\cite{bubeck2015convex,bertsekas2015convex,bagirov2014introduction} can result in algorithms with sub-linear convergence rate of order
\begin{align*}
F(\xx_{k}) - \min_{\xx \in \C} F(\xx) \leq \mathcal{O}\left(\frac{1}{\sqrt{k}}\right),
\end{align*}
where $ \xx_{k} $ is the $ k $-th iterate. However, if $ h $ is ``simple'', then algorithms with superior convergence rates exist. 
In particular, the class of Iterative Shrinkage-Thresholding Algorithms (ISTA) algorithms,~\cite{bredies2008linear,combettes2005signal,daubechies2004iterative,parikh2014proximal}, can improve upon the slow rate of sub-gradient methods and, indeed, recover the convergence rate of the standard gradient descent method, i.e., 
\begin{align*}
	F(\xx_{k}) - \min_{\xx \in \C} F(\xx) \leq \mathcal{O}\left(\frac{1}{k}\right).
\end{align*}

However, ISTA, both empirically and theoretically, has been shown to be very slow, e.g., see~\cite[section 6]{bredies2008linear}. As a result, there have been many efforts to \emph{accelerate} ISTA by non-trivial modifications, all of which are multi-step methods, i.e., the next iterate is computed from several previous ones. Most notably, the celebrated Fast Iterative Shrinkage-Thresholding Algorithm (FISTA)~\cite{beck2009fast} exploits smoothness of $ f $ and simple structure of $ h $ to improve the convergence rate to
\begin{align*}
	F(\xx_{k}) - \min_{\xx \in \C} F(\xx) \leq \mathcal{O}\left(\frac{1}{k^{2}}\right),
\end{align*}
which is known to be optimal for first order (gradient) methods \cite{nemirovskiĭ1983problem} and matches that of Nesterov's accelerated algorithms~\cite{nesterov1983method,nesterov2004introductory} for smooth problems. Similar accelerated multi-step methods have also been investigated for solving non-smooth problems of the form~\eqref{eq:obj}, e.g.,~\cite{nesterov2007gradient,bioucas2007new,elad2007coordinate,tseng2008proximal}. The great theoretical properties as well as empirical performance of such accelerated methods have prompted many authors to try to understand the underlying mechanism and the natural scope of the acceleration
concept, e.g., physical momentum, relations to other first-order algorithms as well as geometrical and continuous-time dynamics point of view~\cite{bubeck2015geometric,flammarion2015averaging,lessard2016analysis,allen2014linear,su2014differential,krichene2015accelerated,wibisono2016variational}. Most relevant to the present paper is the result of~\cite{allen2014linear} in which an acceleration scheme can was designed by an appropriate \emph{linear coupling} of the gradient and mirror descent steps to draw upon their complementary characteristics. The insightful idea of~\cite{allen2014linear} constitutes the first main ingredient for our proposed algorithms.

In addition to acceleration through a multi-step scheme and employing information from previous iterates, another approach to improve the empirical as well as the theoretical properties of first order methods for~\eqref{eq:obj} is by incorporating previous sub-gradients in the form of adaptively choosing a preconditioner for each gradient (mirror) step. This idea was first pioneered in Adagrad~\cite{duchi2011adaptive}, a sub-gradient method designed for online learning,~\cite{hazan2016introduction}. Through the use of the history of the sub-gradients from previous iterations, Adagrad scales the current sub-gradient to \emph{adapt} to the geometry of the domain. In particular, the coordinates of the search direction are non-uniformly scaled in order to take larger steps along the coordinates with smaller sub-derivatives and, correspondingly, smaller steps along those with larger sub-derivatives. Loosely speaking, this makes the optimization problem better-conditioned. For these reasons, Adagrad has been shown to be highly-suited to data fitting problems with, for example, sparse data~\cite{dean2012large,pennington2014glove}. This work has led to many related algorithms that have been widely used in machine learning applications, e.g., RMSProp~\cite{tijmen2012rmsprop}, ESGD~\cite{dauphin2015equilibrated}, Adam~\cite{kingma2014adam}, and Adadelta~\cite{zeiler2012adadelta}. The second critical ingredient for our algorithmic design is based on this successful idea of \emph{adaptivity} for non-uniform scaling of the search direction's coordinates. 

In this paper, we present methods which offer the best of both worlds. More precisely, we draw upon the ideas of linear coupling~\cite{allen2014linear} and adaptivity~\cite{duchi2011adaptive}, introduce a fast linearly-coupled adaptive gradient method (FLAG) along with its relaxation (FLARE), and demonstrate their theoretical and empirical performance for solving the composite problem~\eqref{eq:obj}. We show that FLAG and FLARE can be equivalently regarded as adaptive versions of FISTA or alternatively, as accelerated versions of AdaGrad adopted for problem~\eqref{eq:obj}. In other words, like Nesterov's accelerated algorithm and its proximal variant, FISTA, our methods achieve the optimal convergence rate of $1/k^2$ and like AdaGrad our methods adaptively choose a regularizer, in a way that performs almost as well as the best choice of regularizer in hindsight. These two desirable effects contribute to the improved theoretical properties as well as practical performance of FLAG and FLARE.

The rest of this paper in organized as follows. Notation, assumptions and definitions used throughout the paper are introduced in Section~\ref{sec:notation}. Our main algorithm, FLAG, and its theoretical properties are presented in Section~\ref{sec:flag}. FLAG can, at times, require more computational effort than FISTA due to the sub-routine involving the linear coupling. As a result, in Section~\ref{sec:flare}, we present a relaxed version of FLAG, dubbed FLARE, which by replacing this potentially expensive step of FLAG, alleviates this problem. Sections~\ref{sec:experiments} contains extensive numerical experiments demonstrating the performance of FLAG and FLARE as compared with FISTA. Conclusions and further thoughts are gathered in Section~\ref{sec:conclusions}. The details of the proofs are deferred to Appendix~\ref{sec:appendix}.

%The linear coupling framework in~\cite{allen2014linear} use a fixed coupling ratio to combine gradient and mirror
%descent steps. However,~\cite{bubeck2015geometric} views acceleration as an ellipsoid-like algorithm, and instead, determines an
%appropriate ratio using a line search. This latter idea is also crucial
%for the design of out algorithm, since by introduction of adaptivity, we will no longer know the right ratio in advance. Finally, our method for picking the adaptive stepsizes is closely related
%to the approach in~\cite{hazan2007adaptive}. There, the authors considered the problem of picking the right stepsize to adapt to an
%unknown strong-convexity parameter. We use a related approach, but with different proof techniques, to choose
%the stepsize to adapt to a changing smoothness parameter.

\subsection{Notation, Assumptions and Definitions}
\label{sec:notation}

\paragraph{Notation:}
In what follows, vectors are considered as column vectors and are denoted by bold lower case letters, e.g.,
$\xx$ and matrices are denoted by regular capital letters, e.g., $A$. We overload the ``diag'' operator as follows: for a given matrix $A$ and a vector $\xx$, $\text{diag}(A)$ and
$\text{diag}(\xx)$ denote the vector made from the diagonal elements of $A$ and a
diagonal matrix made from the elements of $\vv$, respectively. Vector
norms $\|\xx\|_{1}$, $\|\xx\|_{2}$ and $\|\xx\|_{\infty}$ denote the
standard $\ell_{1}$, $\ell_{2}$  and $\ell_{\infty}$ respectively. We adopt the \texttt{Matlab} notation for accessing the elements of vectors and matrices, i.e., $i^{th}$ components of a vector $\xx$ is indicated by $\xx(i)$ and $A(i,:)$ denotes the entire $i^{th}$ row of the matrix $A$. Finally, $A_{k} \gets [A_{k-1}, \vv]$ signifies that $A_{k}$ is the augmentation of the matrix $A_{k-1}$ with the column vector $\vv$. The optimal value of $ F $ is denoted by $ F^{*} = \min_{\xx \in \C} F(\xx)$. Finally, the sub-differential of a convex function, $ h $, at a point, $\xx$, is denoted by $\partial h(\xx)$.

\paragraph{Assumptions:}
Throughout this paper, we make the following assumptions for $ f $ and $ h $. 
\begin{enumerate}[label = {\bfseries A.\arabic*}]
	\item \label{assmpt:f} $f$ is convex and continuously differentiable with $L$-Lipschitz continuous gradient, i.e.,
	\begin{align}
	\label{eq:Lip}
	\|\nabla f(\xx) - \nabla f(\yy) \|_{2} \leq L \|\xx- \yy\|_{2}, \quad \forall \xx, \yy \in \C,
	\end{align} 
	and
	\item \label{assmpt:h} $ h $ is convex (but possibly non-smooth).
\end{enumerate}

\paragraph{Definitions:}
The proximal operator~\cite{parikh2014proximal} associated with $f$, $h$ and $L$ is defined as
\begin{align}
\label{eq:prox}
	\prox(\xx) \defeq %\arg\min_{\yy \in \C} h(\yy) + \lin{\nabla f(\xx), \yy-\xx} +\frac{L}{2}\|\yy-\xx\|_2^2 \\
	\arg\min_{\yy\in \C} \; h(\yy) + \frac{L}{2} \|\yy - \big({\xx} - \frac{1}{L}\nabla f(\xx)\big)\|_{2}^{2}.
\end{align}
For a symmetric positive definite (SPD) matrix $S$, define $\psi(\xx) \defeq \xx^{T} S \xx/2 = \|\xx\|_{S}^{2}$. 
The Bregman divergence associated with $\psi$ is defined as $B_{\psi} (\xx, \yy) \defeq \psi(\xx) -\psi(\yy) - \lin{\nabla \psi(\yy), \xx-\yy} = 0.5 \| \xx - \yy\|^{2}_{S}$. It is easy to see that the dual of $\psi(\xx)$ is given by 
\begin{align}
\label{eq:dual_norm}
\psi^{*}(\xx) = \sup_{\vv \in \reals^{d}} \lin{\xx,\vv} - \psi(\vv)  = \hf \xx^{T} S^{-1} \xx = \|\xx\|_{S^{-1}}^{2}.
\end{align}
Note that $\psi$ is $1$-strongly convex with respect to the norm $\|\xx\|_{S} \defeq \sqrt{\xx^{T} S \xx}$, i.e., $\quad \forall \xx,\yy \in \C$, we have $\psi(\xx) \geq \psi(\yy) + \lin{\nabla \psi(\yy), \xx - \yy} + \hf \| \xx - \yy\|^{2}_{S}$.
Finally, throughout our analysis, we will use the fact that, for any $\zz \in \C$,
\begin{align}
\label{eq:breg_eq}
&\lin{S (\xx - \yy),\xx - \zz} \nonumber \\
= &\hf \|\xx - \yy\|_{S}^{2} + \hf \|\xx - \zz\|_{S}^{2} - \hf \|\yy - \zz\|_{S}^{2}.
%
%\lin{S (\xx - \yy),\xx - \zz} &= \lin{\nabla_{\xx} B_{\psi} (\xx, \yy), \xx - \zz} \\
%&= \lin{\nabla \psi(\xx) -\nabla \psi(\yy), \xx - \zz} \\
%&= B_{\psi} (\xx, \yy) + B_{\psi} (\zz, \xx) - B_{\psi} (\zz, \yy) \\
%&= \hf \|\xx - \yy\|_{S}^{2} + \hf \|\xx - \zz\|_{S}^{2} - \hf \|\yy - \zz\|_{S}^{2}.
\end{align}

\section{FLAG}
\label{sec:flag}
In this section, we present our main algorithm, FLAG (Algorithm~\ref{alg:flag}), and give its main convergence properties in Theorem~\ref{thm:final_flag}. As mentioned in Section~\ref{sec:introduction}, FLAG incorporates techniques from linear coupling of~\cite{allen2014linear} and adaptivity of~\cite{duchi2011adaptive}. At a very high-level, the core of FLAG consists of the following five essential ingredients. 
\begin{enumerate}[label = \arabic*.]
	\item A gradient descent step (Step~\ref{alg_step_grad} of Algorithm~\ref{alg:flag}),
	\item Construction of the adaptive regularization (Steps~\ref{alg_step_grad_map}-\ref{algstep:S_k}  of Algorithm~\ref{alg:flag}),
	\item Update of the adaptive stepsize (Step~\ref{algstep:eta_k}  of Algorithm~\ref{alg:flag}),
	\item A mirror descent step (Step~\ref{algstep:mirror}  of Algorithm~\ref{alg:flag}),
	\item Linear combination of the gradient and the mirror descent directions (Step~\ref{algstep:comb} of Algorithm~\ref{alg:flag}).
\end{enumerate}

\begin{algorithm}[H]
	\caption{FLAG}
	\label{alg:flag}
	\begin{algorithmic}[1]
		\STATEx \textbf{Input:} $\xx_1 = \yy_1 = \zz_1$, $\eta_0 = 0$, $G_{0} = [\; \text{empty}\;]$, $ \delta > 0$, $T$, and $\epsilon = {1}/{(6 d T^3) }$ \smallskip
		\FOR{k=1 to T} \smallskip
		\STATE \label{alg_step_grad} $\displaystyle  \yy_{k+1} \gets \prox(\xx_{k})$ \smallskip
		\STATE \label{alg_step_grad_map} $\displaystyle  \pp_k \gets -L(\yy_{k+1}-\xx_{k})$ \smallskip
		\STATE $\displaystyle   \vg_{k} \gets \frac{\pp_k}{\|\pp_k \|_2}$ \smallskip
		\STATE \label{algstep:g_ik} $\displaystyle  G_{k} \gets \big[G_{k-1}, \vg_{k}\big]$ \smallskip
		\STATE \label{algstep:s_ki} $\displaystyle  \vs_{k}(i) \gets \|G_{k}(i,:)\|_2$ \smallskip
		\STATE \label{algstep:S_k}$\displaystyle S_k \gets \text{diag}(\vs_k) + \delta \mathbb{I}$ \smallskip
		\STATE \label{algstep:L_k} $\displaystyle  L_k \gets L{ \vg_{k}^{T} S^{-1}_{k} \vg_{k}}$ \smallskip
		\STATE \label{algstep:eta_k} $\displaystyle  \eta_k \gets \frac{1}{2 L_k} + \sqrt{\frac{1}{4 L_k^2} + \frac{\eta_{k-1}^2 L_{k-1}}{L_k} } $ 
		\STATE \label{algstep:mirror} $\displaystyle  \zz_{k+1} \gets \arg\min_{\zz\in \C} \lin{\eta_k \pp_k , \zz -\zz_{k}} + \hf \|\zz-\zz_{k}\|^{2}_{S_{k}}$ \smallskip
		\STATE \label{algstep:comb} $\displaystyle  \xx_{k+1}\gets \text{BinarySearch}(\zz_{k+1},\yy_{k+1}, \epsilon)$ \smallskip
		\ENDFOR \smallskip
		\STATEx \textbf{Output:} $\yy_{T+1}$
	\end{algorithmic}
\end{algorithm}

The subroutine \textit{BinarySearch} is given in Algorithm~\ref{alg:binsearch}, where $\text{\textit{Bisection}}(r,0,1,\epsilon)$ is the usual bisection routine for finding the root of a single variable function $r(t)$ in the interval $(0,1)$ and to the accuracy of $\epsilon$. More specifically, for a root $r^*$ such that $r(t^*) = 0$ and given $r(0) \cdot r(1) < 0$, the sub-routine $\text{\textit{Bisection}}(r,0,1,\epsilon)$ returns an approximation $t \in (0,1)$ to $t^*$ such that $|t - t^*| \leq \epsilon$ and this is done with only $\log (1/\epsilon)$ function evaluations; see~\cite[Section 3.2]{ascher2011first} for details and  example \texttt{Matlab} code. 
%In other words, if $w$ is such that $\lin{prox(w)-w, y-z}=0$, then \textit{BinarySearch} returns an $x$ such that $\|x-w\|_2\leq \epsilon \|y-z\|_2$.
\begin{algorithm}[H]
	\caption{BinarySearch}
	\label{alg:binsearch}
	\begin{algorithmic}[1]
		\STATEx \textbf{Input:} $\zz$, $\yy$, and $\epsilon$
		\STATE Define the univariate function $r(t) \defeq \lin{\prox\left(t \yy + (1-t)\zz\right)- \left(t \yy + (1-t)\zz\right) ,\yy-\zz}$
		\IF{$r(1) \geq 0$} \label{bin_step_y}
		\STATE Return $\yy$
		\ENDIF
		\IF{$r(0)\leq 0$} \label{bin_step_z}
		\STATE Return $\zz$
		\ENDIF
		\STATE $t = \text{Bisection}(r,0,1,\epsilon)$
		\STATE Return  $\xx = t \yy + (1-t)\zz$
	\end{algorithmic}
\end{algorithm}

We are now ready to present our main result, Theorem~\ref{thm:final_flag}, which gives the convergence properties of Algorithm~\ref{alg:flag}.
\begin{theorem}[Convergence of FLAG]
	\label{thm:final_flag}
	Let Assumptions~\ref{assmpt:f} and~\ref{assmpt:h} hold and define $\displaystyle D_{\infty} \defeq \sup_{\xx, \yy \in \C} \|\xx - \yy\|_\infty$. For any $\uu \in \C$, after $T$ iterations of Algorithm \ref{alg:flag}, we get
	\begin{equation*}
	F(\yy_{_{T+1}}) - F(\uu) \leq \mathcal{O}\left(\frac{\beta L D_{\infty}^{2}}{T^2}\right),
	\end{equation*}
	for some $\beta \in[1,d]$. Furthermore, each iteration takes time at most $\mathcal{O}(\mathcal{T}_{_{\textbf{\prox}}} \cdot\log(dT^3))$, where $ \mathcal{T}_{_{\textbf{\prox}}} $ is the cost of evaluating $ \prox $ in~\eqref{eq:prox}.
\end{theorem}

%\subsection{Comparison of FLAG with AdaGrad and FISTA}
%\label{sec:comparison}
\begin{enumerate}[leftmargin=*,wide=0em, noitemsep,nolistsep, label = {\bfseries Remark \arabic*:}]
	\item Recall that the convergence rate of FISTA is given by
	\begin{equation*}
	F(y_{_{T+1}}) - F^{*} \leq \mathcal{O}\left(\frac{L D_{2}^{2}}{T^2}\right),
	\end{equation*}
	where
	$\displaystyle D_{2} \defeq \sup_{\xx, \yy \in \C} \|\xx - \yy\|_2$~\cite[Theorem 4.4]{beck2009fast}. A quick comparison between FISTA's upper bound with that of FLAG in Theorem~\ref{thm:final_flag} implies that the ``competitive factor'' of FLAG over FISTA is 
	\begin{align*}
	\text{Competitive Factor} = \frac{\beta D_{\infty}^{2}}{D_{2}^{2}}.
	\end{align*}
	For~\eqref{eq:obj}, consider a box-constraint of the form $ \mathcal{C} = \{\xx; \; \|\xx\|_{\infty} \leq 1\} $. It is easy to see that since $ D_{2} = \sqrt{d} D_{\infty} $, we have $ \text{Competitive Factor} \in [1/d,1] $. In such settings, the adaptivity introduced by FLAG can offer a significant improvement in the convergence properties. This is indeed similar to the improvement obtained by Adagrad over proximal sub-gradient methods\footnote{For the non-smooth settings considered by Adagrad, the competitive factor is in the form of ${\sqrt{\beta} D_{\infty}}/{D_{2}}$ }.
	
	\item
	From the proof of Theorem~\ref{thm:final_flag}, we can see that 
	\begin{align*}
	\beta = \frac{\left(\sum_{i=1}^{d} \sqrt{\sum_{t=1}^{T} [\tilde \bgg_{t}]_{i}^2}\right)^{2}}{ T },
	\end{align*}
	where $\tilde \bgg_{t} \defeq \bgg_{t}/ \|\bgg_{t}\|$. For illustration purposes only, let us consider $d=4$, and $T = 3$. Indeed, for the following gradient histories, we can verify that $\beta \in [1,4] $, e.g., 
	\begin{align*}
	[\tilde \bgg_1, \tilde \bgg_2, \tilde \bgg_3] &= \begin{bmatrix}
	1 & -1 & -1 \\
	0 & \;\;\;0 & \;\;\;0\\
	0 & \;\;\;0 & \;\;\;0\\
	0 & \;\;\;0 & \;\;\;0 
	\end{bmatrix} \Longrightarrow \beta = 1, \\
	[\tilde \bgg_1, \tilde \bgg_2, \tilde \bgg_3] &= \begin{bmatrix}
	1 & \;\;\;0 & \;\;\;0 \\
	0 & \;\;\;1 & \;\;\;0 \\
	0 & \;\;\;0 & \;\;\;0 \\
	0 & \;\;\;0 & \;\;\;1 \\
	\end{bmatrix} \Longrightarrow \beta = 3, \\
	[\tilde \bgg_1, \tilde \bgg_2, \tilde \bgg_3] &= \begin{bmatrix}
	\vspace{1mm} \hf & \;\;\;\hf & -\hf \\ \vspace{1mm}
	\hf & \;\;\;\hf & \;\;\;\hf\\ \vspace{1mm}
	\hf & -\hf & -\hf\\ \vspace{1mm}
	\hf & \;\;\;\hf & -\hf 
	\end{bmatrix} \Longrightarrow \beta = 4.
	\end{align*} 
\end{enumerate}

\section{FLARE}
\label{sec:flare}
The ``BinarySearch'' in Step~\ref{algstep:comb} of Algorithm~~\ref{alg:flag} can be the bottleneck of the computations. Indeed, from Theorem~\ref{thm:final_flag} it can be seen that the running time of FLAG, in the worst case, is dominated by the number of $ \prox $  evaluations involved in the root finding procedure of Algorithm~\ref{alg:binsearch}. As a result, despite the fact that FLAG achieves the same accelerated convergence rate as FISTA, its per-iteration cost can be much higher than what adaptivity can make up for; see examples of Section~\ref{sec:experiments}. In this section, we modify FLAG to obtain a relaxed version, FLARE, whose per-iteration complexity is theoretically similar to FLAG, but empirically is shown to be almost identical to that of FISTA, i.e., $ \mathcal{O}(\mathcal{T}_{\prox}) $. 

The proposed relaxation in FLARE will be done by ``guessing'' $ L_{k} $ in FLAG, i.e., Step~\ref{algstep:L_k} of Algorithm~\ref{alg:flag}, at iteration $ k $  and performing the update without immediately resorting to ``BinearySearch''. We subsequently verify in the next iteration that the guessed $L_k$ is not too far from the truth; otherwise, we repeat the previous iteration with a better guess. The resulting relaxation is given in Algorithm~\ref{alg:flare}.

\begin{algorithm}[htb]
	\caption{FLARE}
	\label{alg:flare}
	\begin{algorithmic}[1]
		\STATEx \textbf{Input:} $\xx_1 = \yy_1 = \zz_1$, $\eta_0 = 0$, $G_{0} = [\; \text{empty}\;]$, $ \delta > 0$, $\gamma > 1$, $\lambda > 1$, $T$, and $\epsilon \leq {1}/{(6 d T^3) }$ \smallskip
		\WHILE{$k\leq T$}
		\STATE $\text{accept}\gets \text{FALSE}$
		\STATE $i \gets 0$
		\WHILE{$\text{accept}\  \neq \text{TRUE}$}
		$i=i+1$
		\IF {$i \leq \log \frac{d}{\epsilon}$}
		\STATE \label{e:guess} $\tilde{L}_{k} \gets L_{k-1}\cdot \gamma^i$
		\STATE \label{e:accept_verify} $\displaystyle [\eta_k, \xx_k, G_k, S_k, L_k, \yy_{k+1}, \zz_{k+1}, \text{accept}]\gets \text{A$\&$V}(G_{k-1}, S_{k-1}, \eta_{k-1}, \tilde{L}_{k-1}, \yy_k,\zz_k, \tilde{L}_k, \lambda)$ 
		\ELSE 
		\STATE \label{e:flag_iteration} $\displaystyle [\eta_k, \xx_k, G_k, S_k, L_k, \yy_{k+1}, \zz_{k+1}]\gets \text{FLAGIteration}(G_{k-1}, S_{k-1}, \eta_{k-1}, \tilde{L}_{k-1}, \yy_k,\zz_k, \tilde{L}_k)$
		\STATE $\text{accept}\gets \text{TRUE}$
		\ENDIF
		\ENDWHILE
		\ENDWHILE
	\end{algorithmic}
\end{algorithm}

Algorithm~\ref{alg:flare} involves three main steps. Step~\ref{e:guess} aims at guessing a viable value for $ L_{k} $, which can be used at the present iteration. As depicted here and used in our numerical experiments, we have considered guessing $ L_{k} $ with some multiple of the known $ L_{k-1} $. However, Step \ref{e:guess} can be replaced by any reasonable subroutine that tries to guess the valid ratio for $\tilde{L}_k$. Step~\ref{e:accept_verify} contains a subroutine, dubbed ``A$\&$V'' (short for \emph{Advance and Verify}), which computes $\xx_k$, $\yy_{k+1}$ and $\zz_{k+1}$ using the guess $\tilde{L}_k$, and returns ``accept=TRUE'' if $\tilde{L}_k$ is a sufficiently good guess. Finally, Step~\ref{e:flag_iteration} involves the ``FlagIteration'' subroutine, which, by reverting back to using ``BinarySearch'', computes $\xx_k$, $\yy_{k+1}$ and $\zz_{k+1}$. This step is almost identical to one full iteration of FLAG in \eqref{alg:flag}, though the statements are ordered differently. As shown in the proof of Theorem~\ref{thm:mastertheorem}, the resulting updates generated from this step are always acceptable. In all of our numerical simulations, however, we have never observed FLARE performing Step~\ref{e:flag_iteration}. In fact, most often, the very first guess in Step~\ref{e:guess} is deemed acceptable by Step~\ref{e:accept_verify} leading to FLARE requiring only one ``$ \prox $'' evaluation per iteration (as in FISTA).

\begin{algorithm}[htb]
	\caption{A$\&$V: Advance and Verify}
	\label{alg:advanceandverify}
	\begin{algorithmic}[1]
	\STATEx \textbf{Input:} $G_{k-1}, S_{k-1}, \eta_{k-1}, \tilde{L}_{k-1}, \yy_k,\zz_k, \tilde{L}_k, \lambda$
		%\STATE $\displaystyle  \xx_{k}\gets \text{BinarySearch}(\zz_{k},\yy_{k}, \epsilon)$ 
		\STATE $\displaystyle  \eta_k \gets \frac{1}{2 \tilde{L}_k} + \sqrt{\frac{1}{4 \tilde{L}_k^2} + \frac{\eta_{k-1}^2 \tilde{L}_{k-1}}{\tilde{L}_k} } $ 
		\STATE $\displaystyle \xx_{k} \gets \Big(1-\frac{1}{\eta_{k} \tilde{L}_{k}}\Big) \yy_{k} + \frac{1}{\eta_{k} \tilde{L}_{k}} \zz_{k}$ 
		\STATE $\displaystyle  \yy_{k+1} \gets \prox(\xx_{k})$
		\STATE $\displaystyle  \pp_k \gets -L(\yy_{k+1}-\xx_{k})$
		\STATE $\displaystyle   \vg_{k} \gets \frac{\pp_k}{\|\pp_k \|_2}$
		\STATE $\displaystyle  G_{k} \gets \big[G_{k-1}, \vg_{k}\big]$
		\label{algstep:sk_i_flare_1}
		\STATE $\displaystyle  \vs_{k}(i) \gets \|G_{k}(i,:)\|_2$
		\STATE $\displaystyle S_k \gets \text{diag}(\vs_k) + \delta \mathbb{I}$
		\STATE \label{algstep:L_k_flare_1}$\displaystyle  L_k \gets L{ \vg_{k}^{T} S^{-1}_{k} \vg_{k}}$
		\STATE $\displaystyle  \zz_{k+1} \gets \arg\min_{\zz\in \C} \lin{\eta_k \pp_k , \zz -\zz_{k}} + \hf \|\zz-\zz_{k}\|^{2}_{S_{k}}$
		\STATE \label{algstep:L_k_accept_condition}accept $\gets \tilde{L}_k \in [L_k, \lambda L_k]$
	\end{algorithmic}
\end{algorithm}

\begin{algorithm}[htb]
	\caption{FlagIteration}
	\label{alg:flagiteration}
	\begin{algorithmic}[1]
	\STATEx \textbf{Input:} $G_{k-1}, S_{k-1}, \eta_{k-1}, \tilde{L}_{k-1}, \yy_k,\zz_k, \tilde{L}_k, \epsilon$
		\STATE $\displaystyle  \xx_{k}\gets \text{BinarySearch}(\zz_{k},\yy_{k}, \epsilon)$ 
		\STATE $\displaystyle  \yy_{k+1} \gets \prox(\xx_{k})$
		\STATE $\displaystyle  \pp_k \gets -L(\yy_{k+1}-\xx_{k})$
		\STATE $\displaystyle   \vg_{k} \gets \frac{\pp_k}{\|\pp_k \|_2}$
		\STATE $\displaystyle  G_{k} \gets \big[G_{k-1}, \vg_{k}\big]$
		\label{algstep:sk_i_flare_2}
		\STATE $\displaystyle  \vs_{k}(i) \gets \|G_{k}(i,:)\|_2$
		\STATE $\displaystyle S_k \gets \text{diag}(\vs_k) + \delta \mathbb{I}$
		\STATE \label{algstep:L_k_flare_2}$\displaystyle  L_k \gets L{ \vg_{k}^{T} S^{-1}_{k} \vg_{k}}$
		\STATE \label{algstep:L_k_flare_3}$\displaystyle  \tilde{L}_k \gets L_k$
		\STATE $\displaystyle  \eta_k \gets \frac{1}{2 \tilde{L}_k} + \sqrt{\frac{1}{4 \tilde{L}_k^2} + \frac{\eta_{k-1}^2 \tilde{L}_{k-1}}{\tilde{L}_k} } $ 
		\STATE $\displaystyle  \zz_{k+1} \gets \arg\min_{\zz\in \C} \lin{\eta_k \pp_k , \zz -\zz_{k}} + \hf \|\zz-\zz_{k}\|^{2}_{S_{k}}$
	\end{algorithmic}
\end{algorithm}

%%%%%%%%%%%%%
%%%%%%%%%%%%

The following result describes the main convergence properties of FLARE.
\begin{theorem}[Convergence of FLARE]
	\label{thm:final_flare}
	Let Assumptions~\ref{assmpt:f} and~\ref{assmpt:h} hold and define $\displaystyle D_{\infty} \defeq \sup_{\xx, \yy \in \C} \|\xx - \yy\|_\infty$. For any $\uu \in \C$, after $T$ iterations of Algorithm \ref{alg:flare}, we get
	\begin{equation*}
	F(\yy_{_{T+1}}) - F(\uu) \leq \mathcal{O}\left(\frac{\lambda  \beta L D_{\infty}^{2}}{T^2}\right),
	\end{equation*}
	for some $\beta \in[1,d]$, where $\lambda$ is a constant specified in the input to Algorithm \ref{alg:flare}. Furthermore, each iteration takes time at most $\mathcal{O}(\mathcal{T}_{_{\textbf{\prox}}} \cdot\log(dT^3))$, where $ \mathcal{T}_{_{\textbf{\prox}}} $ is the cost of evaluating $ \prox $ in~\eqref{eq:prox}.
\end{theorem}

Note that by Theorem~\ref{thm:final_flare}, the overall worst-case iteration complexity of FLARE (Algorithm~\ref{alg:flare}), is similar to that of FLAG (Algorithm~\ref{alg:flag}). This is mainly due to the fact that, in worst case, Algorithm \ref{alg:flare} can end up resorting to ``BinarySearch'' when repeated guessing fails. However, through extensive numerical experiments, we have observed that we rarely require more than one ``$ \prox $'' evaluation per iteration. In particular, Step~\ref{e:guess} and~\ref{e:accept_verify} of Algorithm~\ref{alg:flare}, most often, are only performed once, while Step~\ref{e:flag_iteration} is never executed.

\section{Numerical Experiments}
\label{sec:experiments}
We now numerically illustrate the performance of FLAG and FLARE in comparison to FISTA. We first consider comparing the performance of these algorithms with respect to the \emph{total number of iterations}. Admittedly, ``performance vs.\ iterations'' is an unfair measure of comparing these algorithms. Indeed, each iteration of FLAG and FLARE can involve more ``$\prox$'' evaluations than FISTA, and as noted in Section 2, in the worst case such ``$\prox$'' evaluations can dominate the running time. Therefore, we subsequently evaluate these algorithms as measured by \emph{total number of $\prox$ evaluations}, which is arguably more indicative of real world performance. In this light, we demonstrate that FLARE and FISTA perform favorably with respect to FLAG, with FLARE consistently outperforming the rest.  

%\subsection{Experimental Setup}
%\label{sec:experimental_setup}
We compare FLAG, FLARE and FISTA on both regression and classification tasks. For regression experiments, we utilized squared loss $ f(\xx) = \frac{1}{2} \| A\xx - \bb\| _{2}^2 $, where $ A \in \mathbb{R}^{n \times d} $ and $ b\in\mathbb{R}^{n} $ are, respectively, the data matrix and the response vector. For classification experiments, we employed a softmax classifier. In such a classifier, given $C$ classes and a data point $\aa$, the probability that $\aa$ belongs to a class $c \in \{1,2,\ldots,C\}$ is given as
\begin{align*}
	\Pr\left(c|\aa,\xx_{1},\ldots, \xx_{C}\right) = \frac{e^{\lin{\aa,\xx_{c}}}}{\sum_{c' = 1}^{C} e^{\lin{\aa, \xx_{c'}}}},
\end{align*}
where $ \xx_{c} \in \mathbb{R}^{p}$ is the weight vector corresponding to class $ c $. Recall that here there are only $ C-1 $ degrees of freedom, i.e., probabilities all must sum to one. Consequently, for a training data $\{\aa_{i},b_{i}\}_{i=1}^{n} \subset \mathbb{R}^{p} \times \{1,2,\ldots,C\}$, the cross-entropy loss function for $ \xx = [\xx_{1}; \xx_{2}; \ldots; \xx_{C-1}] \in \mathbb{R}^{(C-1)p} $ can be written as
\begin{align*}
		&f(\xx) \triangleq f(\xx_{1},\xx_{2},\ldots,\xx_{C-1}) \\
        = &\sum_{i=1}^{n} \left(\log \left(1+\sum_{c' = 1}^{C-1} e^{\lin{\aa_{i}, \xx_{c'}}}\right)  - \sum_{c = 1}^{C-1}\mathbf{1}(b_{i} = c) \lin{\aa_{i},\xx_{c}}\right).
\end{align*} 
Note that here $ d = (C-1) p $. It then follows that the gradient of $f$ with respect to $\xx_c$ is
\begin{align*}
	\nabla_{\xx_{c}} f(\xx) = \sum_{i=1}^{n} \left(\frac{e^{\lin{\aa_{i}, \xx_{c}}}}{1+\sum_{c' = 1}^{C-1} e^{\lin{\aa_{i}, \xx_{c'}}}}  - \mathbf{1}(b_{i} = c) \right) \aa_{i}.
\end{align*}

For each regression and classification formulation of $ f(\xx) $, we consider two variants for $ h(\xx) $ and $ \mathcal{C} $: unconstrained $ \ell_1$ regularization, i.e.,
$ h(\xx) = \lambda \|\xx\|_{1}, \;\mathcal{C} = \mathbb{R}^{d} $, as well as unregularized box-constrained as $ h(\xx) = 0, \; \mathcal{C} = \{\xx; \; \| \xx \|_{\infty} \leq c\} $, where $ \lambda $ and $ c $ are, respectively, the regularization parameter for $ \ell_{1} $ norm and the infinity ball radius. Recall that for regression, the former variant amounts to the celebrated Lasso~\cite{tibshirani1996regression}. In our experiments, we choose $ \lambda = 0.1 $ and $ c = 1 $.  It is well-known that ``$\prox$'' operator for $\ell_1$ regularization is readily given via soft-thresholding~\cite{parikh2014proximal}, 
while in the case of box constraints, it involves the projection of the gradient step onto the infinity ball of the given radius. 
We tested regression and classification tasks on multiple real data sets. Tables~\ref{table:classification} and~\ref{table:regression}, respectively, summarize the data sets used for these tasks.

%\begin{table}[!htbp]
%\caption{Data Sets for Classification Experiments. ``Var.'' refers to variants used for $ h(\xx) $ and $ \mathcal{C} $, i.e., ``Box'' for box-constrained  and ``$ \ell_{1} $'' for $ \ell_{1} $-regularized variants, as mentioned in Section~\ref{sec:experiments}.\label{table:classification}}
%\begin{center}
%\resizebox{9.4cm}{!}{%
%\centering
%\begin{tabular}{|P{1.7cm}|P{1cm}|P{1cm}|P{0.8cm}|P{0.3cm}|P{0.6cm}|P{0.7cm}|}
%\hline
%{\bf Name}  &{$ \bm{n} $} & {\bf Test Size} & {\bf $ \bm{p} $} &{\bf $ \bm{C} $} &{\bf Var.} &{\bf Ref.} \\
%\hline
%\texttt{20 Newsgroups}         & 10,142 & 1,127 &  53,975 & 20 & Box & \cite{lang1995newsweeder}\\\hline
%\texttt{HARUS}             & 7,767 & 3,162 & 561 & 12 & Box & \cite{anguita2013public}\\\hline
%\texttt{Gisette}             & 6,000 & 1,000 & 5,000 & 2 & $\ell_1$ & \cite{guyon2004Result} \\\hline
%\texttt{Forest Covertype}             & 435,759 & 145,253 & 54 & 7 & $\ell_1$ & \cite{libsvm}\\\hline
%\end{tabular}
%}
%\end{center}
%\end{table}

\begin{table}[!htbp]
	\caption{Data Sets for Classification Experiments. ``$ \#\text{Test} $'' indicates the size of the test set. ``Var.'' refers to variants used for $ h(\xx) $ and $ \mathcal{C} $, i.e., ``Box'' for box-constrained  and ``$ \ell_{1} $'' for $ \ell_{1} $-regularized variants, as mentioned in Section~\ref{sec:experiments}.\label{table:classification}}
	\begin{center}
		%\resizebox{9.4cm}{!}{%
			\centering
			\begin{tabular}{|P{.9cm}|P{1.68cm}|P{0.8cm}|P{1.15cm}|P{1.5cm}|}
				\hline
				{\bf Name}  & \texttt{20 Newsgroups} & \texttt{HARUS} & \texttt{Gisette} & \texttt{Forest Covertype} \\ \hline
				{$ \bm{n} $} & 10,142 & 7,767 & 6,000 & 435,759 \\ \hline
				{\bf $ \#\text{Test} $} & 1,127 & 3,162 & 1,000 & 145,253 \\\hline
				{\bf $ \bm{p} $} & 53,975 & 561 & 5,000 & 54 \\\hline
				{\bf $ \bm{C} $} & 20 & 12 & 2 & 7 \\\hline
				{\bf Var.} & Box & Box & $\ell_1$ & $\ell_1$ \\\hline
				{\bf Ref.} & \cite{lang1995newsweeder} & \cite{anguita2013public} & \cite{guyon2004Result} & \cite{libsvm}\\\hline
			\end{tabular}
		%}
	\end{center}
\end{table}
%\FloatBarrier

%\begin{table}[!htbp]
%\caption{Data Sets for Regression Experiments.\label{table:regression}} 
%\begin{center}
%\resizebox{9.4cm}{!}{%
%\begin{tabular}{p{1.55cm}p{1.4cm}p{1.4cm}p{1.25cm}p{1.1cm}p{1.1cm}}
%{\bf Name}  &{\bf Train Set Size} & {\bf Test Set Size} & {\bf Features} &{\bf Variant} &{\bf Ref.} \\
%\hline \\
%\texttt{Blog Feedback} & 47157 & 5240 &  280 & Box & \cite{buza2014feed}\\ 
%\texttt{Facebook CVD} & 36854 & 4095 & 53 & $\ell_1$ & \cite{Sing1503:Comment} \\
%
%\end{tabular}
%}
%\end{center}
%\end{table}

\begin{table}[!htbp]
	\caption{Data Sets for Regression Experiments. ``$ \#\text{Test} $'' indicates the size of the test set. ``Var.'' refers to variants used for $ h(\xx) $ and $ \mathcal{C} $, i.e., ``Box'' for box-constrained  and ``$ \ell_{1} $'' for $ \ell_{1} $-regularized variants, as mentioned in Section~\ref{sec:experiments}.\label{table:regression}}
	\begin{center}
		%\resizebox{9.4cm}{!}{%
		\centering
		\begin{tabular}{|c|c|c|}
			\hline
			{\bf Name}  & \texttt{Blog Feedback} & \texttt{Facebook CVD} \\ \hline
			{$ \bm{n} $} & 47,157 & 36,854  \\ \hline
			{\bf $ \#\text{Test} $} & 5,240 & 4,095\\\hline
			{\bf $ \bm{d} $} & 280 & 53 \\\hline
			{\bf Var.} & Box & $\ell_1$ \\\hline
			{\bf Ref.} & \cite{buza2014feed} & \cite{Sing1503:Comment} \\\hline
		\end{tabular}
		%}
	\end{center}
\end{table}
%\FloatBarrier

We ran FLAG, FLARE, and FISTA for 1000 iterations each on softmax classification for the data sets enumerated in Table~\ref{table:classification}. Both variants of $h(\xx)$ and $ \mathcal{C} $ are represented. The per iteration loss and test accuracy are displayed in Figures~\ref{fig:20NewsIters},~\ref{fig:HARUSIters},~\ref{fig:GisetteIters}, and~\ref{fig:CoverTypeIters}.

\begin{figure}[!htbp]
\centering
%twenty news
\subfigure{
\includegraphics[width=0.45\textwidth]{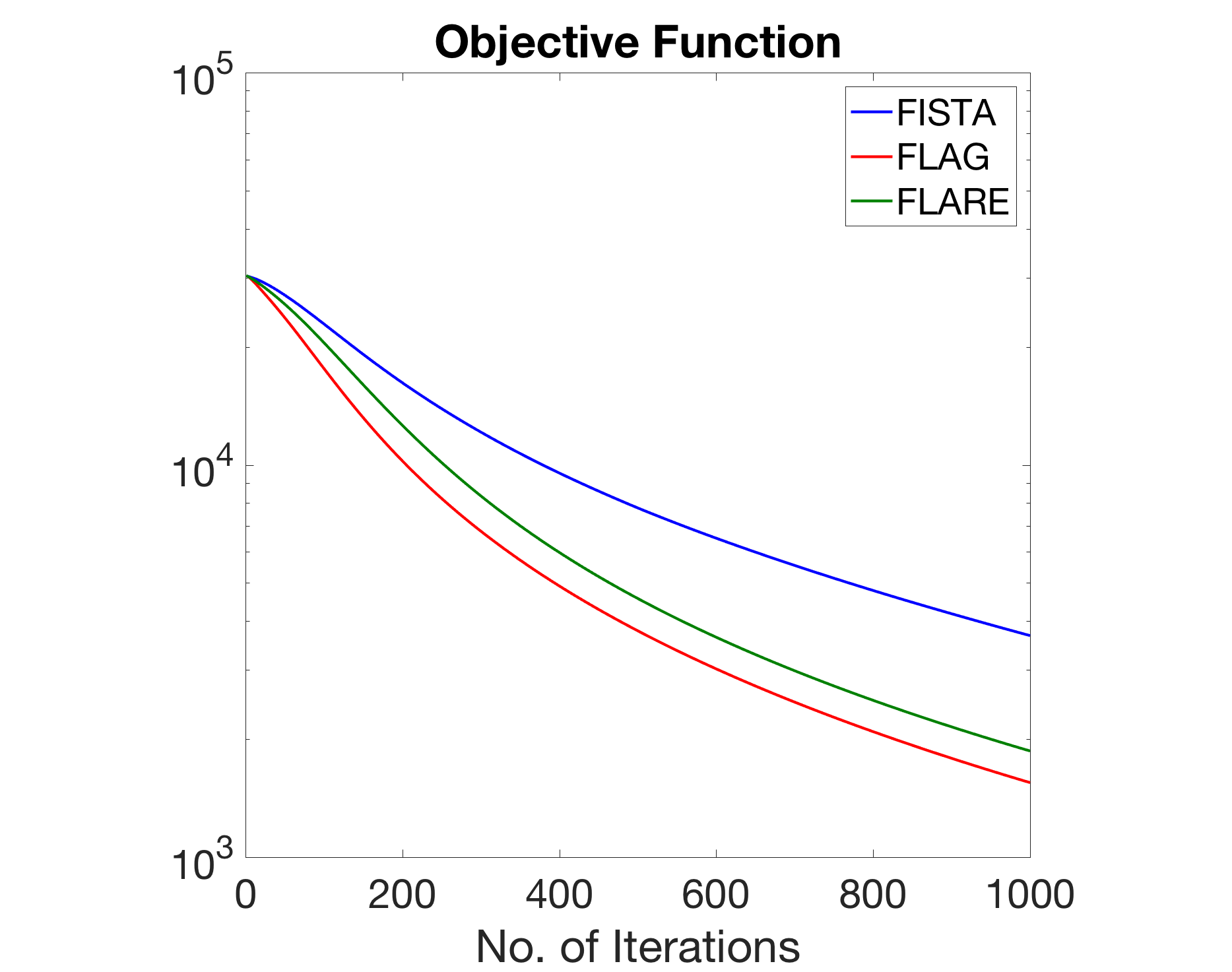}

\includegraphics[width=0.45\textwidth]{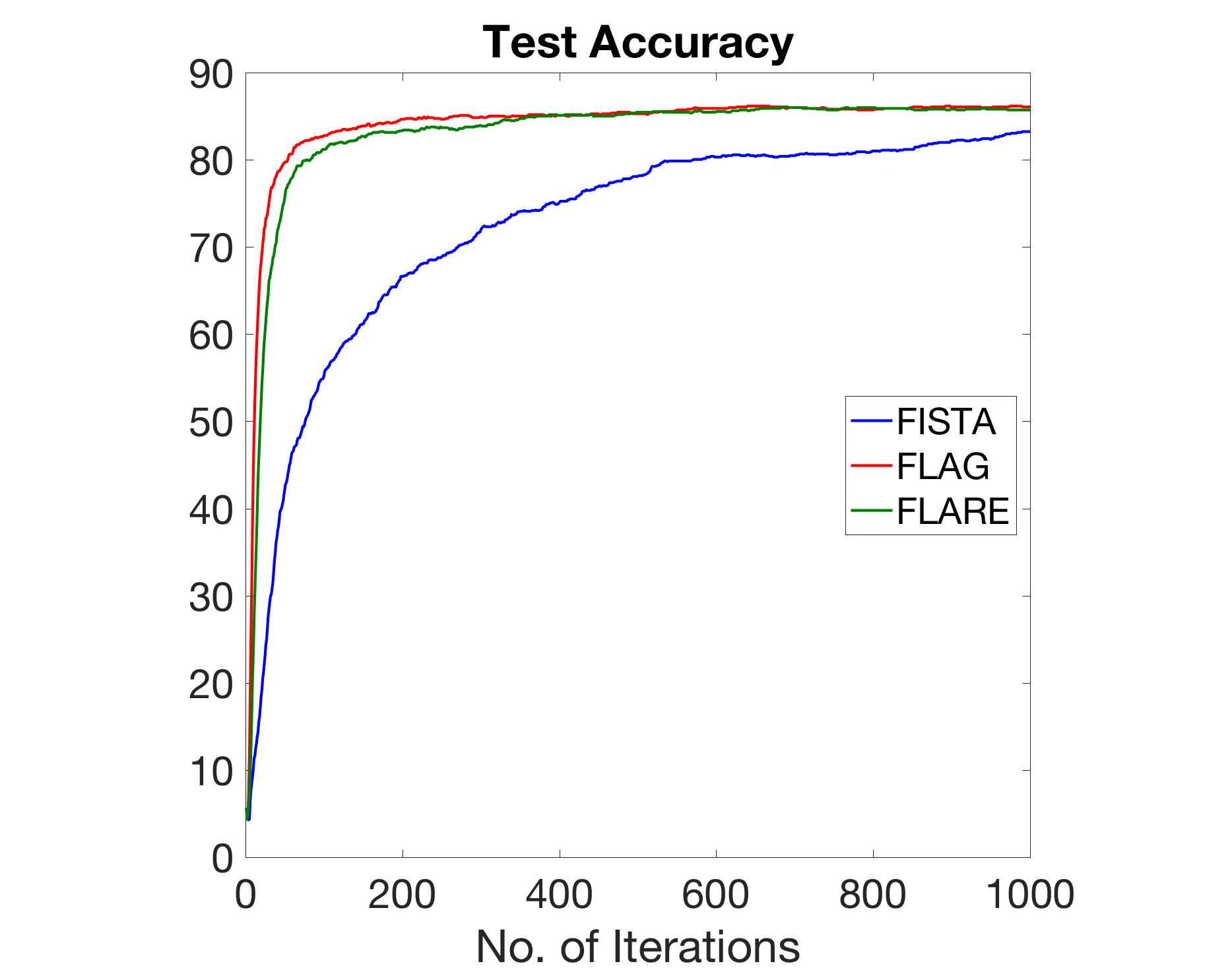}
}

\caption{FLAG, FLARE, and FISTA on box-constrained classification for the \texttt{20 Newsgroups} data set.}
\label{fig:20NewsIters}
\end{figure}
%\FloatBarrier

\begin{figure}[!htbp]
\centering
%HARUS
\subfigure{
\includegraphics[width=0.45\textwidth]{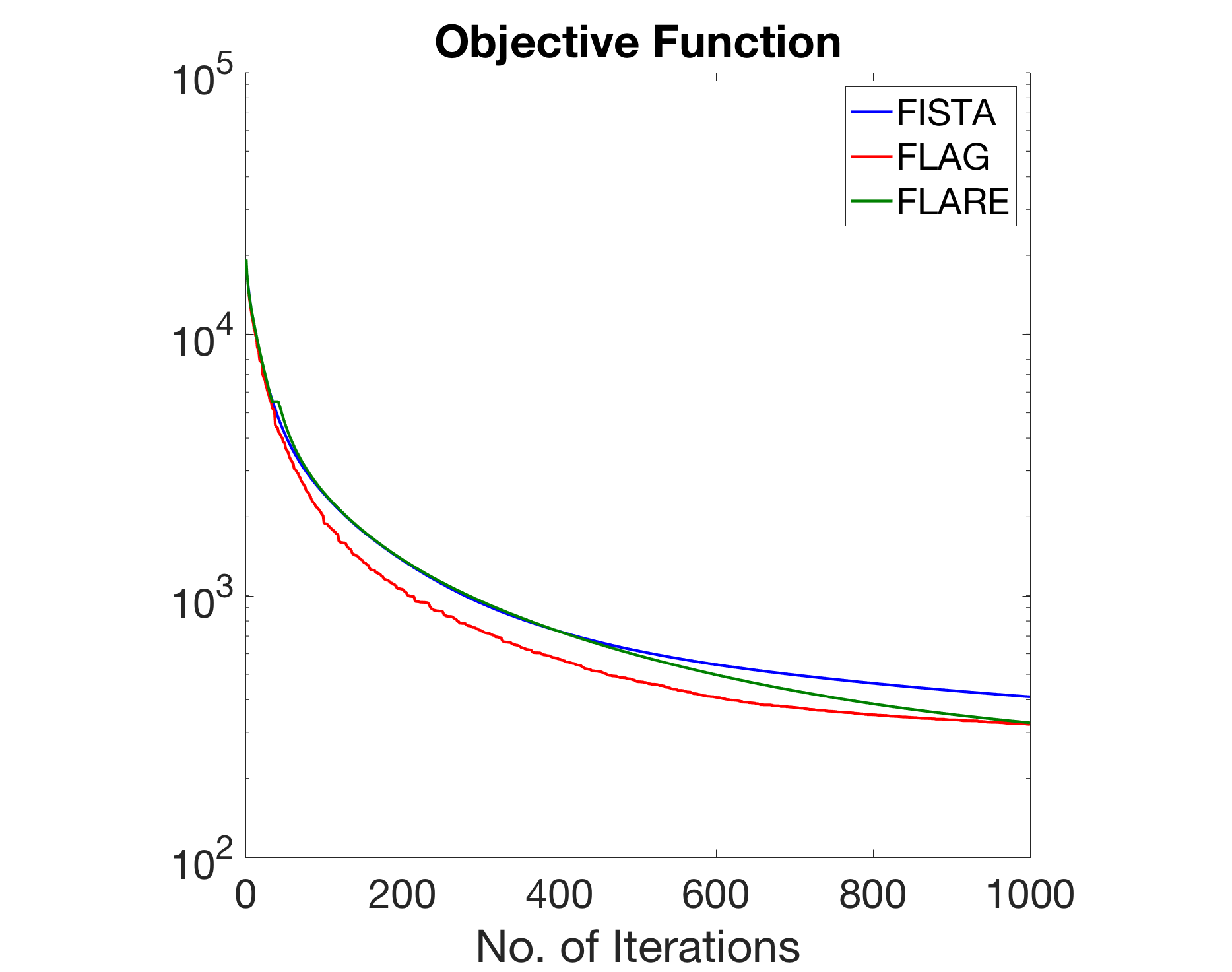}
\includegraphics[width=0.45\textwidth]{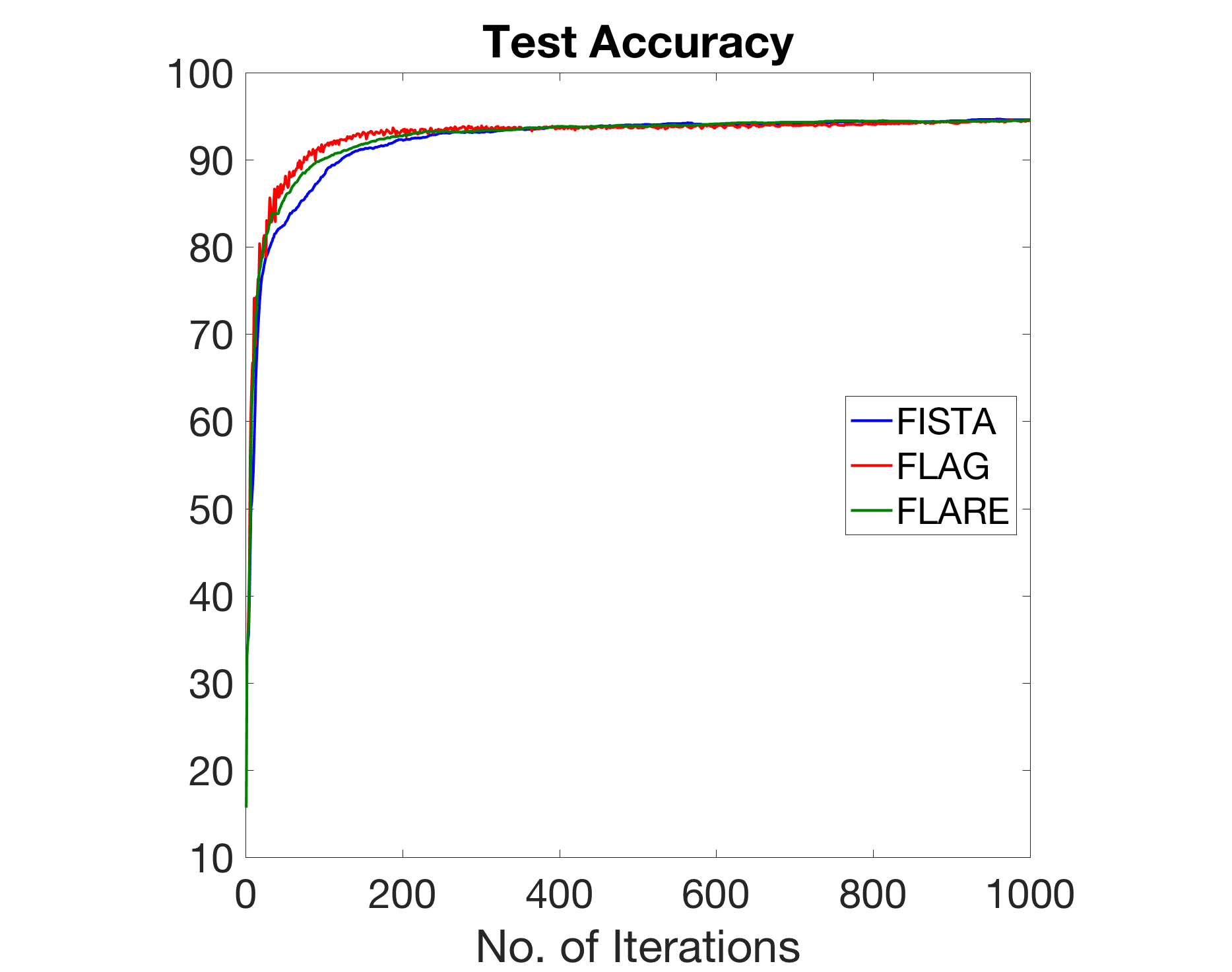}
}
\caption{FLAG, FLARE, and FISTA on box-constrained classification for the \texttt{HARUS} data set.}
\label{fig:HARUSIters}
\end{figure}
%\FloatBarrier

\begin{figure}[!htbp]
\centering
%Gisette
\subfigure{
\includegraphics[width=0.45\textwidth]{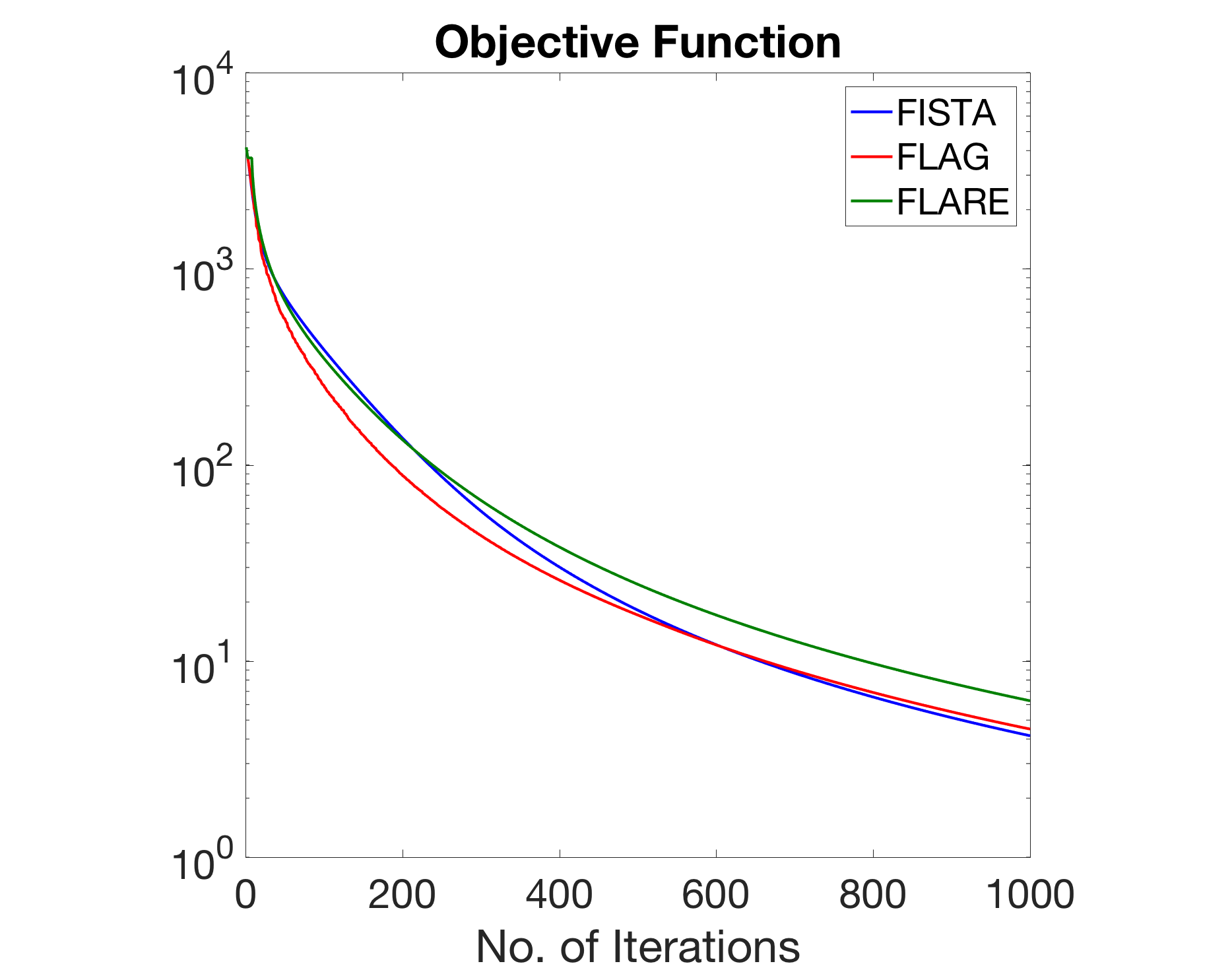}

\includegraphics[width=0.45\textwidth]{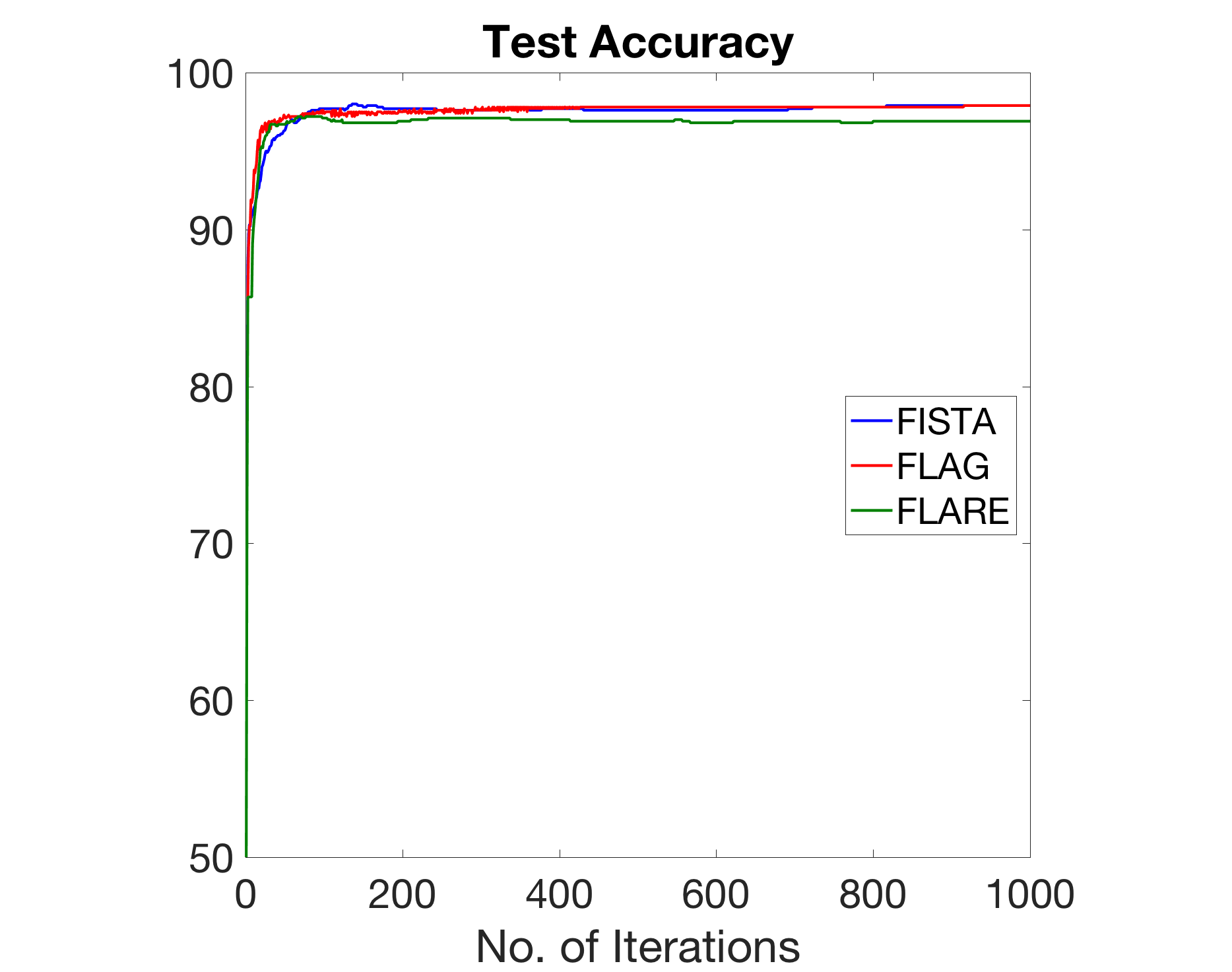}
}
\caption{FLAG, FLARE, and FISTA on $\ell_1$ regularized classification for the \texttt{Gisette} data set.}
\label{fig:GisetteIters}
\end{figure}
%\FloatBarrier

\begin{figure}[!htbp]
\centering
%Forest Covertype
\subfigure{
\includegraphics[width=0.45\textwidth]{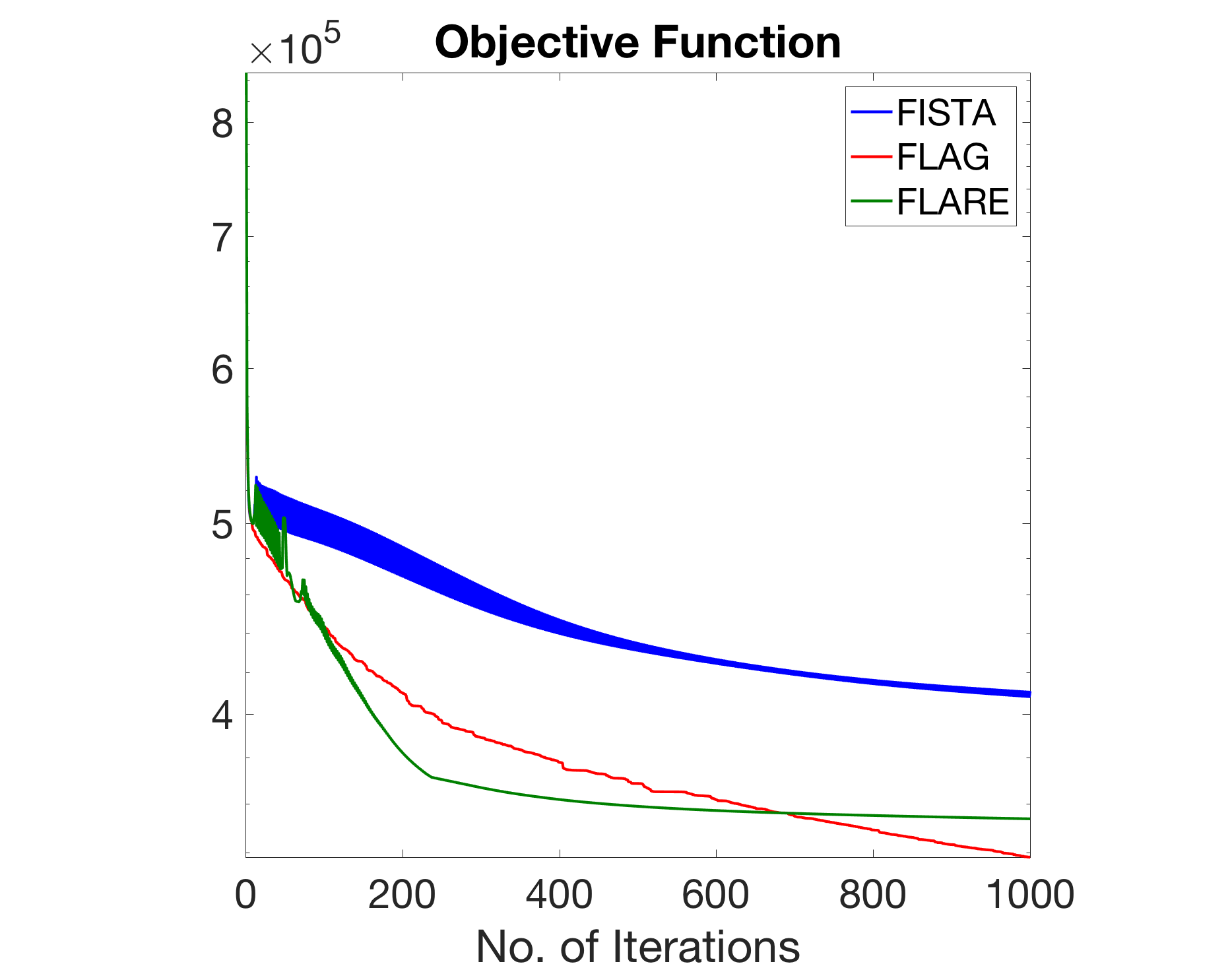}

\includegraphics[width=0.45\textwidth]{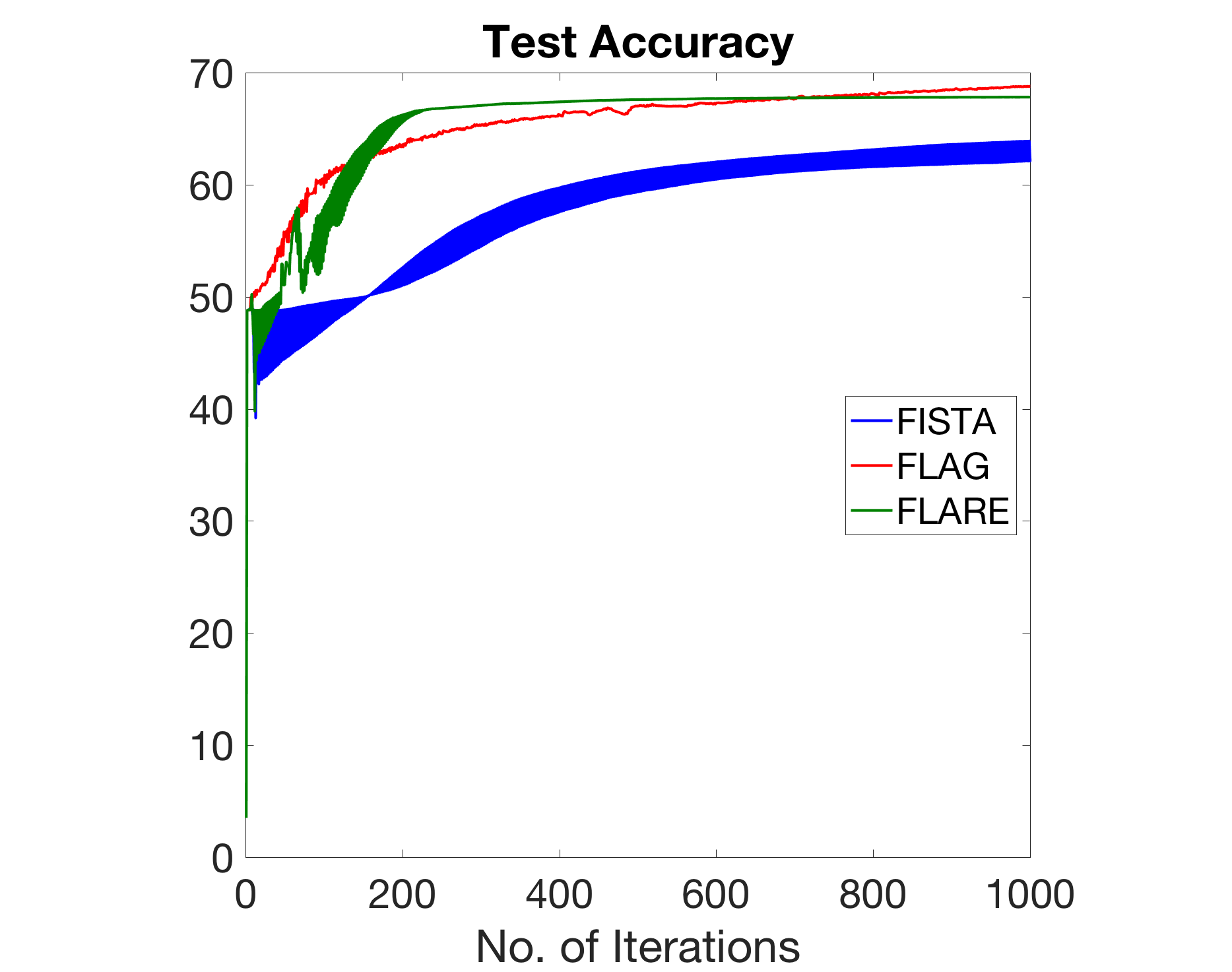}
}

\caption{FLAG, FLARE, and FISTA on $\ell_1$ regularized classification for the \texttt{Forest Covertype} data set.}
\label{fig:CoverTypeIters}
\end{figure}
%\FloatBarrier

We note that on the classification tasks, FLAG and FLARE perform as well as, or better than FISTA, as expected from the theoretical analysis. In particular, on the \texttt{20 Newsgroups} and \texttt{Forest Covertype} data sets, both FLAG and FLARE significantly outperform FISTA. 

For the regression task we also ran FLAG, FLARE, and FISTA for 1000 iterations. The data sets used are enumerated in Table~\ref{table:regression}. The per iteration loss and test error are displayed in Figures~\ref{fig:BlogIters} and~\ref{fig:FBCVDIters}.

\begin{figure}[!htbp]
\centering
\subfigure{
\includegraphics[width=0.45\textwidth]{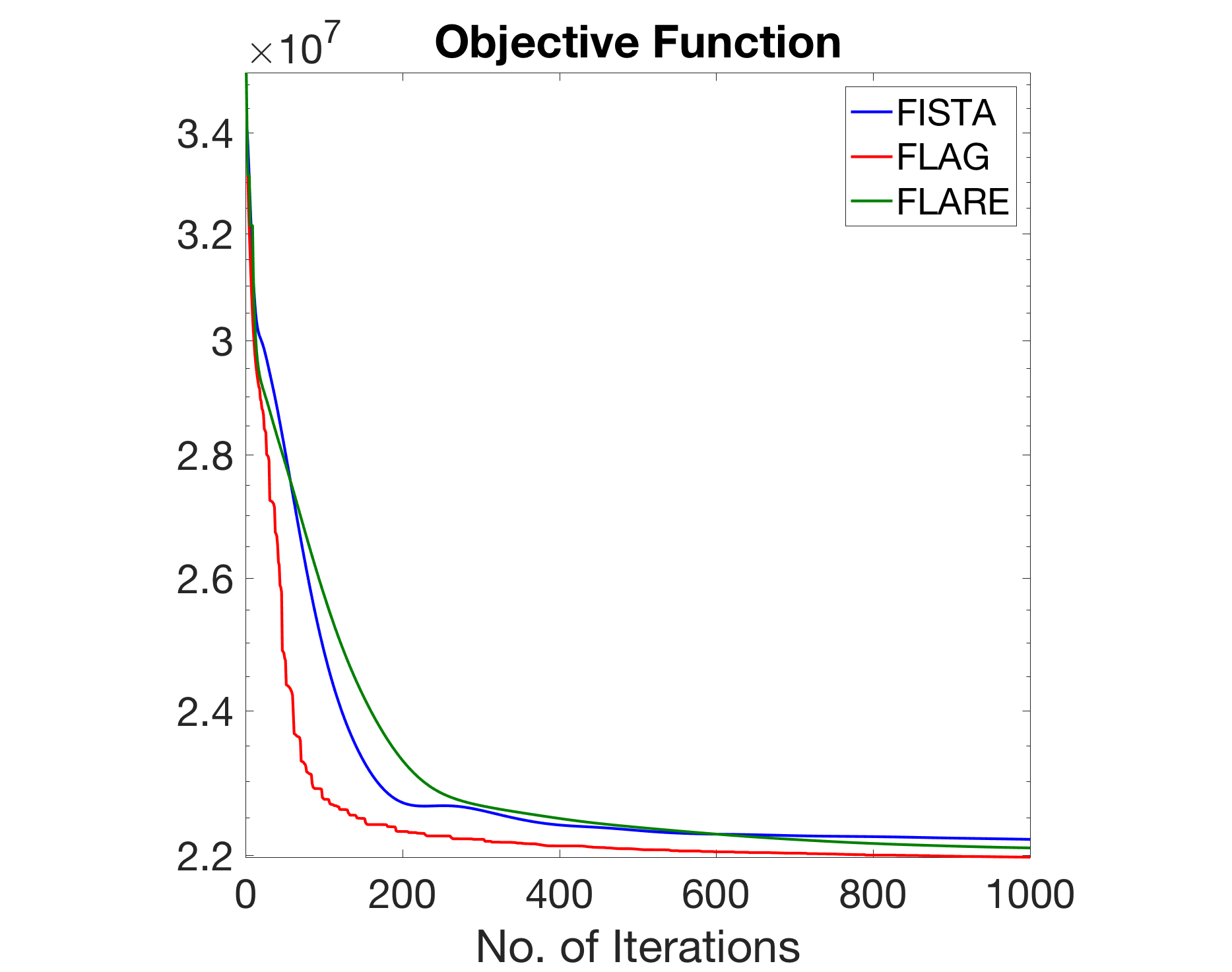}

\includegraphics[width=0.45\textwidth]{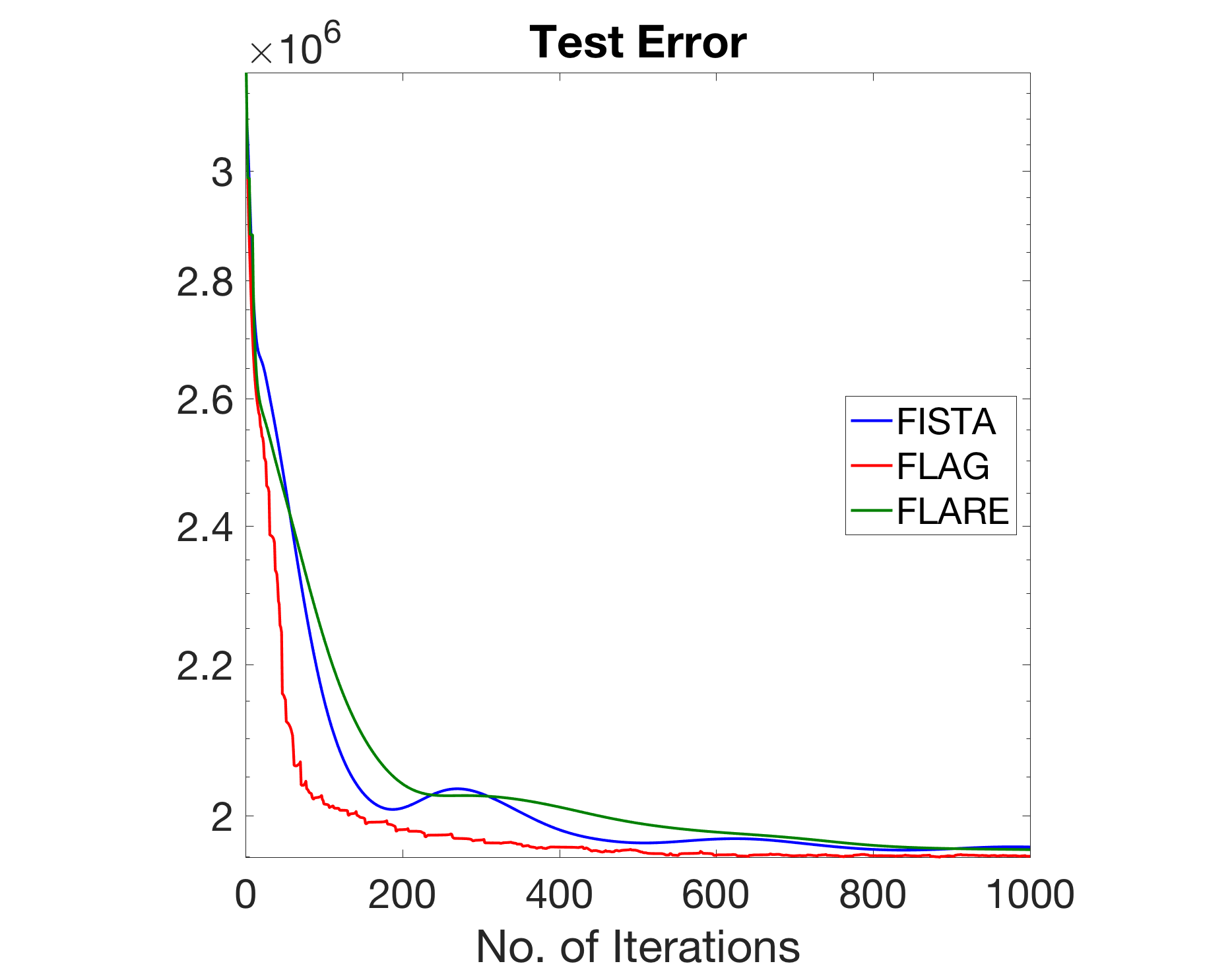}
}
\caption{FLAG, FLARE, and FISTA on box-constrained regression for the \texttt{BlogFeedback} data set.}
\label{fig:BlogIters}
\end{figure}
%\FloatBarrier

\begin{figure}[!htbp]
\centering
\subfigure{
\includegraphics[width=0.45\textwidth]{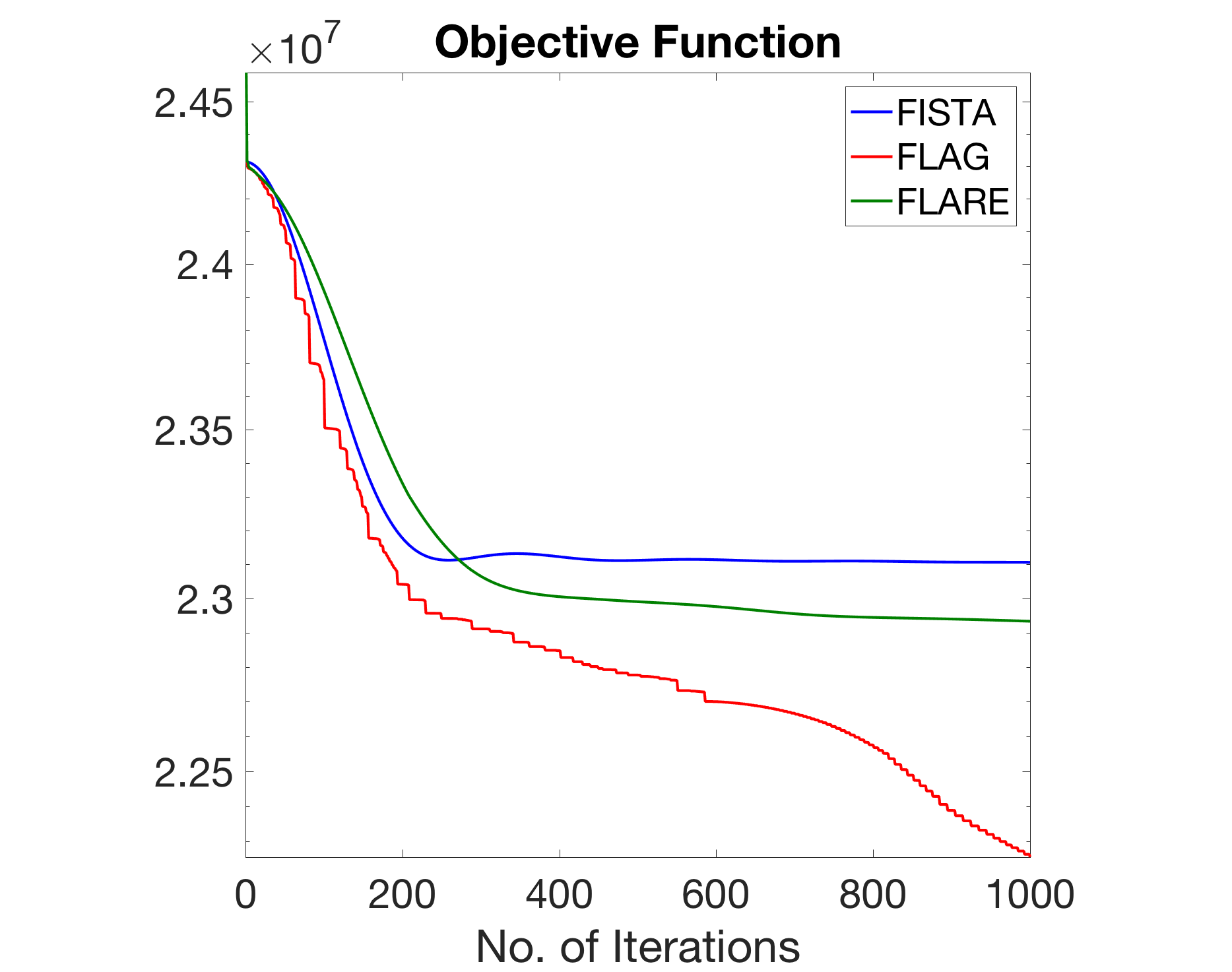}

\includegraphics[width=0.45\textwidth]{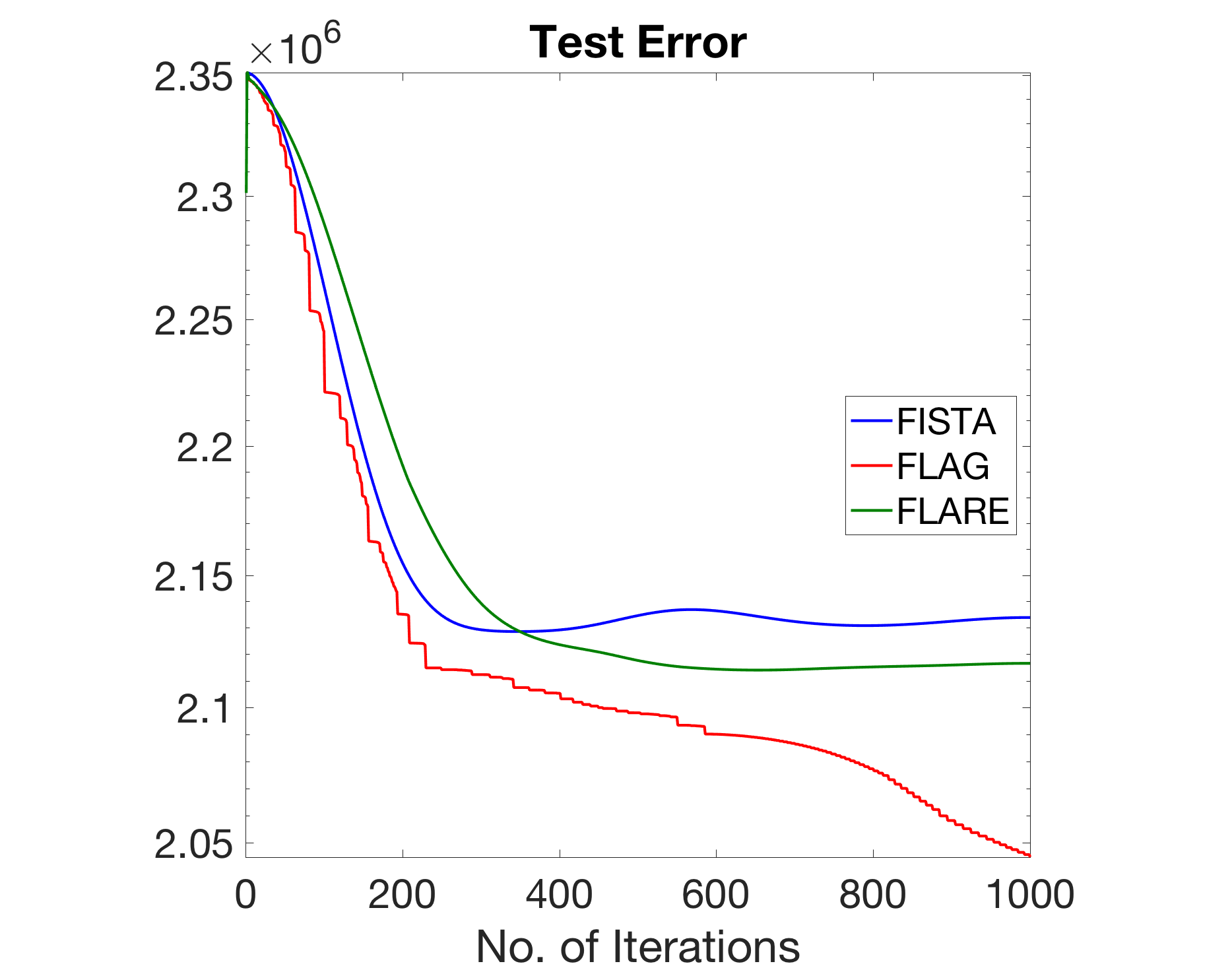}
}
\caption{FLAG, FLARE, and FISTA on $\ell_1$ regularized classification for the \texttt{Facebook CVD} data set.}
\label{fig:FBCVDIters}
\end{figure}
%\FloatBarrier

	Similarly to classification tasks, FLAG and FLARE perform as well as or superior to FISTA. Particularly on the \texttt{Facebook CVD} data set, FLAG significantly outperforms both FISTA and FLARE.
	
    As previously noted, each iteration of FLAG and FLARE can involve more $\prox$ evaluations than FISTA, which can dominate the run time. Thus, comparing the performance of these methods as measured by the number of $\prox$ evaluations is more representative of real world cost than that measured by iterations. We thus repeat the above experiments with the exception that this time we ran each algorithm for 1000 $\prox$ evaluations and tracked the loss and test accuracy versus the number of $\prox$ evaluations. The results of these trials are displayed in Figures~\ref{fig:20NewsProx} --~\ref{fig:FBCVDProx}.
    
It can be seen that, as measured by the number of $\prox$ evaluations, FLARE and FISTA can outperform FLAG due to the possibly significant number of $\prox$ evaluations involved in FLAG's ``BinarySearch'', i.e., Step~\ref{algstep:comb} of Algorithm~\ref{alg:flag}. For all examples, FLARE performs at least as well as FISTA with FLARE outperforming all other algorithms on certain datasets, e.g., Figures~\ref{fig:20NewsProx} and~\ref{fig:CoverTypeProx}. Empirically, after relaxing the ``BinarySearch'' in FLAG, FLARE continues to enjoy the performance advantages afforded by leveraging acceleration and adaptivity, while maintaining the low per-iteration cost of FISTA. 

\begin{figure}[!htbp]
%twenty news
\centering
\subfigure{
\includegraphics[width=0.45\textwidth]{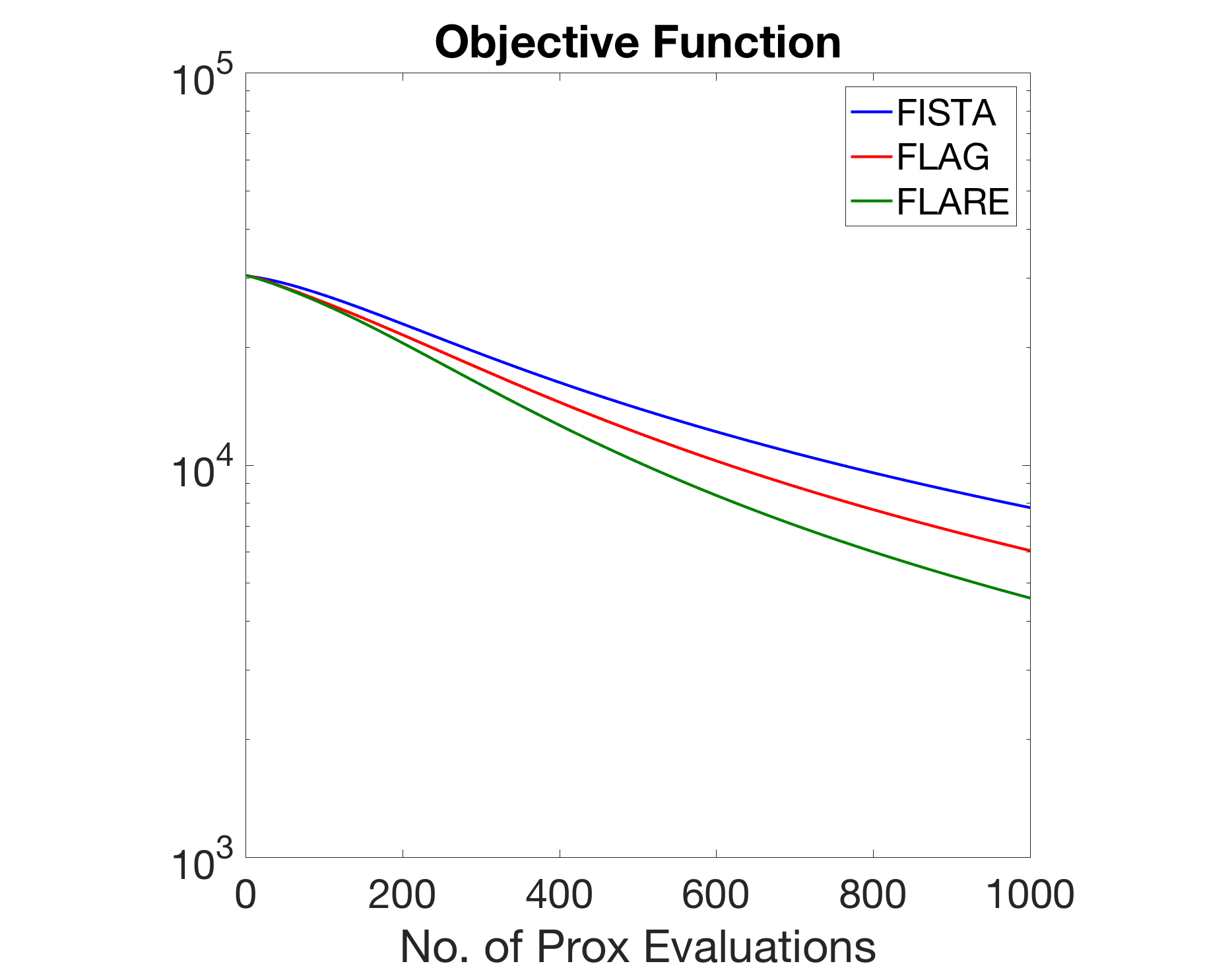}

\includegraphics[width=0.45\textwidth]{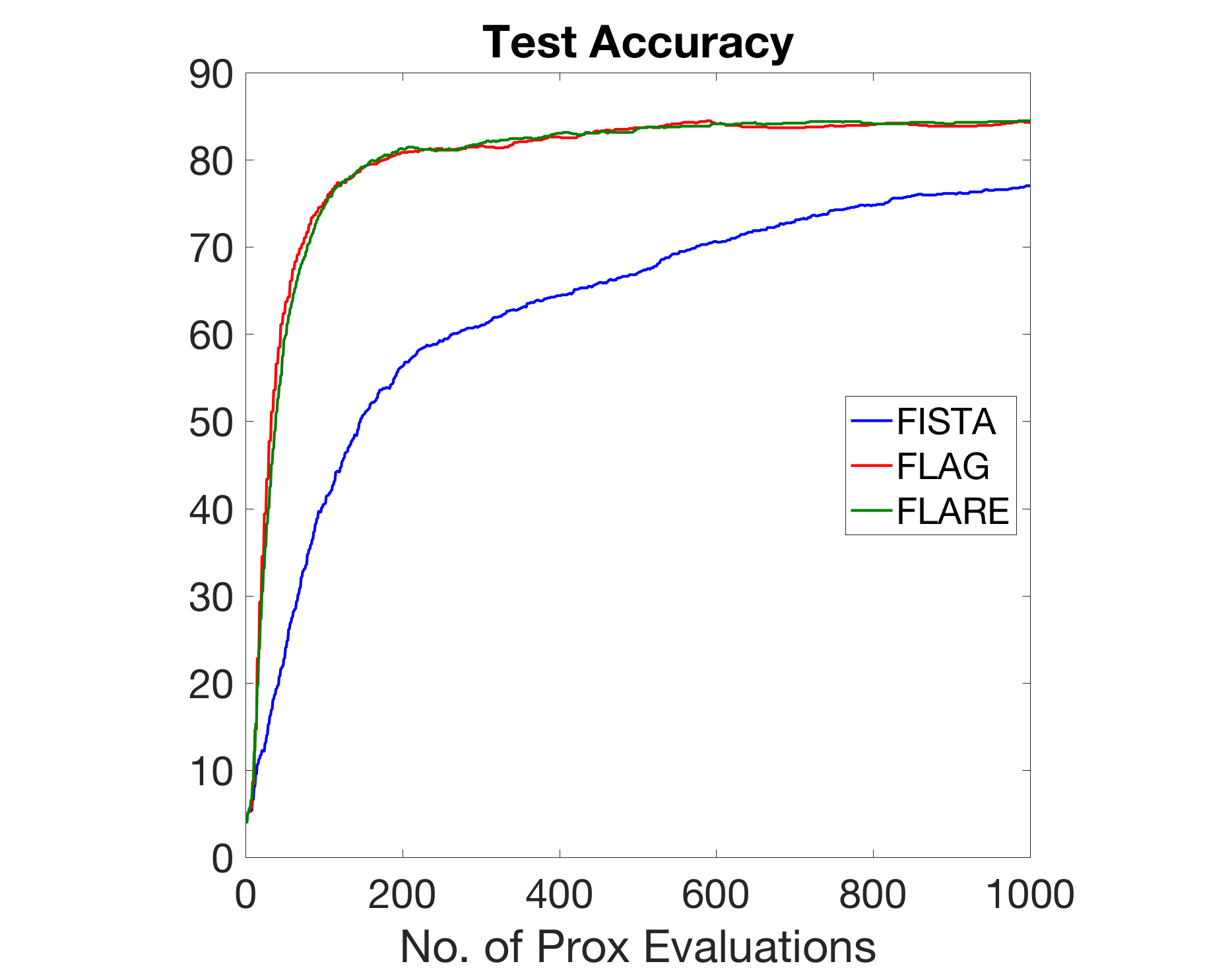}
}
\caption{FLAG, FLARE, and FISTA on box-constrained classification for the \texttt{20 Newsgroups} data set.}
\label{fig:20NewsProx}
\end{figure}
%\FloatBarrier

\begin{figure}[!htbp]
\centering
%HARUS
\subfigure{
\includegraphics[width=0.45\textwidth]{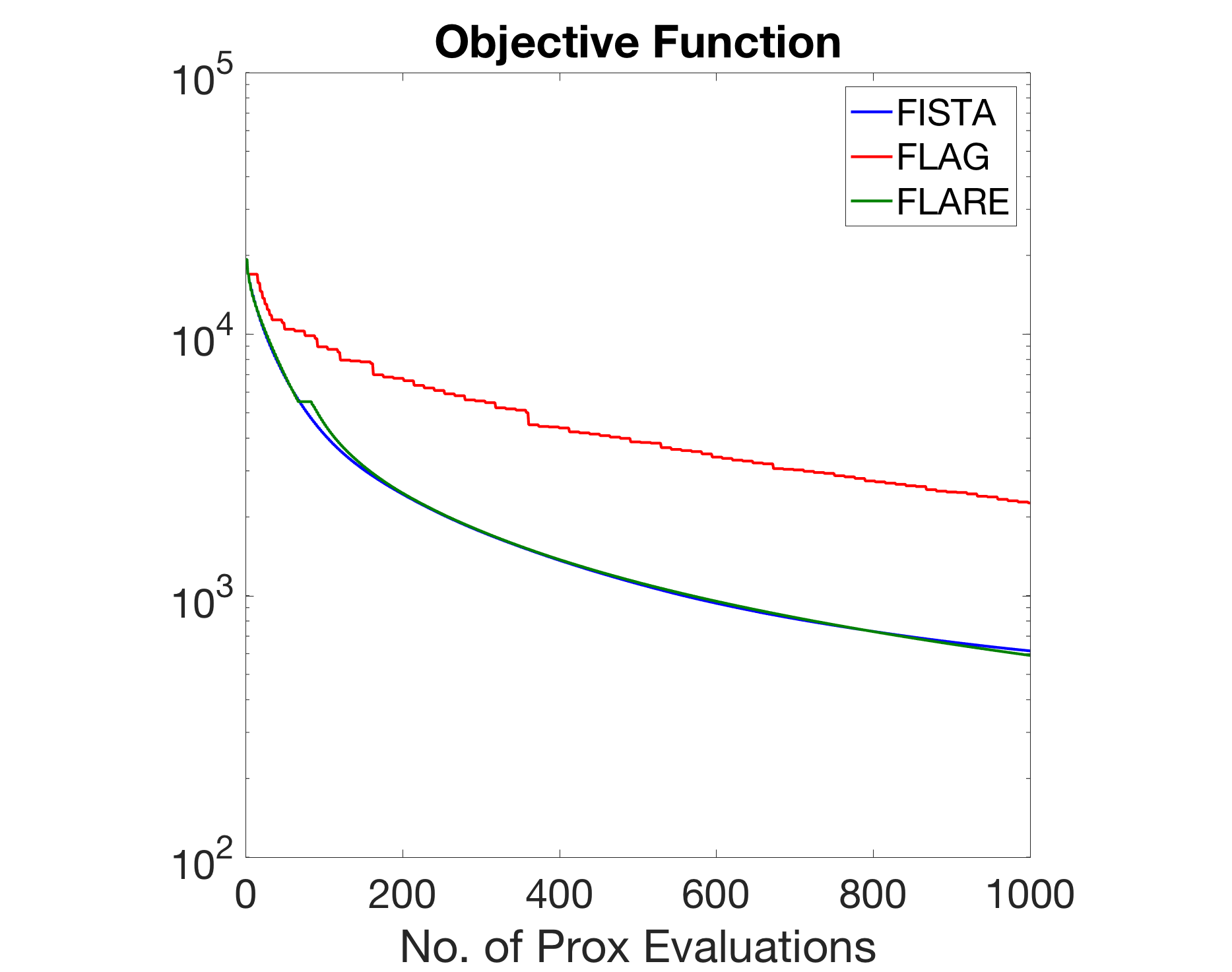}

\includegraphics[width=0.45\textwidth]{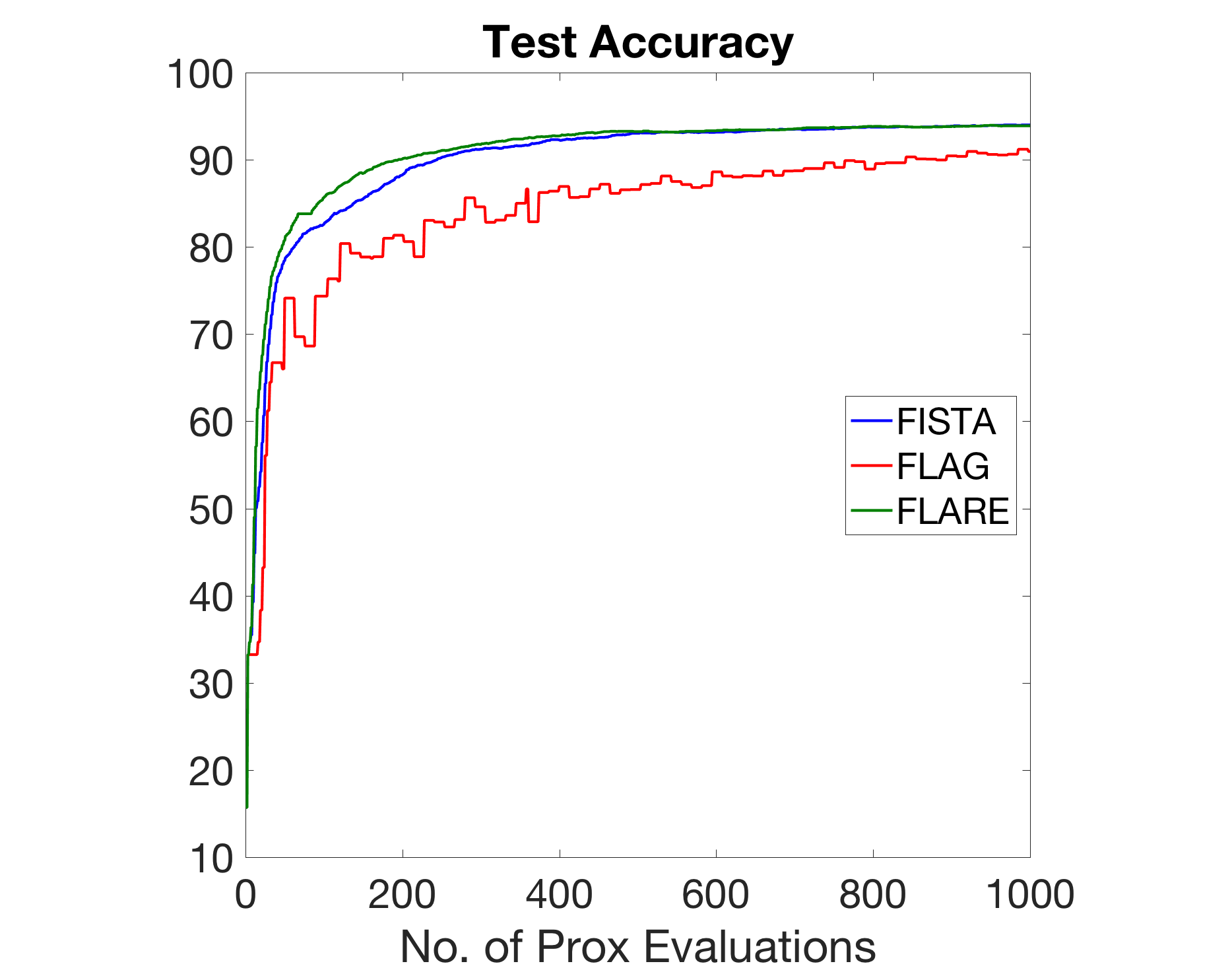}
}
\caption{FLAG, FLARE, and FISTA on box-constrained classification for the \texttt{HARUS} data set.}
\label{fig:HARUSProx}
\end{figure}
%\FloatBarrier

\begin{figure}[!htbp]
\centering
%Gisette
\subfigure{
\includegraphics[width=0.45\textwidth]{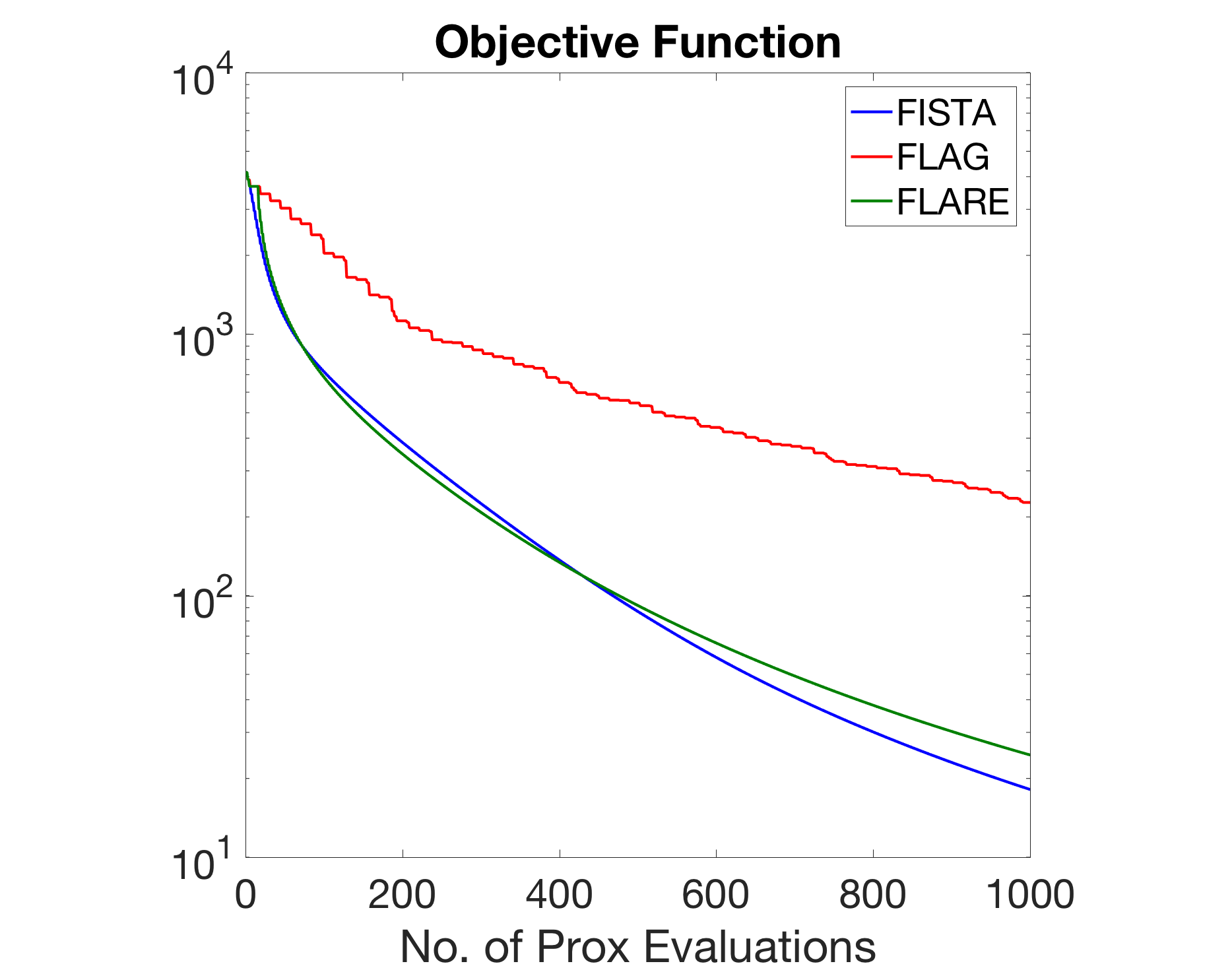}

\includegraphics[width=0.45\textwidth]{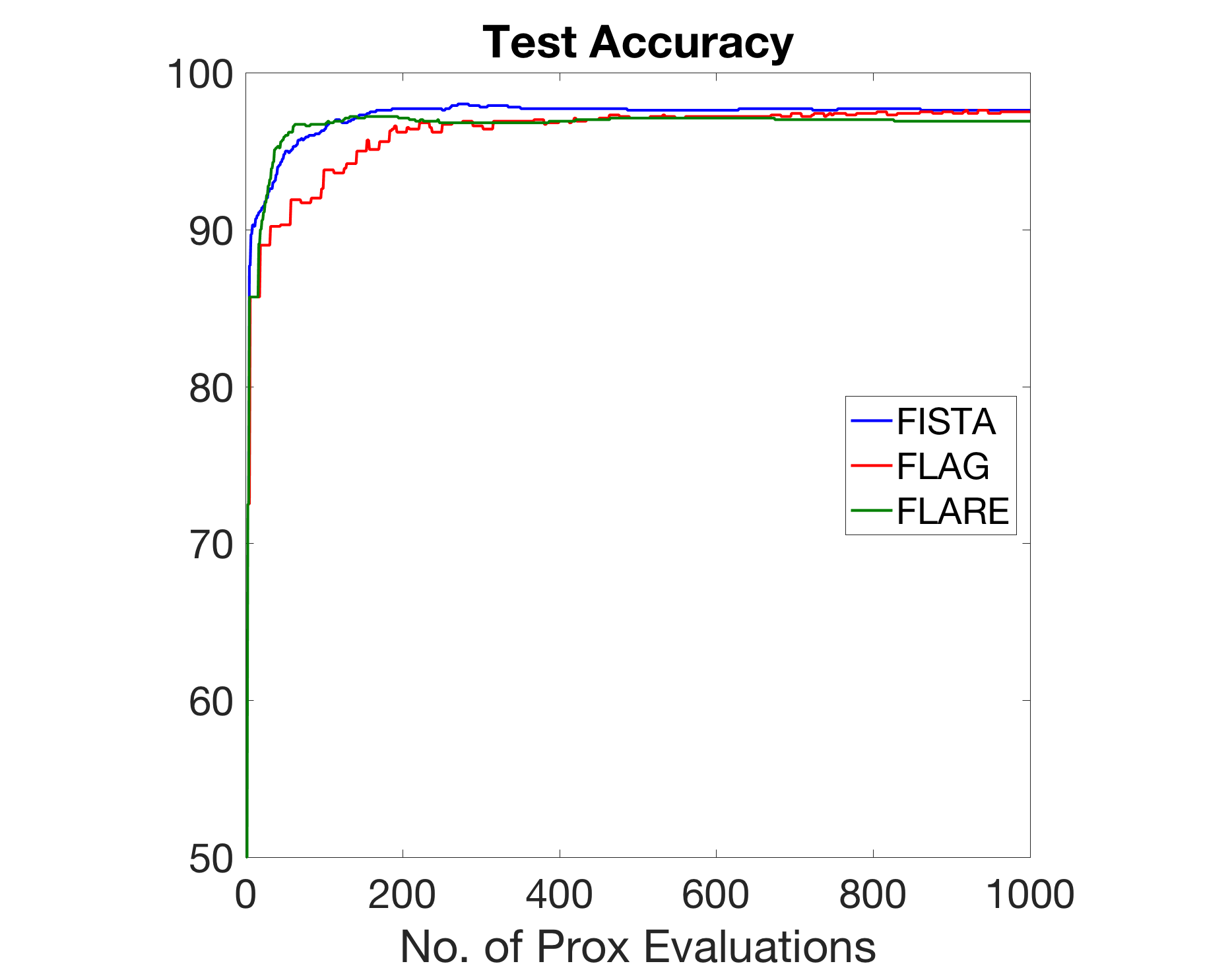}
}
\caption{FLAG, FLARE, and FISTA on $\ell_1$ regularized classification for the \texttt{Gisette} data set.}
\label{fig:GisetteProx}
\end{figure}
%\FloatBarrier

\begin{figure}[!htbp]
%Forest Covertype
\centering
\subfigure{
\includegraphics[width=0.45\textwidth]{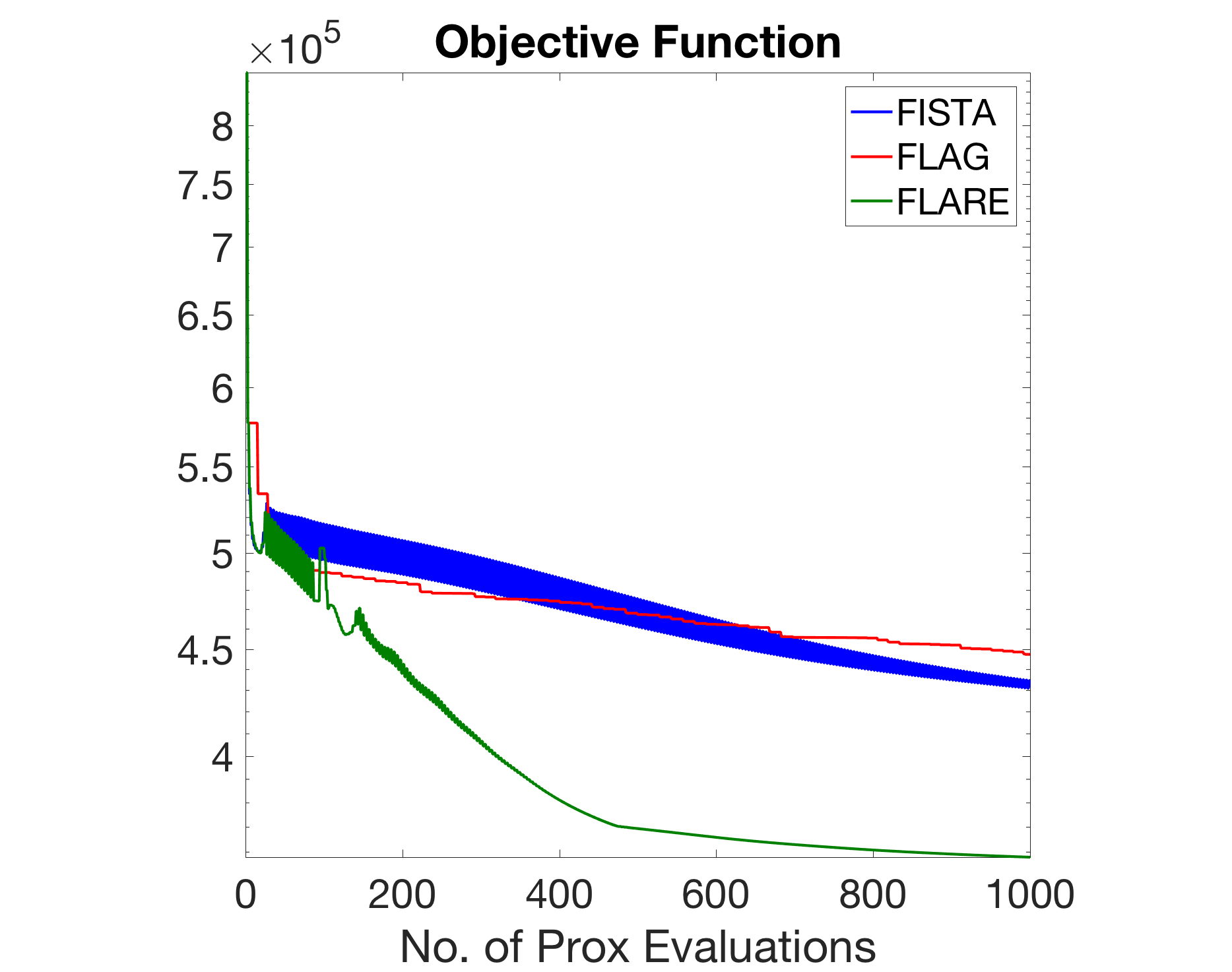}

\includegraphics[width=0.45\textwidth]{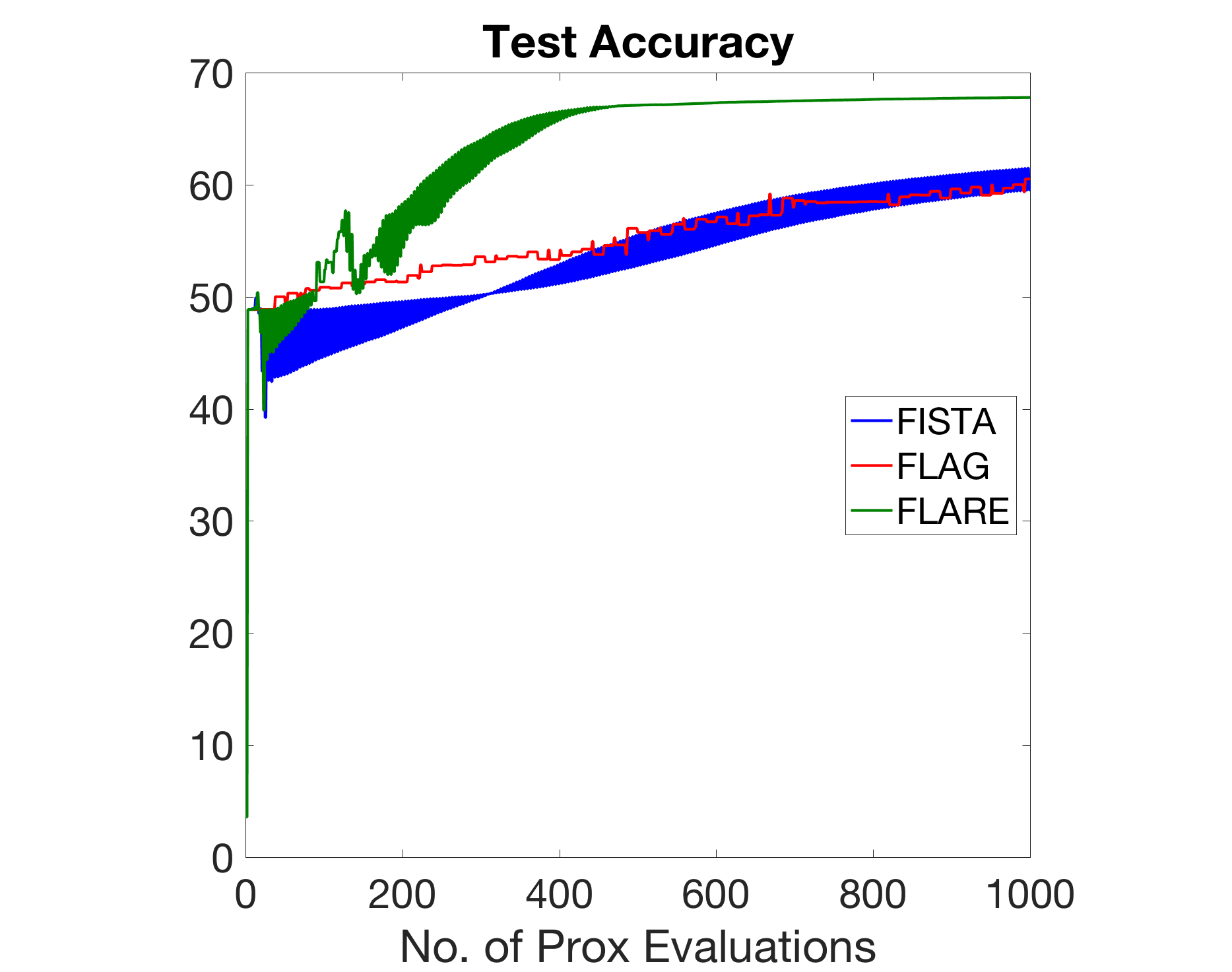}
}
\caption{FLAG, FLARE, and FISTA on $\ell_1$ regularized classification for the \texttt{Forest Covertype} data set.}
\label{fig:CoverTypeProx}
\end{figure}
%\FloatBarrier

\begin{figure}[!htbp]
\centering
\subfigure{
\includegraphics[width=0.45\textwidth]{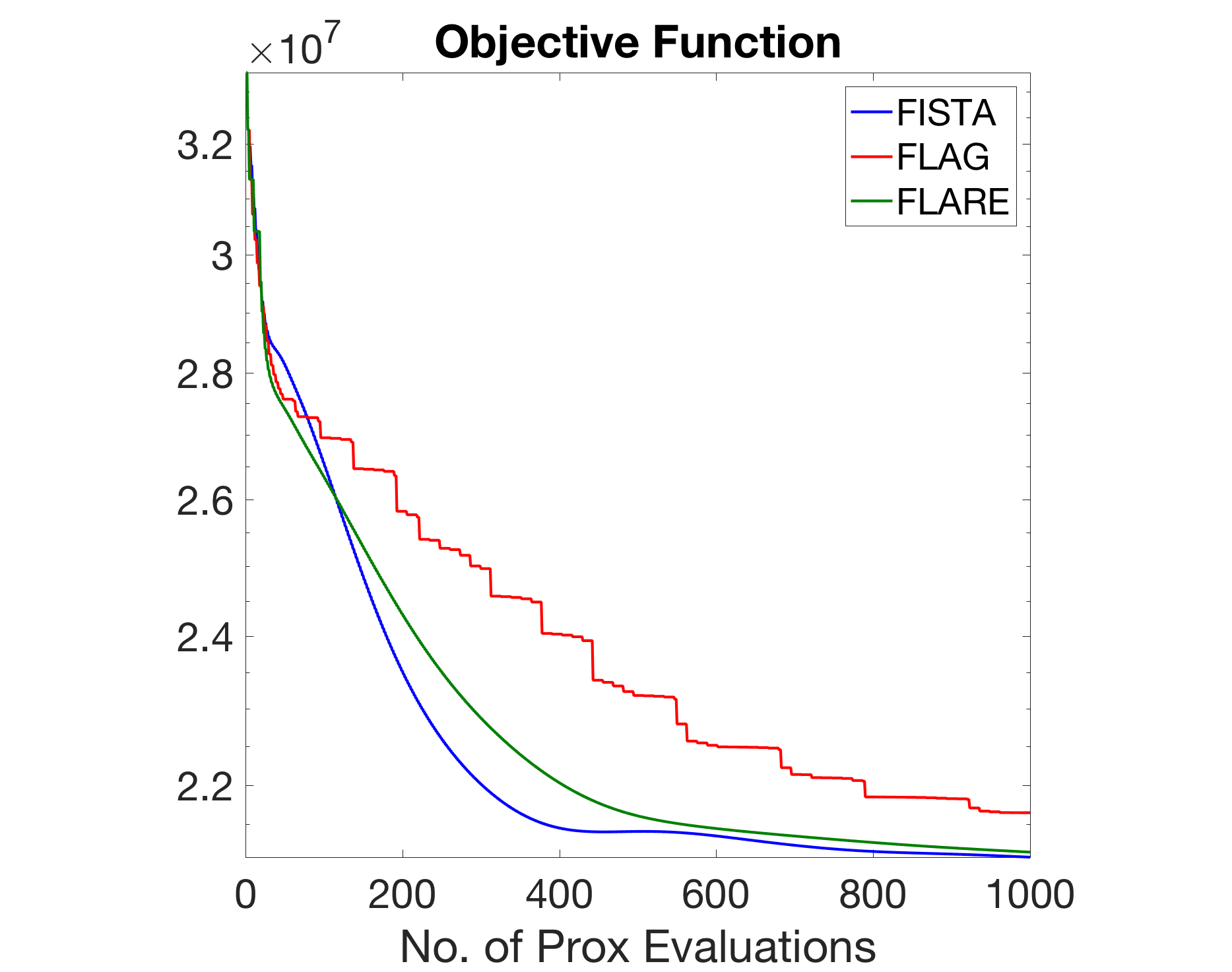}\label{fig: fb a}

\includegraphics[width=0.45\textwidth]{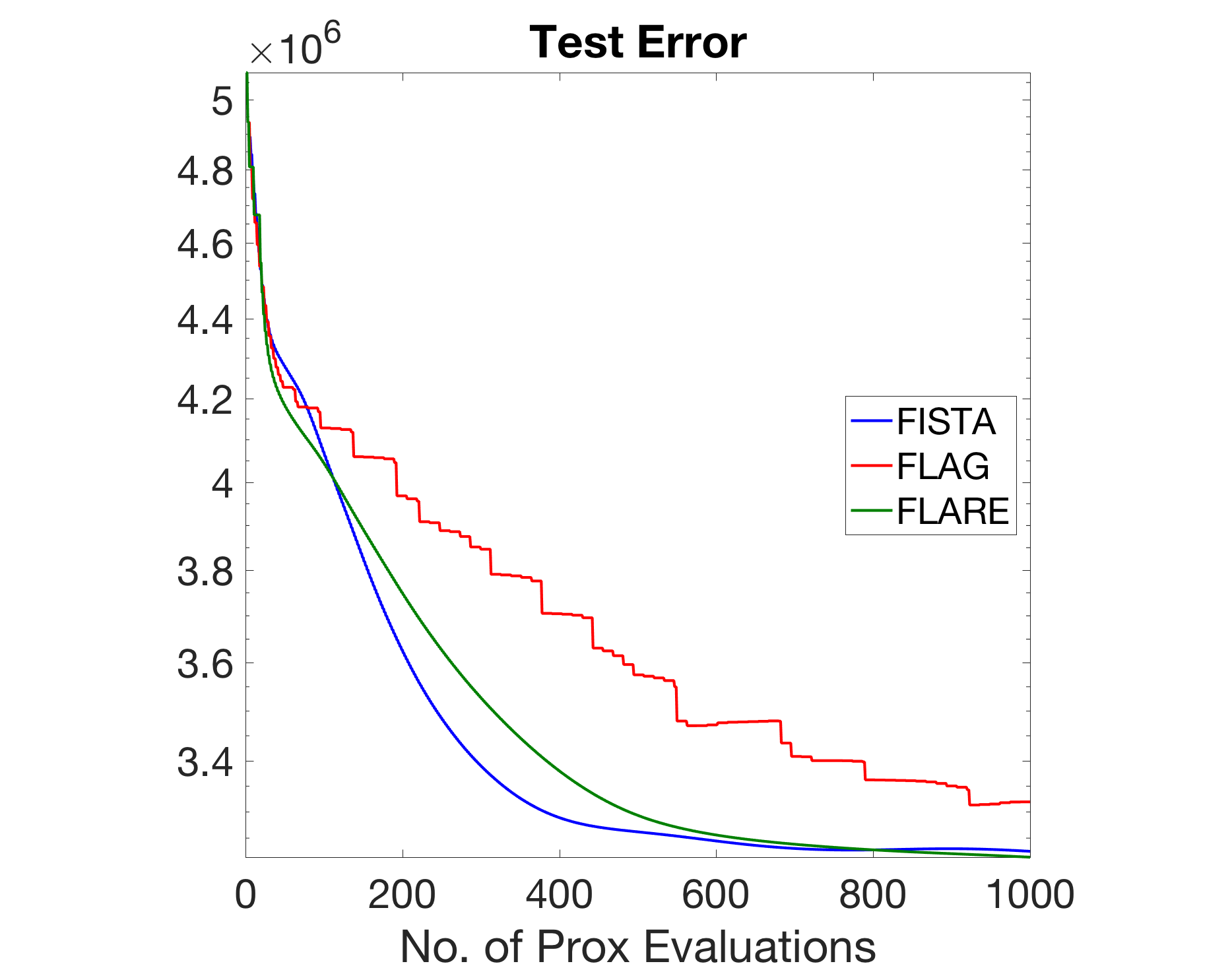}\label{fig: fb b}
}
\caption{FLAG, FLARE, and FISTA on box-constrained regression for the \texttt{BlogFeedback} data set.}
\label{fig:BlogProx}
\end{figure}
%\FloatBarrier

\begin{figure}[!htbp]
\centering
\subfigure{
\includegraphics[width=0.45\textwidth]{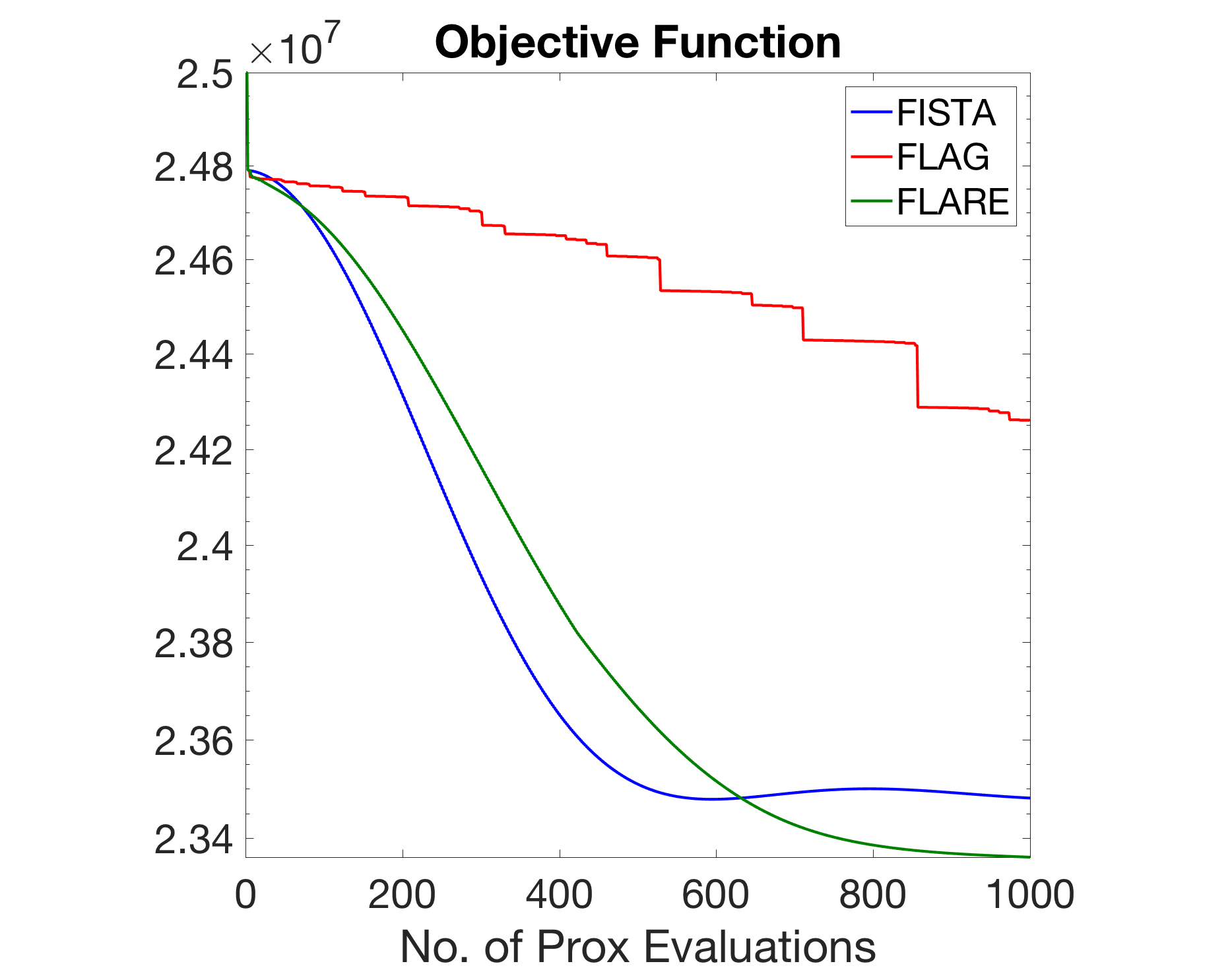}\label{fig: fb c}

\includegraphics[width=0.45\textwidth]{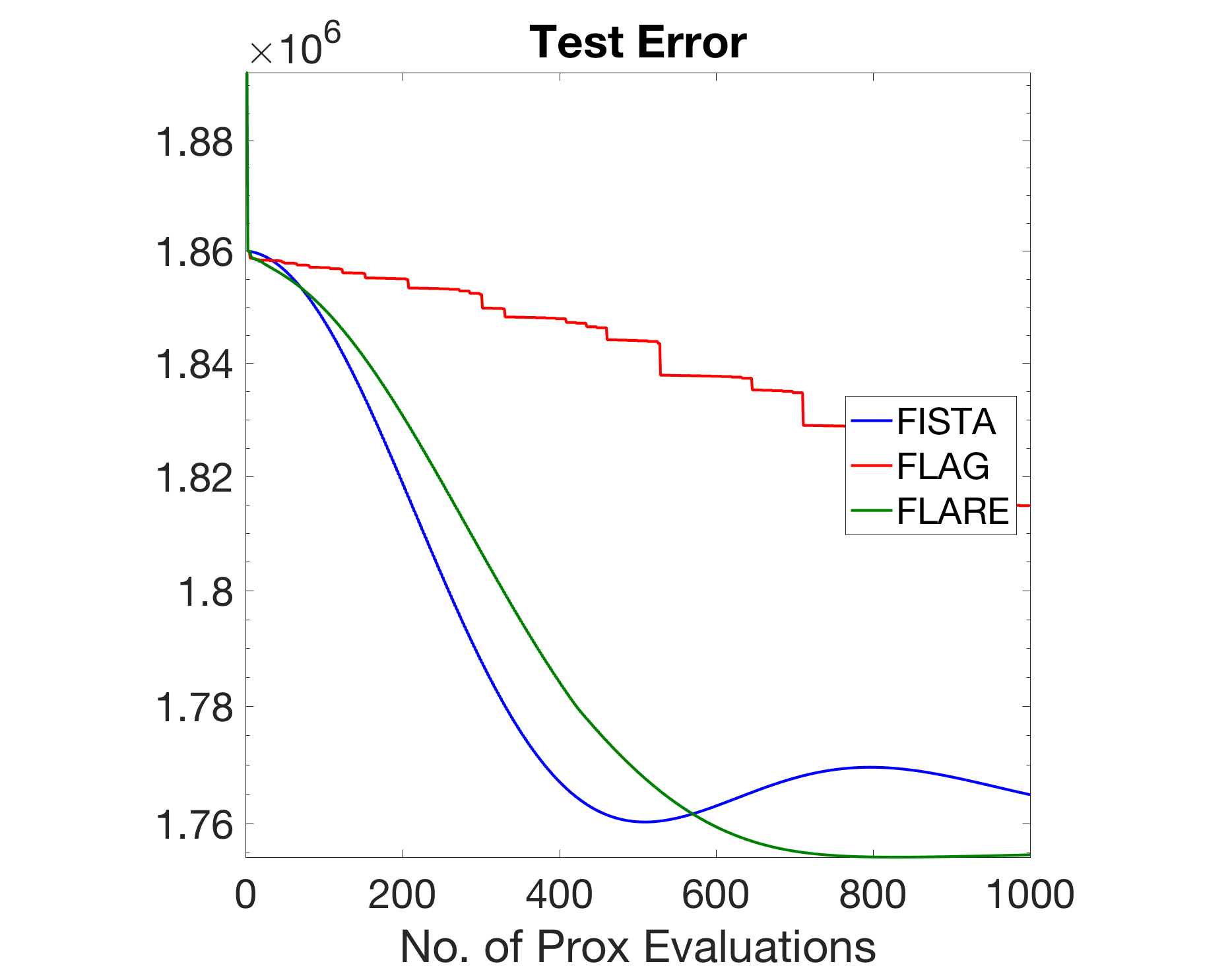}\label{fig: fb d}
}
\caption{FLAG, FLARE, and FISTA on $\ell_1$ regularized regression for the \texttt{Facebook CVD} data set.}
\label{fig:FBCVDProx}
\end{figure}
%\FloatBarrier

\section{Conclusions}
\label{sec:conclusions}
Following the advantages of employing acceleration, e.g., Nesterov's scheme, as well as adaptivity, e.g., Adagrad, here, we considered algorithms that can offer the best of both worlds. Specifically, in the context of composite optimization problem, we theoretically as well as empirically studied FLAG and its relaxation, FLARE, which can achieve this by a particular linear coupling of a simple gradient step with that of a properly scaled mirror update.

We showed that FLAG and FLARE can be equivalently regarded as adaptive versions of FISTA or alternatively, as accelerated versions of AdaGrad. In other words, like Nesterov's accelerated algorithm and its proximal variant, FISTA, our methods achieve the optimal convergence rate of $\mathcal{O}(1/k^2)$ and like AdaGrad our methods adaptively choose a regularizer, in a way that performs almost as well as the best choice of regularizer in hindsight. These two effects contribute to the improved theoretical properties and empirical performance of FLAG and FLARE compared to alternatives, e.g., FISTA. 

Recently, there has been some interesting development in studying the continuous-time limit of acceleration algorithms, e.g.,~\cite{su2014differential,krichene2015accelerated,wibisono2016variational}. In this light, studying adaptive regularization in the continuous time setting is an interesting direction for research, which we intend to pursue.
\section{Acknowledgment}
We would like to acknowledge ARO, DARPA, and NSF for providing partial support of
this work. We gratefully acknowledge the support of the NSF through grant IIS-1619362 and of the Australian Research Council through an Australian Laureate Fellowship (FL110100281) and through the Australian Research Council Centre of Excellence for Mathematical and Statistical Frontiers (ACEMS).

\bibliographystyle{plain} 
\bibliography{references}

\clearpage
\setboolean{@twoside}{false}
\appendix
\section{Proofs}
\label{sec:appendix}
We now give the details for the proof of our main results, i.e., Theorems~\ref{thm:final_flag} and~\ref{thm:final_flare}. Below, we outline the steps for the proof of FLAG's Theorem~\ref{thm:final_flag}. The proof of Theorem~\ref{thm:final_flare} for FLARE follows the same line of reasoning. Also, we note that, in what follows, lemmas/corollaries required for the proof of Theorem~\ref{thm:final_flare}, are given immediately after those of FLAG. 
\begin{enumerate}
	\item FLAG is essentially a combination of mirror descent and proximal gradient descent steps (Lemmas~\ref{lemma:grad_map} and~\ref{lem:mir}).
	
	\item $L_k$ in Algorithm~\ref{alg:flag} plays the role of an ``effective gradient Lipschitz constant'' in each iteration. The convergence rate of FLAG ultimately depends on $\sum_{k=1}^T L_k = L \sum_{k=1}^T \vg_{k}^{T} S_{k}^{-1} \vg_{k}$. (Lemma~\ref{lem:sum_bound} and Corollary~\ref{cor:maintheorem})
	
	\item By picking $S_{k}$ adaptively like in AdaGrad, we achieve a non-trivial upper bound for $\sum_{k=1}^T L_k$. (Lemma~\ref{lem:adareg})
	
	\item FLAG relies on picking an $\xx_k$ at each iteration that satisfies an inequality involving $L_k$ (Corollary \ref{cor:proxavg}). However, because $L_k$ is not known prior to picking $\xx_k$, we must choose an $\xx_k$ to roughly satisfy the inequality for all possible values of $L_k$. We do this by picking $\xx_k$ using binary search. (Lemmas~\ref{lem:proxcont} and~\ref{lem:binsearch}  and Corollary~\ref{cor:proxavg})
	
	\item Finally, we need to pick the right stepsize for each iteration. Our scheme is very similar to the one used in \cite{allen2014linear}, but generalized to handle a different $L_k$ each iteration. (Lemmas~\ref{lem:eta_k} and~\ref{lem:sum_bound} as well as Corollary \ref{cor:maintheorem}).
	
	\item Theorem~\ref{thm:mastertheorem} combines items 1, 2 and 4, above.  Finally, to prove Theorem~\ref{thm:final_flag}, we combine Theorem~\ref{thm:mastertheorem} with items 3 and 5 above.
\end{enumerate}

\subsection{Proof of Theorem~\ref{thm:final_flag} and Theorem \ref{thm:final_flare}}
\label{sec:proof_flag}
First, we obtain the following key result (similar to~\cite[Lemma 2.3]{beck2009fast}) regarding the vector
$\pp = -L(\prox(\xx) - \xx)$, as in Step~\ref{alg_step_grad_map} of FLAG, which is known as the
\textit{Gradient Mapping} of $F$ on $\C$. 

\begin{lemma}[Gradient Mapping]
	\label{lemma:grad_map}
	For any $\xx, \yy \in \C$, we have
    \begin{align*}
	F(\prox(\xx)) &\leq F(\yy)  + \lin{L (\prox(\xx) - \xx), \yy - \xx} \\
    &\quad - \frac{L}{2} \| \xx - \prox(\xx)\|_{2}^{2},
    \end{align*}
	where $\prox(\xx)$ is defined as in~\eqref{eq:prox}. In particular, $F(\prox(\xx)) \leq F(\xx)  - \frac{L}{2} \| \xx - \prox(\xx)\|_{2}^{2}$. 
\end{lemma}
\begin{proof}[Proof of Lemma \ref{lemma:grad_map}]
	This result is the same as Lemma 2.3 in~\cite{beck2009fast}. We bring its proof here for completeness.
	
	For any $\yy \in \C$, any sub-gradient, $\vv$, of $h$ at $\prox(\xx)$, i.e., $\vv \in \partial h(\prox(\xx))$, and by optimality of $\prox(\xx)$ in~\eqref{eq:prox}, we  have
	\begin{align*}
	0 & \leq \lin{\nabla f(\xx) + \vv + L (\prox(\xx) - \xx), \yy - \prox(\xx)} \\
	& = \lin{\nabla f(\xx) + \vv + L (\prox(\xx) - \xx), \yy - \xx} + \lin{\nabla f(\xx) \\
    & \quad + \vv + L (\prox(\xx) - \xx), \xx - \prox(\xx)},
	\end{align*}
	and so
	\begin{align*}
	&\lin{\nabla f(\xx) , \prox(\xx) - \xx} \\
    &\leq \lin{\nabla f(\xx) + \vv + L (\prox(\xx) - \xx), \yy - \xx} \\
	& \quad + \lin{ \vv, \xx - \prox(\xx)} - L \| \xx - \prox(\xx)\|_{2}^{2},
	\end{align*}
	Now from $L$-Lipschitz continuity of $\nabla f$ as well as convexity of $f$ and $h$, we get
	\begin{align*}
	&F(\prox(\xx)) \\
    &= f(\prox(\xx)) + h(\prox(\xx)) \\
	& \leq f(\xx)  + \lin{\nabla f(\xx), \prox(\xx)  -  \xx } \\
    &\quad + \frac{L}{2} \| \prox(\xx) - \xx \|_{2}^2 + h(\prox(\xx)) \\
	& \leq f(\xx)  + \lin{\nabla f(\xx) + \vv + L (\prox(\xx) - \xx), \yy - \xx} \\
	& \quad + \lin{ \vv, \xx - \prox(\xx)} - \frac{L}{2} \| \xx - \prox(\xx)\|_{2}^{2} \\
    &\quad +  h(\prox(\xx)) \\
	& \leq f(\yy)  + \lin{\vv + L (\prox(\xx) - \xx), \yy - \xx} \\
	& \quad + \lin{ \vv, \xx - \prox(\xx)} - \frac{L}{2} \| \xx - \prox(\xx)\|_{2}^{2}\\
    &\quad  +  h(\prox(\xx)) \\
	& = f(\yy)  + \lin{L (\prox(\xx) - \xx), \yy - \xx} \\
	& \quad + \lin{\vv, \yy - \prox(\xx)} - \frac{L}{2} \| \xx - \prox(\xx)\|_{2}^{2}  \\
    &\quad +  h(\prox(\xx))\\
	& \leq F(\yy)  + \lin{L (\prox(\xx) - \xx), \yy - \xx} \\
    &\quad - \frac{L}{2} \| \xx - \prox(\xx)\|_{2}^{2}.
	\end{align*}
	\qed
\end{proof}

The following lemma establishes the Lipschitz continuity of the $\prox$ operator.
\begin{lemma}[Prox Operator Continuity]
	\label{lem:proxcont}
	$\prox:\reals^{d} \rightarrow \reals^{d}$ is a $2$-Lipschitz continuous, that is, for any $\xx, \yy \in \C$, we have
	$$\|\prox (\xx) - \prox (\yy)\|_{2}\leq 2 \|\xx-\yy\|_{2}.$$
\end{lemma}

\begin{proof}[Proof of Lemma \ref{lem:proxcont}]
	By Definition~\eqref{eq:prox}, for any $\xx, \yy, \zz, \zz' \in \C$,  $\vv \in \partial h(\prox(\xx))$, and $\ww \in \partial h(\prox(\yy))$, we have
	%\begin{align*}
	%&\lin{\nabla f(\xx) + \vv + L (\prox(\xx) - \xx), \zz - \prox(\xx)} \geq 0, \\
	%&\lin{\nabla f(\yy) + \ww + L (\prox(\yy) - \yy), \zz' - \prox(\yy)} \geq 0, \\
	%\end{align*}
	%which give
	\begin{align*}
	&\lin{\vv, \zz - \prox(\xx)} \\
    &\quad \geq -\lin{\nabla f(\xx) + L (\prox(\xx) - \xx), \zz - \prox(\xx)}, \\
	&\lin{\ww, \zz' - \prox(\yy)} \\
    &\quad \geq -\lin{\nabla f(\yy) + L (\prox(\yy) - \yy), \zz' - \prox(\yy)}.
	\end{align*}
	In particular, for $\zz = \prox(\yy)$ and $\zz' = \prox(\zz)$, we get
	\begin{align*}
	&\lin{\vv, \prox(\yy) - \prox(\xx)} \\
    &\quad \geq -\lin{\nabla f(\xx) + L (\prox(\xx) - \xx), \prox(\yy) - \prox(\xx)}, \\
	&\lin{\ww, \prox(\yy) - \prox(\xx)} \\
    &\quad \leq \lin{\nabla f(\yy) + L (\prox(\yy) - \yy), \prox(\xx) - \prox(\yy)}.
	\end{align*}
	
	By monotonicity of sub-gradient, we get
	%\begin{align*}
	%&h(\prox(\yy)) \geq h(\prox(\xx)) + \lin{\vv, \prox(\yy) - \prox(\xx)},\\
	%&h(\prox(\xx)) \geq h(\prox(\yy)) + \lin{\ww, \prox(\xx) - \prox(\yy)},
	%\end{align*}
	%which upon adding the last two inequalities, we get
	\begin{align*}
	\lin{\vv, \prox(\yy) - \prox(\xx)} \leq \lin{\ww, \prox(\yy) - \prox(\xx)}.
	\end{align*}
	So
	\begin{align*}
	&\lin{\nabla f(\xx) + L (\prox(\xx) - \xx), \prox(\xx) - \prox(\yy)} \\
    &\leq \lin{\nabla f(\yy) + L (\prox(\yy) - \yy), \prox(\xx) - \prox(\yy)},
	\end{align*}
	and as a result
	\begin{align*}
	&\lin{\nabla f(\xx) + L (\prox(\xx) - \xx), \prox(\xx) - \prox(\yy)}\\
	& = \langle \nabla f(\xx) + L \left(\prox(\xx) - \prox(\yy) + \prox(\yy)- \xx\right)\\
    &\quad , \prox(\xx) - \prox(\yy)\rangle\\
	& = L \|\prox(\xx) - \prox(\yy)  \|_{2}^{2} \\
    &\quad + \lin{\nabla f(\xx) + L (\prox(\yy)- \xx), \prox(\xx) - \prox(\yy)} \\
	& \leq \lin{\nabla f(\yy) + L (\prox(\yy) - \yy), \prox(\xx) - \prox(\yy)},
	\end{align*}
	which gives
	\begin{align*}
	& L \|\prox(\xx) - \prox(\yy)  \|_{2}^{2} \\
    &\leq \lin{\nabla f(\yy) - \nabla f(\xx)+ L (\xx - \yy), \prox(\xx) - \prox(\yy)} \\
	&\leq \left(\|\nabla f(\yy) - \nabla f(\xx)\|_{2}  \right. \\
	& \quad\quad \quad \quad\quad \quad \left.+ L \|\xx-\yy\|_{2} \right) \|\prox(\xx) - \prox(\yy)\|_{2} \\
	&\leq 2 L \|\xx-\yy\|_{2} \|\prox(\xx) - \prox(\yy)\|_{2} ,
	\end{align*}
	and the result follows. \qed
\end{proof}

Using $\prox$ operator continuity Lemma \ref{lem:proxcont}, we can
conclude that given any $\yy, \zz \in \C$, if
$\lin{\prox(\yy)-\yy,\yy-\zz}< 0$ and $\lin{\prox(\zz)-\zz,\yy-\zz} >
0$, then there must be a $t^{*} \in (0,1)$ for which $\ww = t^{*} \yy
+ (1-t^{*})\zz$ gives $\lin{\prox(\ww)-\ww,\yy-\zz} = 0$.
Algorithm~\ref{alg:binsearch} finds an approximation to $\ww$ in
$\mathcal{O}(\log L/\epsilon)$ iterations.
\begin{lemma}[Binary Search Lemma]
	\label{lem:binsearch}
	Let $\xx=\text{BinarySearch}(\zz,\yy, \epsilon)$ defined as in Algorithm~\ref{alg:binsearch}. Then one of 3 cases happen:
	\begin{enumerate}[label = (\roman{*})]
		\item \label{lem_item_x_eq_y} $\xx=\yy$ and $\lin{\prox(\xx) - \xx,\xx-\zz}\geq 0$,
		\item \label{lem_item_x_eq_z}  $\xx=\zz$ and $\lin{\prox(\xx) - \xx,\yy-\xx}\leq 0$, or
		%\item $\xx=w$ where $w$ is such that $\lin{\prox(w) - w,\yy-\zz}= 0$. (Note that this is only approximate, but we use equality here to avoid clutter.)
		\item \label{lem_item_x_eq_w} $\xx=t \yy + (1-t) \zz$ for some $t \in (0,1)$ and $| \lin{\prox(\xx) - \xx,\yy-\zz} | \leq 3 \| \yy - \zz \|_{2}^{2} \epsilon$.
	\end{enumerate}
\end{lemma}

\begin{proof}[Proof of Lemma \ref{lem:binsearch}]
	Items~\ref{lem_item_x_eq_y} and~\ref{lem_item_x_eq_z}, are simply Steps~\ref{bin_step_y} and~\ref{bin_step_z}, respectively. For item~\ref{lem_item_x_eq_w}, we have 
	\begin{align*}
	&\|\xx - \ww\|_{2} \\
    &= \| t \yy + (1-t)\zz - t^{*} \yy - (1-t^{*})\zz \|_{2} \\
    &= \| (t - t^{*}) \yy - (t - t^{*}) \zz \|_{2} \\
    &\leq \epsilon \| \yy - \zz \|_{2}.
	\end{align*}
	Now it follows that
	\begin{align*}
	&| \lin{\prox(\xx) - \xx,\yy-\zz} | \\
    &= | \lin{\prox(\xx) - \xx,\yy-\zz} - \lin{\prox(\ww) - \ww,\yy-\zz} | \\
	& \leq \| \lin{\prox(\xx) - \prox(\ww) ,\yy-\zz} \|_{2} + | \lin{\xx - \ww ,\yy-\zz} | \\
	& \leq \| \prox(\xx) - \prox(\ww)\|_{2} \| \yy-\zz \|_2 \\
    &\quad + \| \xx - \ww \|_{2}\|\yy-\zz \|_{2} \\
	 & \leq 2\| \xx - \ww\|_{2} \| \yy-\zz \|_{2} \\
     &\quad + \| \xx - \ww \|_{2}\|\yy-\zz \|_{2} \\
	& = 3 \| \xx - \ww\|_{2} \| \yy-\zz \|_{2} \\
	&\leq 3 \epsilon\| \yy - \zz \|_{2}^{2} .
	\end{align*}
    Here, the third inequality follows by Lemma~\ref{lem:proxcont}
	\qed
\end{proof}

Using the above result, we can prove the following:

\begin{corollary}
	\label{cor:proxavg}
	Let $\xx_{k}$, $\yy_k$, $\zz_{k}$ and $\epsilon_{k}$ be defined as in Algorithm \ref{alg:flag} and $\eta_{k} L_{k} \geq 1$. Then for all $k \geq 1$,
	\begin{align*}
	\lin{\pp_k, \xx_{k}-\zz_{k}} \leq (\eta_{k} L_{k} -1) \lin {\pp_k, \yy_{k} - \xx_{k}} + \frac{DL\eta_k L_k}{T^3}.
	\end{align*} 
\end{corollary}

\begin{proof}[Proof of Corollary \ref{cor:proxavg}]
	Note that by Step~\ref{alg_step_grad_map} of Algorithm \ref{alg:flag}), $\pp_k=-L(\prox(\xx_{k}) - \xx_{k})$. For $k = 1$, since $\xx_1 = \yy_1 = \zz_1$, the inequality is trivially true. For $k \geq 2$, we consider the three cases of Lemma \ref{lem:binsearch}: (i) if $\xx_{k}=\yy_{k}$, the right hand side is $1/T \geq 0$ and the left hand side is $\lin{\pp_k, \xx_{k}-\zz_{k}} = \lin{-L(\prox(\xx_{k})-\xx_{k}), \xx_{k}-\zz_{k}} \leq 0$, (ii) if $\xx_{k}=\zz_{k}$, the left hand side is $0$ and $\lin{\pp_k, \yy_{k}-\xx_{k}} = \lin{-L(\prox(\xx_{k})-\xx_{k}), \yy_{k}-\xx_{k}} \geq 0$, so the inequality holds trivially, and (iii) in this last case, for some $t \in (0,1)$, we have 
	\begin{align*}
	&\lin{\pp_k, \xx_{k}-\zz_{k}} \\
    &= \lin{-L(\prox(\xx_{k})-\xx_{k}), t \yy_{k} + (1-t) \zz_{k} -\zz_{k}} \\
	& = -L t \lin{(\prox(\xx_{k})-\xx_{k}), \yy_{k} -\zz_{k}}, 
	\end{align*}
	and
	\begin{align*}
	&\lin{\pp_k, \yy_{k}-\xx_{k}} \\
    &= \lin{-L(\prox(\xx_{k})-\xx_{k}), \yy_{k} - t \yy_{k} - (1-t) \zz_{k}} \\
	&= -L (1-t) \lin{(\prox(\xx_{k})-\xx_{k}), (\yy_{k} -\zz_{k})}. 
	\end{align*}
	Hence
	\begin{align*}
	&\lin{\pp_k, \xx_{k}-\zz_{k}} - (\eta_{k} L_{k} - 1 )\lin {\pp_k, \yy_{k} - \xx_{k}} \\
    &\leq |\lin{\pp_k, \xx_{k}-\zz_{k}} - (\eta_{k} L_{k} - 1 ) \lin {\pp_k, \yy_{k} - \xx_{k}}| \\
	&= |(-L t + (\eta_{k} L_{k} - 1 ) L (1-t)) \\
	& \quad\quad \quad \quad\quad \quad \lin{(\prox(\xx_{k})-\xx_{k}), (\yy_{k} -\zz_{k})}| \\
	&\leq 3 |(-L t + (\eta_{k} L_{k} - 1 ) L (1-t))| \| \yy_{k} - \zz_{k} \|_{2}^{2} \epsilon_{k} \\
	& = 3 |\eta_{k} L_{k} (1 - t) + 1| L \| \yy_{k} - \zz_{k} \|_{2}^{2} \epsilon_{k} \\
	& = 3 (\eta_{k} L_{k}+ 1) L \| \yy_{k} - \zz_{k} \|_{2}^{2} \epsilon_{k} \\
	& = 6 \eta_{k} L_{k} L \| \yy_{k} - \zz_{k} \|_{2}^{2} \epsilon_{k} \\
	& = \frac{6 D\eta_{k} L_{k} L \| \yy_{k} - \zz_{k} \|_{2}^{2}}{D} \frac{1}{6dT^3} \\
	& \leq \frac{DL\eta_k L_k}{T^3},
	\end{align*}
	where in the last line we used the fact that $\|y_k-z_k\|_2^2 \leq Dd$
	\qed
\end{proof}

Similar to \ref{cor:proxavg} for Algorithm \ref{alg:flag}, the following Lemma proves an analogous result for Algorithm \ref{alg:flare}.
\begin{corollary}\label{cor:proxavgflare}
	Let $\xx_{k}$, $\yy_k$, $\zz_{k}$ and $\epsilon_{k}$ be defined as in Algorithm \ref{alg:flare} and $\eta_{k} \tilde{L}_{k} \geq 1$. Then for all $k \geq 1$,
	\begin{align*}
	\lin{\pp_k, \xx_{k}-\zz_{k}} \leq (\eta_{k} \tilde{L}_{k} -1) \lin {\pp_k, \yy_{k} - \xx_{k}} + \frac{DL\eta_k \tilde{L}_k}{T^3}.
	\end{align*} 
\end{corollary}

\begin{proof}[Proof of Corollary \ref{cor:proxavgflare}]
We consider two cases:

1. If $\xx_k$ is generated through Algorithm \ref{alg:flagiteration}, then $\xx_k = \text{BinarySearch}(\yy_k,\zz_k,\epsilon)$ and $\tilde{L}_k = L_k$, so the statement follows from Corollary \ref{cor:proxavg}.

2. If $x_k$ is generated through Algorithm \ref{alg:advanceandverify}, then $ \xx_{k} = \Big(1-\frac{1}{\eta_{k} \tilde{L}_{k}}\Big) \yy_{k} + \frac{1}{\eta_{k} \tilde{L}_{k}} \zz_{k}$, and so satisfies
$$\lin{\pp_k, \xx_{k}-\zz_{k}} = (\eta_{k} \tilde{L}_{k} -1) \lin {\pp_k, \yy_{k} - \xx_{k}}  .$$
\end{proof}

Next, we state a result regarding the mirror descent step. Similar results can be found in most texts on online optimization, e.g. \cite{allen2014linear}.
\begin{lemma}[Mirror Descent Inequality]
	\label{lem:mir}
	Let $\zz_{k+1} = \arg\min_{\zz\in \C} \lin{\eta_k \pp_k, \zz -\zz_{k}} + \hf \|\zz - \zz_{k}\|_{S_{k}}^{2}$ and $D \defeq \sup_{\xx, \yy \in \C} \|\xx - \yy\|_\infty^2$ be the diameter of $\C$ measured by infinity norm. Then for any $\uu \in \C$, we have
	\begin{align*}
	\sum_{k=1}^T \lin{\eta_k \pp_k, \zz_{k} - u} \leq \sum_{k=1}^{T} \frac{\eta_{k}^{2}}{2} \|\pp_k\|_{S_{k}^{*}}^{2} + \frac{D}{2} \|\vs_{T}\|_{1}  .
	%&\leq \sum_{k=1}^T \frac{\eta_k^2}{2} \|\pp_k\|_{\psi_k^*}^2+\sum_{k=1}^T\left[ B_{\psi_k}(u,\zz_{k}) - B_{\psi_k}(u,\zz_{k+1})\right]\\
	%&\leq \sum_{k=1}^T \frac{\eta_k^2}{2} \|\pp_k\|_{\psi_k^*}^2+\sum_{k=1}^{T-1}\left[ B_{\psi_{k+1}}(u,\zz_{k+1}) - B_{\psi_k}(u,\zz_{k+1})\right] + B_{\psi_1}(u,z_1)
	\end{align*}
\end{lemma}

\begin{proof}[Proof of Lemma \ref{lem:mir}]
	For any $\uu \in \C$ and by optimality of $\zz_{k+1}$, we have $\lin{\eta_{k} \pp_{k},\zz_{k+1} - \uu} \leq \lin{S_{k}(\zz_{k+1} - \zz_{k}),\uu - \zz_{k+1}}$ .
	Hence, using~\eqref{eq:breg_eq} and ~\eqref{eq:dual_norm}, it follows that
	\begin{align*}
	&\lin{\eta_k \pp_k, \zz_{k}-\uu}\\
	&= \lin{\eta_k \pp_k, \zz_{k}-\zz_{k+1}} + \lin{\eta_k \pp_k, \zz_{k+1}-\uu}\\
	&\leq \lin{\eta_k \pp_k, \zz_{k}-\zz_{k+1}} - \lin{S_{k}(\zz_{k+1} - \zz_{k}),\zz_{k+1} - \uu}\\
	&= \lin{\eta_k \pp_k, \zz_{k}-\zz_{k+1}} - \hf \| \zz_{k+1} - \zz_{k} \|_{S_{k}}^{2} \\
    &\quad - \hf \| \zz_{k+1} - \uu \|_{S_{k}}^{2} + \hf \| \uu - \zz_{k} \|_{S_{k}}^{2}\\
	&\leq \sup_{\zz \in \reals^d} \left\{\lin{\eta_k \pp_k, \zz} - \hf \| \zz \|_{S_{k}}^{2}\right\} \\
    &\quad - \hf \| \zz_{k+1} - \uu \|_{S_{k}}^{2} + \hf \| \uu - \zz_{k} \|_{S_{k}}^{2}\\
	&= \frac{\eta_{k}^{2}}{2} \|\pp_k\|_{S_{k}^{*}}^{2} - \hf \| \uu - \zz_{k+1} \|_{S_{k}}^{2} + \hf \| \uu - \zz_{k} \|_{S_{k}}^{2}.
	%&=\lin{\eta_k \pp_k, \zz_{k}-\zz_{k+1}} + B_{\psi_k}(u,\zz_{k}) - B_{\psi_k}(u,\zz_{k+1})-B_{\psi_k}(\zz_{k+1},\zz_{k})\\
	%&\leq \lin{\eta_k \pp_k, \zz_{k}-\zz_{k+1}} - \frac{1}{2}\|\zz_{k}-\zz_{k+1}\|_{\psi_k}^2 + B_{\psi_k}(u,\zz_{k}) - B_{\psi_k}(u,\zz_{k+1})\\
	%&\leq \frac{\eta_k^2}{2}\|\pp_k\|_{\psi_k^*}^2 + B_{\psi_k}(u,\zz_{k}) - B_{\psi_k}(u,\zz_{k+1})
	\end{align*}
	Now recalling from Steps~\ref{algstep:g_ik}-~\ref{algstep:S_k} of Algorithm~\ref{alg:flag} that $S_{k} = \text{diag}(\vs_{k}) + \delta \mathbb{I}$ and $\vs_{k} \geq \vs_{k-1}$, we sum over $k$ to get 
	\begin{align*}
	&\sum_{k=1}^{T} \lin{\eta_k \pp_k, \zz_{k}-u} \\
    &\leq \sum_{k=1}^{T} \frac{\eta_{k}^{2}}{2} \|\pp_k\|_{S_{k}^{*}}^{2} + \hf \| \uu - \zz_{1} \|_{S_{1}}^{2} \\
    &\quad + \sum_{k=2}^{T} \hf \| \uu - \zz_{k} \|_{S_{k}}^{2} - \hf \| \uu - \zz_{k} \|_{S_{k-1}}^{2} \\
	& = \sum_{k=1}^{T} \frac{\eta_{k}^{2}}{2} \|\pp_k\|_{S_{k}^{*}}^{2} + \hf \| \uu - \zz_{1} \|_{S_{1}}^{2} \\
    &\quad + \hf \sum_{k=2}^{T}  \lin{(S_{k} - S_{k-1}) (\uu - \zz_{k}), \uu - \zz_{k}} \\
	& \leq \sum_{k=1}^{T} \frac{\eta_{k}^{2}}{2} \|\pp_k\|_{S_{k}^{*}}^{2} + \hf \| \uu - \zz_{1}\|_{\infty}^{2} \lin{\vs_{1},\vec{1}} \\
    &\quad + \hf \sum_{k=2}^{T}  \| \uu - \zz_{k}\|_{\infty}^{2} \lin{\vs_{k} - \vs_{k-1},\vec{1}} \\
	& \leq \sum_{k=1}^{T} \frac{\eta_{k}^{2}}{2} \|\pp_k\|_{S_{k}^{*}}^{2} + \frac{D}{2} \lin{s_{1},\vec{1}} + \frac{D}{2} \sum_{k=2}^{T}   \lin{\vs_{k} - \vs_{k-1},\vec{1}} \\
	& = \sum_{k=1}^{T} \frac{\eta_{k}^{2}}{2} \|\pp_k\|_{S_{k}^{*}}^{2} + \frac{D}{2} \|\vs_{T}\|_{1}  .
	\end{align*}
	\qed
\end{proof}

Finally, we state a similar result to that of~\cite{duchi2011adaptive} that captures the benefits of using $S_k$ in FLAG. 
\begin{lemma} [AdaGrad Inequalities]
	\label{lem:adareg}
	Define $q_{_T} \defeq \sum_{i=1}^d \|G_{T}(i,:)\|_2$, where $G_{k}$ is as in Step~\ref{algstep:g_ik} of Algorithm~\ref{alg:flag}. We have
	\begin{enumerate}[label = (\roman{*})]
		\item \label{lem_item_q_low} $\sum_{k=1}^T \vg_{k}^{T} S_{k}^{-1} \vg_{k} \leq 2q_{_T},$
		\item \label{lem_item_q}
		$q_{_T}^{2}=\min_{S \in \mathcal{S}}\sum_{k=1}^T \vg_{k} ^T S^{-1} \vg_{k}$, 
		where $\mathcal{S} \defeq \{S \in \reals^{d \times d} \; | \; S \text{ is diagonal}, \; S_{ii} > 0, \; \text{trace}(S) \leq 1 \}$, and
		%\item $q=\min_{\|s\|_1\leq q}\sum_{k=1}^T \vg_{k} ^T S^{-1} \vg_{k}$
		\item \label{lem_item_bnd}$\sqrt{T}\leq q_{_T}\leq\sqrt{dT}$.
		%\item $q = \sqrt{d} \sqrt{\min_{\|s\|_1\leq d}\sum_{k=1}^T \vg_{k} ^T S^{-1} \vg_{k}}$
	\end{enumerate}
\end{lemma}

\begin{proof}[Proof of Lemma \ref{lem:adareg}]
	To prove part~\ref{lem_item_q_low}, we use the following inequality introduced in the proof of  Lemma 4 in \cite{duchi2011adaptive}: for any arbitrary real-valued sequence of $\{a_i\}_{i=1}^{T}$ and its vector representation as $a_{1:T} = [a_1, a_2,\ldots, a_T]$, we have
	$$\sum_{k=1}^T \frac{a_k^2}{\|a_{1:k}\|_2}\leq 2\|a_{1:T}\|_{2}.$$
	So it follows that
	\begin{align*}
    &\sum_{k=1}^T \vg_{k}^{T} S_{k}^{-1} \vg_{k} \\
    &= \sum_{k=1}^T \sum_{i=1}^{d} \frac{\vg_{k}^{2}(i)}{\vs_{k}^{2}(i)} \\
    &= \sum_{i=1}^{d} \sum_{k=1}^T \frac{\vg_{k}^{2}(i)}{\vs_{k}(i)} \\
    &= \sum_{i=1}^{d} \sum_{k=1}^T \frac{\vg_{k}^{2}(i)}{\|G_{k}(i,:)\|_2 } \\
    &\leq 2q_{_T},
    \end{align*}
	where the last equality follows from the definition of $\vs_k$ in Step~\ref{algstep:s_ki} of Algorithm~\ref{alg:flag}.
	
	For the rest of the proof, one can easily see that
	\begin{align*}
	\sum_{k=1}^T \vg_{k} ^T S^{-1} \vg_{k} = \sum_{k=1}^T \sum_{i=1}^{d} \frac{\vg_{k}^{2}(i)}{\vs(i)} = \sum_{i=1}^{d} \frac{a(i)}{\vs(i)},
	\end{align*}
	where $a(i) \defeq \sum_{k=1}^T  \vg_{k}^{2}(i)$ and $\vs = \text{diag}(S)$. Now the Lagrangian for $\lambda \geq 0$ and $\bm{\nu} \geq \vec{0}$, can be written as
	\begin{align*}
	\mathcal{L}(\vs,\lambda, \bm{\nu}) &= \sum_{i=1}^{d} \frac{a(i)}{\vs(i)} + \lambda \left(\sum_{i=1}^{d} \vs(i) - 1\right) + \lin{\bm{\nu}, \vs}.
	\end{align*}
	Since the strong duality holds, for any primal-dual optimal solutions, $S^{*}, \lambda^{*}$ and $\bm{\nu}^{*}$, it follows from complementary slackness that $\bm \nu^{*} = \vec{0}$ (since $\vs^{*} > \vec{0}$). Now requiring that ${\partial \mathcal{L}(\vs^{*},\lambda^{*}, \bm{\nu}^{*})}/{\partial \vs(i)}  = 0$ gives $\lambda^{*} \vs^{*}(i) = \sqrt{a_{i}} > 0$, which since $\vs^{*}(i) > 0$, implies that $\lambda^{*} > 0$. As a result, by using complementary slackness again, we must have $\sum_{i=1}^{d} \vs^{*}(i) = 1$. Now simple algebraic calculations gives $\vs^{*}(i) = \sqrt{a_{i}}/(\sum_{i=1}^{d} \sqrt{a_{i}})$ and part~\ref{lem_item_q} follows.
	
	For part~\ref{lem_item_bnd}, recall that $\|\vg_{k}\|_2=1$. Now, since $\lambda_{\min} (S^{01}) \geq 1$, one has $1 \leq \vg_{k} ^T S^{-1} \vg_{k}$, and so $q_{_T} \geq 1$. One the other hand, consider the optimization problem
	\begin{align*}
	\begin{aligned}
	\max \; & \sum_{i=1}^d \|G_{T}(i,:)\|_2 = \sum_{i=1}^{d} \sqrt{\sum_{k=1}^{T} \vg_{i}^{2}(k)} \\
	\text{s.t.} &  \; \|\vg_{k}\|^{2}_2 = 1, \; k =1,2,\ldots,T.
	\end{aligned}
	\end{align*}
	The Lagrangian can be written as
	\begin{align*}
	\mathcal{L}(\{\vg_{k}\}_{k=1}^{T},\{\lambda\}_{k=1}^{T}) = &\sum_{i=1}^d \sqrt{\sum_{k=1}^{T} \vg_{i}^{2}(k)} \\
	& + \sum_{k=1}^{T} \lambda_k \left(1 - \sum_{i=1}^{d} \vg^{2}_{i}(k)\right).
	\end{align*}
	By KKT necessary condition, we require that $\partial \mathcal{L}(\{\vg_{k}\}_{k=1}^{T},\{\lambda\}_{k=1}^{T}) / \partial \vg_{i}(k) = 0$, which implies that
	$\lambda_{k} = {1}/({2 \sqrt{\sum_{k=1}^{T} \vg_{i}^{2}(k)} }), \quad i = 1,2,\ldots,d$. Hence, $T = \sum_{i=1}^{d} \sum_{k=1}^{T} \vg_{i}^{2}(k)  = {d}/({4\lambda_{k}^{2}})$,
	and so $2 \lambda_{k} = \sqrt{{d}/{T}}$, which gives $q_{_T} \leq \sqrt{dT}$. \qed
\end{proof}

We can now prove the central theorems of which is used to obtain FLAG's main result.
\begin{theorem}
	\label{thm:mastertheorem}
	Let $D \defeq \sup_{\xx, \yy \in \C} \|\xx - \yy\|_\infty^2$. For any $\uu \in \C$, after $T$ iterations of Algorithm \ref{alg:flag}, we get
	%\begin{align*}
	%\sum_{k=1}^T \eta_k (F(\yy_{k+1}) - F(\uu))   \leq 1 + \frac{D}{2}\|\vs_T\|_1 + \sum_{k=1}^T\eta_k( \eta_k L_k -1 )\left( F(\yy_{k}) - F(\yy_{k+1}) \right).
	%\end{align*}
	\begin{align*}
	&\sum_{k=1}^T \Big\{\left(\eta^{2}_{k-1} L_{k-1} - \eta_{k}^{2} L_{k} + \eta_{k} \right) F(\yy_{k}) - \eta_{k} F(\uu)  \Big\} \\
    &\quad + \eta_{T}^{2} L_{T} F(\yy_{T+1}) \\
	&\leq \sum_{k=1}^T \frac{DL \eta_k^2 L_k}{T^3} + \frac{D}{2}\|\vs_T\|_1.
	\end{align*}
	%$$\sum_{k=1}^T \eta_k (F(\yy_{k}) - F(u)) \leq \sum_{k=1}^T (\eta_k^2 L_k - \eta_{k+1}^2 L_{k+1}) F(\yy_{k}) + 2D\|s_T\|_1$$
\end{theorem}

\begin{proof}[Proof of Theorem \ref{thm:mastertheorem}]
	Noting that $\pp_{k} = -L (\yy_{k+1} - \xx_{k})$ is the gradient mapping of $F$ on $\C$, it follows that
	\begin{align*}
	&\sum_{k=1}^T \eta_k (F(\yy_{k+1}) - F(\uu))  \\
    & = \sum_{k=1}^T \eta_k (F(\prox(\xx_{k})) - F(\uu))\\
	& \leq \sum_{k=1}^T \eta_k \lin{\pp_{k}, \xx_{k}-\uu} - \frac{\eta_{k}}{2 L} \| \pp_{k} \|_{2}^{2} \\  
	&= \sum_{k=1}^T \eta_k\lin{\pp_k, (\zz_{k} - \uu)} + \sum_{k=1}^T\eta_k\lin{\pp_k, \xx_{k} - \zz_{k}} - \frac{\eta_{k}}{2 L} \| \pp_{k} \|_{2}^{2} \\
	%\label{line:122}
	&\leq \sum_{k=1}^{T} \frac{\eta_{k}^{2}}{2} \|\pp_k\|_{S_{k}^{-1}}^{2} + \frac{D}{2} \|\vs_{T}\|_{1}  + \sum_{k=1}^T\eta_k\lin{\pp_k, \xx_{k} - \zz_{k}} - \frac{\eta_{k}}{2 L} \| \pp_{k} \|_{2}^{2} \\
	&= \sum_{k=1}^{T} \frac{\eta_{k}(\eta_{k} L_{k}-1)}{2L}\|\pp_k\|_{2}^{2} + \frac{D}{2} \|\vs_{T}\|_{1}  + \sum_{k=1}^T\eta_k\lin{\pp_k, \xx_{k} - \zz_{k}} \\
	&\leq \sum_{k=1}^{T} \frac{\eta_{k}(\eta_{k} L_{k}-1)}{2L}\|\pp_k\|_{2}^{2} + \frac{D}{2} \|\vs_{T}\|_{1} \\
	& \quad \quad \quad+ \sum_{k=1}^T \left(\eta_k ( \eta_k L_k -1 ) \lin{\pp_k, \yy_{k} - \xx_{k}} + \frac{DL\eta_k^2 L_k}{T^3} \right) \\
	&\leq \sum_{k=1}^T \frac{DL\eta_k^2 L_k}{T^3} + \frac{D}{2}\|\vs_T\|_1 \\
	&\quad + \sum_{k=1}^T\eta_{k}(\eta_{k} L_{k}-1) \left( F(\yy_{k}) - F(\yy_{k+1}) \right). \quad\text{(Lemma~\ref{lemma:grad_map}) }
	%\label{l:212}
	\end{align*}
Here, the first inequality is by Lemma~\ref{lemma:grad_map}, the second inequality is by Lemma \ref{lem:mir}, the third equality is by Step~\ref{algstep:L_k} of Algorithm \ref{alg:flag}, and the second last inequality is by Corollary~\ref{cor:proxavg}.
	Now we have
	\begin{align*}
	& \sum_{k=1}^T \eta_k (F(\yy_{k+1}) - F(\uu))  - \eta_{k}(\eta_{k} L_{k}-1) \left( F(\yy_{k}) - F(\yy_{k+1}) \right) \\
	& =\sum_{k=1}^T \eta_k  F(\yy_{k+1}) - \eta_k  F(\uu)  - \eta_{k}(\eta_{k} L_{k}-1) F(\yy_{k}) \\
    & \quad + \eta_{k}(\eta_{k} L_{k}-1) F(\yy_{k+1}) \\
	& = \sum_{k=1}^T \eta_{k}^{2} L_{k} F(\yy_{k+1}) - \eta_k  F(\uu)  - \eta_{k}(\eta_{k} L_{k}-1) F(\yy_{k}) \\
	&= \eta_{T}^{2} L_{T} F(\yy_{T+1}) \\
    &\quad + \sum_{k=1}^T \eta_{k-1}^{2} L_{k-1} F(\yy_{k}) - \eta_k  F(\uu)  - \eta_{k}(\eta_{k} L_{k}-1) F(\yy_{k}) \\
	& = \eta_{T}^{2} L_{T} F(\yy_{T+1}) \\
    &\quad + \sum_{k=1}^T \left( \eta_{k-1}^{2} L_{k-1} - \eta_{k}^{2} L_{k} + \eta_{k} \right) F(\yy_{k}) - \eta_k  F(\uu),
	\end{align*}
	and the result follows. \qed
\end{proof}

Once again, we present the analog of Theorem \ref{thm:mastertheorem} for Algorithm \ref{alg:flare}.
\begin{theorem} 
	\label{thm:mastertheoremflare}
	Let $D \defeq \sup_{\xx, \yy \in \C} \|\xx - \yy\|_\infty^2$. For any $\uu \in \C$, after $T$ iterations of Algorithm \ref{alg:flag}, we get
	\begin{align*}
	&\sum_{k=1}^T \Big\{\left(\eta^{2}_{k-1} \tilde{L}_{k-1} - \eta_{k}^{2} \tilde{L}_{k} + \eta_{k} \right) F(\yy_{k}) - \eta_{k} F(\uu)  \Big\} \\
    &\quad + \eta_{T}^{2} \tilde{L}_{T} F(\yy_{T+1}) \\
	&\leq \sum_{k=1}^T \frac{D\tilde{L} \eta_k^2 \tilde{L}_k}{T^3} + \frac{D}{2}\|\vs_T\|_1.
	\end{align*}
	%$$\sum_{k=1}^T \eta_k (F(\yy_{k}) - F(u)) \leq \sum_{k=1}^T (\eta_k^2 L_k - \eta_{k+1}^2 L_{k+1}) F(\yy_{k}) + 2D\|s_T\|_1$$
\end{theorem}

\begin{proof}[Proof of Theorem \ref{thm:mastertheoremflare}] Parts of this proof which differ from the proof of Theorem \ref{thm:mastertheorem} are bolded.
	Noting that $\pp_{k} = -L (\yy_{k+1} - \xx_{k})$ is the gradient mapping of $F$ on $\C$, it follows that
	\begin{align*}
	&\sum_{k=1}^T \eta_k (F(\yy_{k+1}) - F(\uu))  \\
    &= \sum_{k=1}^T \eta_k (F(\prox(\xx_{k})) - F(\uu))\\
	&\leq \sum_{k=1}^T \eta_k \lin{\pp_{k}, \xx_{k}-\uu} - \frac{\eta_{k}}{2 L} \| \pp_{k} \|_{2}^{2} \quad \\  
	&= \sum_{k=1}^T \eta_k\lin{\pp_k, (\zz_{k} - \uu)} + \sum_{k=1}^T\eta_k\lin{\pp_k, \xx_{k} - \zz_{k}} \\
    &\quad - \frac{\eta_{k}}{2 L} \| \pp_{k} \|_{2}^{2} \\
	&\leq \sum_{k=1}^{T} \frac{\eta_{k}^{2}}{2} \|\pp_k\|_{S_{k}^{-1}}^{2} + \frac{D}{2} \|\vs_{T}\|_{1}  + \sum_{k=1}^T\eta_k\lin{\pp_k, \xx_{k} - \zz_{k}} \\
    &\quad - \frac{\eta_{k}}{2 L} \| \pp_{k} \|_{2}^{2}\\
	&= \sum_{k=1}^{T} \frac{\eta_{k}(\eta_{k} \tilde{L}_{k}-1)}{2L}\|\pp_k\|_{2}^{2} + \frac{D}{2} \|\vs_{T}\|_{1}  \\
	& \quad \quad \quad + \sum_{k=1}^T\eta_k\lin{\pp_k, \xx_{k} - \zz_{k}} \\
	&\leq \sum_{k=1}^{T} \frac{\eta_{k}(\eta_{k} \tilde{L}_{k}-1)}{2L}\|\pp_k\|_{2}^{2} + \frac{D}{2} \|\vs_{T}\|_{1}   \\
	& \quad + \sum_{k=1}^T \left(\eta_k ( \eta_k \tilde{L}_k -1 ) \lin{\pp_k, \yy_{k} - \xx_{k}} + \frac{DL\eta_k^2 \tilde{L}_k}{T^3} \right)\\
	&\leq \sum_{k=1}^T \frac{DL\eta_k^2 \tilde{L}_k}{T^3} + \frac{D}{2}\|\vs_T\|_1 \\
    &\quad + \sum_{k=1}^T\eta_{k}(\eta_{k} \tilde{L}_{k}-1) \left( F(\yy_{k}) - F(\yy_{k+1}) \right). 
	%\label{l:212}
	\end{align*}
    Here, the first inequality follows from Lemma~\ref{lemma:grad_map}, the second inequality follows from Lemma \ref{lem:mir}, the last equality follows from Steps~\ref{algstep:L_k_flare_1} and \ref{algstep:L_k_accept_condition} of Algorithm~\ref{alg:advanceandverify}, Steps \ref{algstep:L_k_flare_2} and \ref{algstep:L_k_flare_3} of Algorithm~\ref{alg:flagiteration}, and the second last inequality follows from Corollary~\ref{cor:proxavgflare}, and the last equality follows from Lemma~\ref{lemma:grad_map}.
	
	Now we have
	\begin{align*}
	& \sum_{k=1}^T \eta_k (F(\yy_{k+1}) - F(\uu))  \\
	& \quad - \eta_{k}(\eta_{k} \tilde{L}_{k}-1) \left( F(\yy_{k}) - F(\yy_{k+1}) \right) \\
	&= \sum_{k=1}^T \eta_k  F(\yy_{k+1}) - \eta_k  F(\uu)  - \eta_{k}(\eta_{k} \tilde{L}_{k}-1) F(\yy_{k}) \\
    &\quad + \eta_{k}(\eta_{k} \tilde{L}_{k}-1) F(\yy_{k+1})\\
	& =\sum_{k=1}^T \eta_{k}^{2} L_{k} F(\yy_{k+1}) - \eta_k  F(\uu)  - \eta_{k}(\eta_{k} \tilde{L}_{k}-1) F(\yy_{k}) \\
	& =\eta_{T}^{2} \tilde{L}_{T} F(\yy_{T+1})\\
    & \quad + \sum_{k=1}^T \eta_{k-1}^{2} \tilde{L}_{k-1} F(\yy_{k}) - \eta_k  F(\uu)  \\
    & \quad - \eta_{k}(\eta_{k} \tilde{L}_{k}-1) F(\yy_{k}) \\
	& = \eta_{T}^{2} \tilde{L}_{T} F(\yy_{T+1}) \\
    &\quad + \sum_{k=1}^T \left( \eta_{k-1}^{2} \tilde{L}_{k-1} - \eta_{k}^{2} \tilde{L}_{k} + \eta_{k} \right) F(\yy_{k}) - \eta_k  F(\uu),
	\end{align*}
	and the result follows. \qed
\end{proof}

We now set out to put the final piece of the proof in place: choosing the stepsize $\eta_k$ for the mirror descent step.

\begin{lemma}
	\label{lem:eta_k}
	For the choice of $\eta_k$ in Algorithm~\ref{alg:flag} and $k \geq 1$, we have
	\begin{enumerate}[label = (\roman*)]
		\item \label{eta_k_rec} $\eta_k^2 L_{k} = \sum_{i=1}^k \eta_i$,
		\item $\eta^{2}_{k-1} L_{k-1} - \eta_{k}^{2} L_{k} + \eta_{k} = 0$, and
		\item $\eta_kL_k \geq 1$.
	\end{enumerate}
\end{lemma}
\begin{proof}
	We prove~\ref{eta_k_rec} by induction. For $k = 1$, is is easy to verify that $\eta_1 = 1/L_{1}$, and so $\eta_{1}^2 L_{1} = \eta_{1}$ and the base case follows trivially. Now suppose $\eta_{k-1}^2 L_{k-1} = \sum_{i=1}^{k-1} \eta_i$. Re-arranging~\ref{eta_k_rec} for $k$ gives
	\begin{align*}
	0 = \eta_k^2 L_{k} -\eta_{k} - \sum_{i=1}^{k-1} \eta_i = \eta_k^2 L_{k} -\eta_{k} - \eta_{k-1}^2 L_{k-1}. 
	\end{align*}
	Now, it is easy to verify that the choice of $\eta_k$ in Algorithm~\ref{alg:flag} is a solution of the above quadratic equation. The rest of the items follow immediately from part~\ref{eta_k_rec}. \qed
\end{proof}

Once again, the FLARE analog of Lemma \ref{lem:eta_k} is

\begin{lemma}
	\label{lem:eta_k_flare}
For the choice of $\eta_k$ in Algorithm \ref{alg:flare} and $k\geq 1$, we have
\begin{enumerate}[label = (\roman*)]
		\item $\eta_k^2 \tilde{L}_{k} = \sum_{i=1}^k \eta_i$,
		\item $\eta^{2}_{k-1} \tilde{L}_{k-1} - \eta_{k}^{2} \tilde{L}_{k} + \eta_{k} = 0$, and
		\item $\eta_k\tilde{L}_k \geq 1$.
	\end{enumerate}
\end{lemma}
\begin{proof}[Proof of Lemma \ref{lem:eta_k_flare}]
Completely identical to proof of Lemma \ref{lem:eta_k}.
\end{proof}

\begin{corollary} 
	\label{cor:maintheorem}
	Let $D \defeq \sup_{\xx, \yy \in \C} \|\xx - \yy\|_\infty^2$. For any $\uu \in \C$, after $T$ iterations of Algorithm \ref{alg:flag}, we get
	%\begin{align*}
	%\sum_{k=1}^T \eta_k (F(\yy_{k+1}) - F(\uu))   \leq 1 + \frac{D}{2}\|\vs_T\|_1 + \sum_{k=1}^T\eta_k( \eta_k L_k -1 )\left( F(\yy_{k}) - F(\yy_{k+1}) \right).
	%\end{align*}
	\begin{align*}
	F(\yy_{T+1}) - F(\uu) \leq  \frac{LD}{T^2} + \frac{D\|\vs_T\|_1}{2 \sum_{k=1}^T \eta_{k}}.
	\end{align*}
	%$$\sum_{k=1}^T \eta_k (F(\yy_{k}) - F(u)) \leq \sum_{k=1}^T (\eta_k^2 L_k - \eta_{k+1}^2 L_{k+1}) F(\yy_{k}) + 2D\|s_T\|_1$$
\end{corollary}
\begin{proof}[Proof of corollary \ref{cor:maintheorem}]
	The result follows from Theorem~\ref{thm:mastertheorem} and Lemma~\ref{lem:eta_k} as well as noting that $\eta_k^2 L_k =\sum_{i=1}^k \eta_{i}\le \sum_{i=1}^T \eta_{i} =  \eta_T^2 L_T $. \qed
\end{proof}

The FLARE analog:
\begin{corollary} 
	\label{cor:maintheoremflare}
	Let $D \defeq \sup_{\xx, \yy \in \C} \|\xx - \yy\|_\infty^2$. For any $\uu \in \C$, after $T$ iterations of Algorithm \ref{alg:flare}, we get
	\begin{align*}
	F(\yy_{T+1}) - F(\uu) \leq  \frac{LD}{T^2} + \frac{D\|\vs_T\|_1}{2 \sum_{k=1}^T \eta_{k}}.
	\end{align*}
\end{corollary}
\begin{proof}[Proof of corollary \ref{cor:maintheoremflare}]
	The result follows from Theorem~\ref{thm:mastertheoremflare} and Lemma~\ref{lem:eta_k_flare} as well as noting that $\eta_k^2 L_k =\sum_{i=1}^k \eta_{i}\leq \sum_{i=1}^T \eta_{i} =  \eta_T^2 \tilde{L}_T $. \qed
\end{proof}

Finally, it only remains to lower bound  $\sum_{k=1}^T \eta_{k}$, which is done in the following Lemma.

\begin{lemma}
	\label{lem:sum_bound}
	For the choice of $\eta_k$ in Algorithm~\ref{alg:flag}, we have
	\begin{align*}
	\sum_{k=1}^T \eta_{k} \geq \frac{T^3}{1000 \sum_{k=1}^T L_k}  .
	\end{align*}
\end{lemma}

\begin{proof}[Proof of Lemma \ref{lem:sum_bound}]
	We prove by induction on $T$. For $T=1$, we have $\eta_1 = 1/L_{1}$, and the base case holds trivially. Suppose the desired relation holds for $T-1$. We have
	\begin{align*}
	\sum_{k=1}^T \eta_{k} &= \sum_{k=1}^{T-1} \eta_{k} + \eta_{T} \\
	 &\geq \frac{(T-1)^3}{1000\sum_{k=1}^{T-1} L_k} + \frac{1}{2 L_T} \\
     &+ \sqrt{\frac{1}{4 L_T^2} + \frac{(T-1)^3}{1000L_T \sum_{k=1}^{T-1} L_k} } \\
	&\geq \frac{(T-1)^3}{1000\sum_{k=1}^{T-1} L_k} + \sqrt{\frac{(T-1)^3}{1000L_T \sum_{k=1}^{T-1} L_k} } \\
	&\geq \frac{(T-1)^3}{1000\sum_{k=1}^{T-1} L_k} + \sqrt{\frac{T^3}{8000 L_T \sum_{k=1}^{T-1} L_k} }.
	\end{align*}
    Here, the first inequality is by the induction hypothesis on $\eta_{k}$. 
	Now if 
	\begin{align*}
	\frac{(T-1)^3}{1000\sum_{k=1}^{T-1} L_k} \geq \frac{T^3}{1000\sum_{k=1}^{T} L_k},
	\end{align*}
	then we are done. Otherwise denoting $\alpha \defeq \sum_{k=1}^{T} L_k$, we must have that
	\begin{align*}
	L_{T} &\leq \frac{\alpha T^3 - \alpha (T-1)^3}{T^3} \\
	&= \frac{\alpha T^3 - \alpha \left( T^3 - 3T^2 + 3T -1 \right)}{T^3} \\
	&= \frac{ \alpha( 3T^2 - 3T +1 ) }{T^3} \\
	&\leq \frac{4 \sum_{k=1}^T L_{k}}{T}.
	\end{align*}
	Hence, we get
	\begin{align*}
	\sum_{k=1}^T \eta_{k} &\geq \frac{(T-1)^3}{1000\sum_{k=1}^{T-1} L_k} + \sqrt{ \frac{T^4}{32000 L_T \left(\sum_{k=1}^{T} L_k \right)^2 } } \\
	&\geq \frac{(T-1)^3}{1000\sum_{k=1}^{T} L_k} + \frac{4T^2}{1000 \sum_{k=1}^{T} L_k  } \\
	&\geq \frac{T^3}{1000 \sum_{k=1}^{T} L_k  }.
	\end{align*}
	\qed
\end{proof}

\begin{remark}[Remark:]
	We note here that we made little effort to minimize constants, and that we used rather sloppy bounds such as $T-1\geq T/2$. As a result, the constant appearing above is very conservative and a mere by-product of our proof technique. %We have numerically verified that a much smaller constant (e.g., $10$) indeed satisfies the bound above.
\end{remark}
\begin{lemma}
	\label{lem:sum_bound_flare}
	For the choice of $\eta_k$ in Algorithm~\ref{alg:flare}, we have
	\begin{align*}
	\sum_{k=1}^T \eta_{k} \geq \frac{T^3}{\lambda \cdot 1000 \sum_{k=1}^T L_k}   .
	\end{align*}
\end{lemma}

\begin{proof}[Proof of Lemma \ref{lem:sum_bound_flare}]
Once again, exactly identical to the proof of Lemma \ref{lem:sum_bound}, we have 
	\begin{align*}
	\sum_{k=1}^T \eta_{k} \geq \frac{T^3}{ 1000 \sum_{k=1}^T \tilde{L}_k}  .
	\end{align*}
Finally, using the guarantee that $\tilde{L}_k \leq \lambda L_k$ from Step \ref{algstep:L_k_accept_condition} of Algorithm \ref{alg:advanceandverify} and Step \ref{algstep:L_k_flare_3} from Algorithm \ref{alg:flagiteration}, we get the conclusion.
\end{proof}

The proof of FLAG's main result, Theorem~\ref{thm:final_flag}, follows rather immediately.
\begin{proof}[Proof of Theorem~\ref{thm:final_flag}]
	The result follows immediately from Lemma~\ref{lem:sum_bound} and Corollary~\ref{cor:maintheorem} and noting that $\sum_{k=1}^T L_k = L \sum_{k=1}^T \vg_{k}^T S^{-1}_{k} \vg_k \leq 2 L q_{_T}$ by Lemma~\ref{lem:adareg} and $\|\s_T\|_1 = q_{_T}$ by Step~\ref{algstep:s_ki} of Algorithm~\ref{alg:flag} and definition of $q_{_T}$ in Lemma~\ref{lem:adareg}.
	This gives
	\begin{align*}
	F(\yy_{T+1}) - F(\uu) \leq  \frac{LD}{T^2} + \frac{q_{_T}^2}{T}\frac{1000LD}{T^2} \leq \frac{q_{_T}^2}{T}\frac{1001LD}{T^2}.
	\end{align*}
	Now from Lemma~\ref{lem:adareg}, we see that $ \beta \defeq  {q_{_T}^2}/{T} \in [1,d]$. Finally, the run-time per iteration follows from having to do $\log_2({1}/{\epsilon})$ calls to bisection, each taking $\mathcal{O}(\mathcal{T}_{_{\prox}})$ time. \qed
\end{proof}

The proof of FLARE's main result, Theorem \ref{thm:final_flare}, is obtained similarly to that of Theorem~\ref{thm:final_flag}.
\begin{proof}[Proof of Theorem~\ref{thm:final_flare}]
	The result follows immediately from Lemma~\ref{lem:sum_bound_flare} and Corollary~\ref{cor:maintheoremflare} and noting that $\sum_{k=1}^T L_k = L \sum_{k=1}^T \vg_{k}^T S^{-1}_{k} \vg_k \leq 2 L q_{_T}$ by Lemma~\ref{lem:adareg} and $\|\s_T\|_1 = q_{_T}$ by Step~\ref{algstep:sk_i_flare_1} of Algorithm~\ref{alg:advanceandverify} and Step \ref{algstep:sk_i_flare_2} of Algorithm \ref{alg:flagiteration} and definition of $q_{_T}$ in Lemma~\ref{lem:adareg}.
	This gives
	\begin{align*}
	F(\yy_{T+1}) - F(\uu) &\leq  \frac{LD}{T^2} + \frac{q_{_T}^2}{T}\frac{1000\lambda LD}{T^2} \\
	&\leq \frac{q_{_T}^2}{T}\frac{1001\lambda LD}{T^2}.
	\end{align*}
	Now from Lemma~\ref{lem:adareg}, we see that $ \beta \defeq  {q_{_T}^2}/{T} \in [1,d]$. Finally, we try to guess a suitable $\tilde{L}_k$ for $\log (d/\epsilon)$ times, and resort to BinarySearch after. If we resort to Algorithm \ref{alg:flagiteration} (essentially BinarySeaerch), we make $\log(1/{\epsilon})$ calls to bisection, so overall the number of inner iterations per outer iteration is same as Algorithm \ref{alg:flag}. Each inner iteration takes $\mathcal{O}(\mathcal{T}_{_{\prox}})$ time in the worst case (if we have to resort to Algorithm \ref{alg:flagiteration} each time). \qed
\end{proof}

\if false
\subsection{Proof of Theorem~\ref{thm:final_flare}}
\label{sec:proof_flare}
The proof of Theorem~\ref{thm:final_flare} follows very similar line of reasoning as that of Theorem~\ref{thm:final_flag}, with very minor modifications. Indeed, Lemmas \ref{lemma:grad_map},~\ref{lem:mir},~\ref{lem:adareg},~\ref{lem:eta_k},~\ref{lem:sum_bound} hold without any change. 
We also have the following immediate lemma.
\begin{lemma}
	\label{lem:flare}
	At each iteration $ k $ of Algorithm~\ref{alg:flare} we have 
	\begin{enumerate}[label = (\roman*)]
		\item \label{lem_flare_pk} $\|\pp_k\|_{S_{k}^{-1}}^{2} \leq L_{k} \|\pp_{k}\|^{2}/L$,
		\item \label{lem_flare_couple} $\lin{\pp_k, \xx_{k}-\zz_{k}} = (\eta_{k} L_{k} -1) \lin {\pp_k, \yy_{k} - \xx_{k}}$,
		\item \label{lem_flare_max_inner_iter} after at most ${\log\left( \frac{\sqrt{k-1}+\delta}{\delta}\right)}/{\log(\gamma)}$ repetitions of Steps~\ref{alg_step_Lk_gamma_flare}, we have $ L{ \vg_{k}^{T} S^{-1}_{k} \vg_{k}} \leq L_{k}$, i.e., wrong guess for $ L_{k} $ at iteration $ k-1 $ is corrected, and
		\item \label{lem_flare_sum_Lk} $ \sum_{k=1}^{T} L_{k} \leq 2 \gamma L q_{_{T}} $, where $q_{_{T}} $ is as in Lemma~\ref{lem:adareg}.
	\end{enumerate}
\end{lemma}
\begin{proof}
	Parts~\ref{lem_flare_pk} and~\ref{lem_flare_couple} immediately follow from the corresponding steps in the algorithm. For Part~\ref{lem_flare_max_inner_iter}, note that by Steps~\ref{alg_step_vg_k_flare}--\ref{algstep:S_k_flare} of Algorithm~\ref{alg:flare}, for all $ k $, we have
	\begin{align*}
	\delta \mathbb{I} \preceq S_{k} \preceq (\sqrt{k}+\delta) \mathbb{I}.
	\end{align*}
	Hence, it follows that, for all $ k $
	\begin{align*}
	\frac{L}{\sqrt{k} + \delta} \leq L \vg_{k}^{T} S^{-1}_{k} \vg_{k} \leq \frac{L}{\delta}.
	\end{align*}
	As a result, at iteration $ k $, Step~\ref{alg_step_Lk_gamma_flare} of Algorithm~\ref{alg:flare} is performed at most
	\begin{align*}
	t_{0} = \ceil*{\frac{\log\left( \frac{\sqrt{k-1}+\delta}{\delta}\right)}{\log(\gamma)}},
	\end{align*}
	times, upon which, we are guaranteed to have $ L \vg_{k}^{T} S^{-1}_{k} \vg_{k} \leq \gamma^{t_{0}} L \vg_{k-1}^{T} S^{-1}_{k-1} \vg_{k-1} \defeqr L_{k} $. 
	Suppose $ t \in [0,t_{0}]$ is the first time that  Step~\ref{alg_step_Lk_gamma_flare} of Algorithm~\ref{alg:flare} gives $ L \vg_{k}^{T} S^{-1}_{k} \vg_{k} \leq L_{k} $. In this case, we must have $ L_{k} \leq \gamma L \vg_{k}^{T} S^{-1}_{k} \vg_{k} $, since otherwise, this implies that $ t-1 $ must have been the first, which is a contradiction. Finally for Part~\ref{lem_flare_sum_Lk}, by Part~\ref{lem_item_q_low} of Lemma~\ref{lem:adareg}, we have
	\begin{align*}
	\sum_{k=1}^{T} L_{k} \leq  \gamma \sum_{k=1}^{T} L \vg_{k}^{T} S^{-1}_{k} \vg_{k} \leq 2 \gamma L q_{_{T}}. 
	\end{align*}
	\qed
\end{proof}

The following theorem is equivalent to Theorem~\ref{thm:mastertheorem}, but adopted for FLARE, and hence, its proof is omitted.
\begin{theorem} 
	\label{thm:mastertheorem_flare}
	Let $D \defeq \sup_{\xx, \yy \in \C} \|\xx - \yy\|_\infty^2$. For any $\uu \in \C$, after $T$ iterations of Algorithm \ref{alg:flare}, we get
	\begin{align*}
	\sum_{k=1}^T \Big\{\left(\eta^{2}_{k-1} L_{k-1} - \eta_{k}^{2} L_{k} + \eta_{k} \right) F(\yy_{k}) - \eta_{k} F(\uu)  \Big\}+ \eta_{T}^{2} L_{T} F(\yy_{T+1}) \leq \frac{D}{2}\|\vs_T\|_1.
	\end{align*}
\end{theorem}
Note that for the proof of Theorem \ref{thm:mastertheorem_flare}, we make use of various parts of Lemma~\ref{lem:flare}.  
Finally, we state, without proof, the following corollary which is equivalent to Corollary~\ref{cor:maintheorem}.
\begin{corollary} 
	\label{cor:maintheorem_flare}
	Let $D \defeq \sup_{\xx, \yy \in \C} \|\xx - \yy\|_\infty^2$. For any $\uu \in \C$, after $T$ iterations of Algorithm \ref{alg:flare}, we get
	\begin{align*}
	F(\yy_{T+1}) - F(\uu) \leq  \frac{D\|\vs_T\|_1}{2 \sum_{k=1}^T \eta_{k}}.
	\end{align*}
\end{corollary}

Using these results, the proof of Theorem~\ref{thm:final_flare} follows as that of Theorem~\ref{thm:final_flag}. 
\fi

\end{document}